\documentclass[a4paper,10pt]{amsart} %revised
\usepackage{amsmath,amssymb,amscd,amsfonts,amsthm,amsrefs}
\usepackage[colorlinks,linkcolor=blue,anchorcolor=blue,citecolor=blue]{hyperref}
\usepackage{multirow}
\usepackage{multicol}
\numberwithin{equation}{section}
\newcommand{\p}{\partial}
\newcommand{\vphi}{\varphi}

\newcommand{\thmref}[1]{Theorem~\ref{#1}}
\newcommand{\secref}[1]{Section~\ref{#1}}
\newcommand{\lemref}[1]{Lemma~\ref{#1}}
\newcommand{\corref}[1]{Corollary~\ref{#1}}
\newcommand{\defnref}[1]{Definition~\ref{#1}}  %added
\newcommand{\propref}[1]{Proposition~\ref{#1}}  %added
\def\La{\Lambda} %revised

\newcommand{\Om}{\Omega}
\newcommand{\Na}{\nabla}

\def\p{\partial}
\def\b{\beta}

\def\vol{{\rm vol}} %revised
\def\Vol{{\rm Vol}} %added

\def\Hess{\mathop{\rm Hess}\nolimits}
\def\tr{\mathop{\rm tr}\nolimits}

\def\i{{\rm i}} %added
\def\j{{\rm j}} %added
\def\a{\alpha}
\DeclareMathOperator{\Tr}{Tr}

\newtheorem{theorem}{Theorem}[section]
\newtheorem{thm}[theorem]{Theorem}
\newtheorem{rem}[theorem]{Remark}
\newtheorem{defn}[theorem]{Definition}
\newtheorem{lemma}[theorem]{Lemma}
\newtheorem{lem}[theorem]{Lemma}

\newtheorem{prop}[theorem]{Proposition}

\newtheorem{cor}[theorem]{Corollary}
\newtheorem{ques}[theorem]{Question}

\allowdisplaybreaks[4]  %added

\title{The space of closed $G_2$-structures. I. Connections}

\author{Pengfei Xu}
  \address{Tongji University, Shanghai 200092, P.R. CHINA}
  \email{1830948@tongji.edu.cn}
\author{Kai Zheng}
  \address{University of Chinese Academy of Sciences, Beijing 100049, P.R.China; Tongji University, Shanghai 200092, P.R. CHINA}
  \email{KaiZheng@amss.ac.cn}
\date{\today}
\begin{document}
\maketitle

\begin{abstract}
In this article, we develop foundational theory for geometries of the space of closed $G_2$-structures in a given cohomology class as an infinite-dimensional manifold. We introduce Sobolev-type metrics, construct their Levi-Civita connections, formulate geodesic equations, and analyse the variational structures of torsion free $G_2$-structures under these Sobolev-type metrics. 

 \end{abstract}

\tableofcontents

\section{Introduction}

Let $M$ be a compact, oriented 7-manifold and $c\in H^3(M)$ be a de-Rham cohomology class.
We assume $\mathcal M$ is the non-empty space of all \textit{positive, closed} 3-forms $\vphi$ in a fixed de-Rham cohomology class $c$.
Aiming to understand the question on the global structure of the moduli space of torsion free $G_2$-structures suggested by Joyce \cite[Page 254]{MR1787733} and the uniqueness problem for torsion free $G_2$-structures in a connected component of $\mathcal M$, proposed in Donaldson \cites{MR3702382,MR3932259,MR3966735}, we explore geometry of $\mathcal M$ in this article and sequels, including introducing various metrics on it, comparing their geometries and discussing their further applications to these questions for torsion free $G_2$-structures. 
In this article, we focus on Levi-Civita connections and geodesic equations, inspired from the seminal works in K\"ahler geometry \cite{MR1736211,MR1863016,MR1969662,MR1165352,MR909015,Calabi}.

%%%%%%%%%%%%%%%%%%%%%%%%%%%%%%%%%%%%%%%%%%%%%%%%%%%%%%%%%%%%%%%%%%%%%%%%%%%%%%%%%%%%%%%%%%%%%%%%%%%%%%%%%%%%%%%%%%%%%%%%
The structure of this article is organised as following.
\secref{sec2} provides an overview of some basic knowledge of $G_2$ geometry. 
We recall basic contraction identities for $G_2$-structures in \secref{sec:2.1} and the $G_2$-module decomposition in \secref{sec:2.3}. In order to fix notations, we define and compare two inner products, one for tensors and the other one for forms. 

In \secref{sec:2.4}, we start with reminiscing with the operators $\i$ and $\j$ concerning $3$-forms introduced by Bryant \cite{BR05}*{Section 2.6}. In order to carry out the computation of variation of Hodge star operator and many other operators in \secref{sec:3.1}, we need to further 
%write the $\j$ operator in local coordinates %(Proposition \ref{prop:local_ex_j}) and more %importantly 
generalise the $\i$ and $\j$ operators to any $p$-forms in \defnref{def:operator_i general} and \defnref{def:operator_j general}. We also prove several identities involving the generalised $\i$ and $\j$ operators, which will be used frequently throughout this article.

In Section \ref{Torsion and curvatures}, we first recollect the torsion form and the curvature tensor for closed $G_2$-structures. Then, we will estimate the minimal eigenvalue of the torsion for closed $G_2$-structures in \thmref{thm:estimate_tau*tau}. As a result, we observe that the Bakry-Emery Ricci curvature for $G_2$-Laplacian solitons is bounded in terms of the scalar curvature in Proposition \ref{BER curvature}. It is related to a conjecture of the $G_2$-Laplacian flow, stating that it could be extended under the scalar curvature bound. See Lotay \cite{MR4295856}*{Section 4.4}.

In \secref{sec3}, we focus on various geometric quantities on the tangent space of $\mathcal M$. 
We calculate variations of various operators for $G_2$-structures, such as the volume element, the Hodge star, the codifferential and the Laplacian. They will be used essentially in the construction of compatible, and also Levi-Civita, connections for various metrics on $\mathcal M$. Similar calculation has been appeared in many pioneering works including \cite{BR05,MR1787733,Kar09,MR3613456,arXiv:1101.2004}. However, in our construction of connections, explicit formulas of these operators are required, which are complicated but necessary. We include those formulas and computation in \secref{sec:3.1}, which is the technical part of this article. 

In \secref{sec:3.2}, we characterise the tangent space $T\mathcal M$.
At any point $\vphi\in \mathcal M$, $T_\vphi\mathcal M$ consists of all d-exact $3$-forms, denoted by $X$. We apply the Hodge decomposition to define a $d$-exact 3-form $u$ satisfying
\begin{equation*}
	\Delta_\vphi u=X.
\end{equation*}
Equivalently, $u$ is also written as $G_\vphi X$, where $G_\vphi$ is the Green operator. We call $u$ \textit{the potential form} of $X$ in Definition \ref{potential form}.

In \secref{sec:3.3}, we define general notions of metric and connection on $T\mathcal M$. The compatibility between the abstract metric and connection yields a criterion in \thmref{thm:comp_General}, which leads to further discussions on compatible connections in \secref{section:Dirichletmetric} and \secref{L2metric}.

A Dirichlet type metric in the infinite-dimensional space $\mathcal M$ was introduced by Bryant-Xu in \cite[Section 1.5]{arXiv:1101.2004} as
$
	{\mathcal G}^D(X,Y)\doteq\int_M g_\vphi(GX, Y)  \vol_\vphi, \forall X,Y\in T_\vphi\mathcal M.
$
With the help of our potential forms $u=G_\vphi X,\ v=G_\vphi Y$, we rewrite this metric as the form
\begin{align*}
	{\mathcal G}^D(X,Y)\doteq\int_M g_\vphi(\delta u,\delta v)  \vol_\vphi.
\end{align*}
%They showed that the $G_2$ Laplacian flow is the gradient flow of the volume functional under this metric.
In \secref{section:Dirichletmetric}, we call it \textit{the Dirichlet metric}, comparing with the Dirichlet metric in K\"ahler geometry, c.f. \cites{MR2915541,MR3576284,MR3412344}. 

We will use the Dirichlet metric of this particular form to construct Levi-Civita connections. However, we face a difficulty that it is not clear how to get the Levi-Civita connections from the expression of $\p_t {\mathcal G}^D(\cdot,\cdot)$ directly, which is very different from their counterparts in K\"ahler geometry. Alternatively, in order to overcome this problem, we first construct metric connections in Section \ref{sec:4.1}, which might have non-vanishing torsion tensor, but serves as grounds for existence of Levi-Civita connections. Then we refine these metric connections to further construct the Levi-Civita connections, with the help of the contorsion tensor, which is motivated from the classic construction in finite-dimensional Riemannian geometry. 

In Definition \ref{defn:con} in Section \ref{sec:4.1}, three metric connections for the Dirichlet metric are defined, i.e. they are compatible with the Dirichlet metric (Proposition \ref{compatible2}). Moreover, their linear combinations are also metric connections (Proposition \ref{compatible combination}).

Having all these ready, in \secref{sec:4.2}, we construct the \textit{Levi-Civita connection} $D^D$ for the Dirichlet metric $\mathcal G^D$.
Precisely speaking, we let $\vphi(s,t)$ be a family of $G_2$-structures in $\mathcal M$ and denote the directional derivatives by $X=\vphi_t=\Delta_\vphi u, \ Y=\vphi_s=\Delta_\vphi v$. 
\begin{thm}\label{thm1}
	 The Levi-Civita connection for the Dirichlet metric $\mathcal G^D$ is $D^D_tY=Y_t+P^D(\vphi, \vphi_t, Y)$, where $P^D$ is given as
	\begin{align*}
		P^D(\vphi,\vphi_t,Y)=&d\bigg\{\frac{1}{2}\bigg[g(\vphi,\vphi_t)\delta v+g(\vphi,Y)\delta u\bigg]\\
		&-\frac{1}{2}\delta\bigg[g(\delta u,\delta v)\vphi-\i_\vphi \j_{\delta u}\delta v\bigg]
		-\frac{1}{4}\bigg[\i_{\delta v}\j_\vphi \vphi_t +\i_{\delta u}\j_\vphi Y\bigg]
		\bigg\}.
	\end{align*} 
\end{thm}
The precise statements and proofs are given in \defnref{defn:DDcon}, \propref{prop:DDexplicit} and \propref{prop:DD-Levicivita}. 
We can get the expression of \textit{the geodesic equation} immediately from the Levi-Civita connection $D^D$, see also \propref{prop:DDgeo}. We present it here for future geometric applications.
\begin{thm}
	Let $\vphi(t)$ be a geodesic equation under the Levi-Civita connection $D^D$. Then the geodesic equation satisfies a system	
	\begin{equation*}
\left\{
\begin{aligned}
\vphi_t&=\Delta_\vphi u,\\
		\vphi_{tt}=&d\bigg\{
		-g(\vphi_t,\vphi)\delta u
		+\frac{1}{2}\delta\bigg[g(\delta u, \delta u)\vphi-\i_\vphi \j_{\delta u}(\delta u)\bigg]+\frac{1}{2}\i_{\delta u}\j_\vphi \vphi_t\bigg\}.	
\end{aligned}
\right.
\end{equation*} 
\end{thm}

\begin{rem}If we set $\a=\delta u$, then the geodesic equations become $\vphi_t=d\a$ and 
	\begin{equation*}
		d\bigg\{\a_{t}+g(d\a,\vphi)\a
		-\frac{1}{2}\delta\bigg[g(\a, \a)\vphi-\i_\vphi \j_{\a}(\a)\bigg]
		-\frac{1}{2}\i_{\a}\j_\vphi (d\a)\bigg\}=0.	
\end{equation*} Note that $\delta \a=0$. Equivalently, the vector field, which is dual to the differential form $\a$, is divergence free.
We may compare the geodesic equation of $\a$ with the Euler equation in fluid dynamics, which is treated as geodesic in the space of diffeomorphisms \cite{MR271984}.
\end{rem}

In \secref{L2metric}, we introduce different Sobolev-type metrics and the variational structures for the torsion free $G_2$-structures. We summarise various energy functionals and their associated gradient flows. We also propose many questions related to the existence and uniqueness problems of the torsion free $G_2$-structures for future studies.

In Definition \ref{defn:L2}, we introduce another two metrics, the Laplacian metric and the $L^2$ metric
\begin{align*}
\mathcal G^L_{\vphi}(X,Y)=\int_M g_{\vphi}(X,Y) \vol_{\vphi} ,
\quad \mathcal G^M_\vphi(X,Y)=\int_M g_{\vphi}(u,v) \vol_{\vphi}.
 \end{align*}

 We also construct their Levi-Civita connections. 
\begin{thm}
Both the Laplacian metric $\mathcal G^L$ and the $L^2$ metric $\mathcal G^M$ admit a Levi-Civita connection $D^L_t$ and $D^M_t$, respectively.
\end{thm} 
Their precise expressions are provided in \defnref{defn:L_con} and their proofs are given in \thmref{prop:L_con}.
Their constructions are more involved, but in the same spirit as the process we obtain the Levi-Civita connection of the Dirichlet metric. We collect their proofs in \secref{sec:proof1}.
 
  Under different metrics, the first variation of the volume function gives us different Euler-Lagrange equations. Then, we define several energy functionals to be $L^2$-norm of these critical equations (Definitions \ref{E energies}). We derive their first order variations and their decreasing flows in \thmref{prop:Var_Dirichlet energy}. We put them all together in a table in Definition \ref{decreasing flows}. The detailed proof of 	\thmref{prop:Var_Dirichlet energy} is given in \secref{sec:pf2}.

\secref{sec:5.3} is devoted to the Dirichlet energy and the Dirichlet flow restricted in a fixed cohomology class, as shown in the table in Definition \ref{decreasing flows}. The Dirichlet energy is the norm of $d\tau$ under the Dirichlet metric $\mathcal G^D$, where $\tau$ the torsion of closed $G_2$-structures, and the Dirichlet flow \eqref{eq:Drichletflow} is the gradient flow of the Dirichlet energy under the Laplacian metric $\mathcal G^L$. The Dirichlet energy is concave at torsion free structures in \propref{Dirichlet energy convex}. At last, we decompose the Dirichlet flow into a PDE system, following the same idea as the study of the Calabi flow \cite{MR3010550}.

We remark that for general $G_2$-structures without closedness assumption, the short time existence, stability and monotonicity have been proved in Wei\ss-Witt \cites{MR2980500,MR2995206}. Meanwhile, when restricted in the $G_2$-structures, which generate the same Riemannian metric, the Dirichlet flow is studied in \cite{MR4203649,MR4215279,MR4020314}.

The $G_2$-structure induces a Riemannian metric $g_\vphi$ naturally. The Laplacian metric could be extended to the space of $G_2$-structures, not necessary to be restricted in the given class $c$.
In \thmref{EbinLaplacian} in Section \ref{sec:5.5}, we derive a relation between the Laplacian metric and the Ebin metric, which is defined on the space of Riemannian metrics.

%%%%%%%%%%%%%%%%%%%%%%%%%%%%%%%%%%%%%%%%%%%%%%%%%%%%%%%%%%%%%%%%%%%%%%%%%%%%%%%%%%%%%%%%%%%%%%%%%%%%%%%%%%%%%%%%%%%%
\medskip

\noindent \textbf{Acknowledgements.} 
The authors are partially supported by NSFC grant No. 12171365 and the Fundamental Research Funds for the Central Universities.

\section{Review of \texorpdfstring{$G_2$}{G2}-structures }\label{sec2}
The $G_2$ geometry is widely used in physics and the torsion free $G_2$-structures have attracted great attention, see \cites{BR87,MR3881202,MR4324180,J96B,MR2481746,MR1745014,MR4417724,MR4295857} and references therein.

In this section, we first reminisce about basic definitions and some known properties of $G_2$-structures, c.f. \cites{BR05,MR3932259,MR1871001,MR1863733,MR3959094,MR3838115,MR3966735,MR1787733} and references therein. We will then prove some useful identities for further applications.

Let  $G_2$ be the exceptional simple Lie group of dimension 14.
We set $\{e^i;1\leq i\leq 7\}$ to be the standard basis of $\mathbb R^7$. The group $G_2$ is the subgroup of the general linear group $GL(7,\mathbb R)$, which keeps the following 3-form $\vphi_0$ fixed,
\begin{align*}
\vphi_0 &=e^{123}+e^{145}+e^{167}+e^{246}+e^{275}-e^{347}-e^{356}, \quad e^{ijk}=e^i\wedge e^j\wedge e^k.
\end{align*}
We then identify an oriented real vector space $V$ of dimension 7 with $\mathbb R^7$ by an isomorphism between them.
A 3-form $\vphi$ in $\La^3(V^\ast)$ is said to be \textit{positive}, if the quadratic form $Q_\vphi(u,v)$ defined by
\begin{align}\label{eq:G2metric}
V \otimes  V \rightarrow \La^7(V^\ast),\quad
(u,v)\mapsto (u\lrcorner \vphi)\wedge (v\lrcorner \vphi)\wedge\vphi
\end{align}
 is positive definite.
 % Thus we can view $Q_{\vphi}$ as tensor in $Sym(T^*M \otimes T^*M) \otimes \Lambda^7(M)$
 The set of all positive 3-forms, denoted by $\mathcal P$, forms a single open $GL(V)$-orbit in $\La^3(V^\ast)$ and the stabilizer of a positive 3-form in $GL(V)$ is also isomorphic to $G_2$.
These notions are further extended to the manifold $M$ in a natural way, via viewing the tangent space $T_pM$ as a real vector space $V$ by diffeomorphisms at each point $p\in M$. A positive 3-form on $M$ is also called a \textit{$G_2$ structure}.

% We let $c$ be a cohomology class in $H^3(M,\mathbb R)$. Let $\vphi$ be a positive 3-form in $c$. ???
 Due to the positivity shown by $Q_\vphi$ above, a $G_2$ structure $\vphi$ on the manifold $M$ defines a unique Riemannian metric $g_\vphi$. Together with the Hodge star operator $\ast_\vphi$, they are formulated by the following identity
\begin{align}\label{eq:G2metric1}
6 g_\vphi= \ast_\vphi Q_\vphi.
\end{align}
 We let $\psi$ be the associated dual 4-form
\begin{align*}
\psi\doteq \ast_\vphi\vphi.
\end{align*}
Usually, we denote a manifold with such a $G_2$ structure as $(M, \vphi, g_\vphi)$.
In Fern\'andez-Gray \cite{MR696037}, a positive 3-form is \textit{torsion free}, if $\vphi$ is both closed and co-closed, that is
\begin{align}\label{torsion free equation}
d\vphi=d\psi=0.
\end{align}
Equivalently, the holonomy group of $g$ is contained in $G_2$.
The manifold $M$ equipped with such torsion free form $\vphi$ is called a \textit{$G_2$ manifold}.
%Meanwhile, we also call the associated Riemannian metric $g_\vphi$ a \textit{$G_2$-metric}.

The co-differential operator acting on $k$-forms is defined to be
\begin{align*}
\delta_\vphi\doteq (-1)^{k}\ast_{\vphi} d \ast_{\vphi}.
\end{align*} Consequently, acting on 3-forms, the Hodge Laplacian operator $\Delta_\vphi$ is given by
\begin{align}\label{Hodge Laplacian}
\Delta_\vphi\doteq -d\ast_{\vphi} d\ast_{\vphi}+\ast_{\vphi} d\ast_{\vphi} d
\end{align}
and the torsion free equation \eqref{torsion free equation} is also equivalent to
\begin{align}\label{Hodge Laplacian torsion free}
d\vphi=\delta_\vphi \vphi=0\text{ or } \Delta_\vphi\vphi=0.
\end{align}

In order to find a torsion free $G_2$ structure, it is sufficient to restrict ourself in the space of all \textit{positive, closed} 3-forms $\vphi$ in a fixed de-Rham cohomology class $c\in H^3(M)$. We denote this space by $\mathcal M_{c}$. In this article, we plan to study the geometry of $\mathcal M_{c}$. We assume $\mathcal M_{c}$ is non-empty and write it as $\mathcal M$ for short.

Hitchin \cite{MR1871001} introduced a variational structure for $G_2$ manifolds, that is the volume of positive 3-form $\vphi$,
\begin{align*}
\Vol(\vphi)=\int_M\vol_\vphi=\frac{1}{7}\int_M\vphi\wedge \psi
\end{align*} where $\vol_\vphi\doteq\ast_\vphi1$ is the \textit{volume element} of $g_\vphi$.
The first variation of the volume is given by
\begin{align}\label{gradient vol}
D_\vphi \Vol(X)=\frac{1}{3}\int_MX \wedge \psi,\quad X=\frac{\p \vphi(t)}{\p t}\vert_{t=0}\in \Om^3.
\end{align}Restricted in a fixed cohomology $c$, $X$ is $d$-exact and the variation vanishes for all $X$ if and only if the critical point $\vphi$ is torsion free.

Hereafter, we write the metric $g_\vphi$ and the volume form $\vol_\vphi$ as $g$ and $\vol$ in short. Also, we denote the operators $*_\vphi, \ \delta_\vphi, \ \Delta_\vphi$ by $*, \ \delta,\ \Delta$ when there is no confusing. 
%%%%%%%%%%%%%%%%%%%%%%%%%%%%%%%%%%%%%%%%%%%%%%%%%%%%%%%%%%%%%%%%%%%%%%%%%%%%%%%%%%%%%%%%%%%%%%%%%%%%%%%%%%%%%%%%%%%%%%%%%%%%%%%%%%%%%%%%%%%%%%%%%%%%%%%%%%%%%%
\subsection{Contraction identities and inner products}\label{sec:2.1}
Let $\{ x_i \}_{1\leq i \leq 7}$ be a local coordinate on $M$. We write $dx^{i_1}\wedge \cdots \wedge dx^{i_k}$ as $dx^{i_1\cdots i_k}$ for short. 
A $k$-form $\omega$ is written as
\begin{align*}
	\omega=\frac{1}{k!} \omega_{i_1\cdots i_k} dx^{i_1\cdots i_k}
\end{align*}
with the coefficients $\omega_{i_1\cdots i_k}$ skew-symmetric regarding to the indices $i_1,\cdots ,i_k$. For instance, we write a $G_2$ structure $\vphi$ and its dual form $\psi=*_\vphi \vphi$ locally as 
$$ \vphi=\frac{1}{3!}\vphi_{ijk}dx^{ijk}, \ \ \psi=\frac{1}{4!}\psi_{ijkl}dx^{ijkl}.$$
 
\begin{lemma}\label{lem:contraction}
	The following contraction identities hold true for any $G_2$ structure $\vphi$ and its dual $4$-form $\psi$.
\begin{equation*} 
	\begin{aligned}
		& \varphi_{ijk}\varphi_{abn} g^{ia}g^{jb}=6g_{kn}; \quad
		\psi_{ijkl}\psi_{abmn} g^{ia}g^{jb}g^{km}=24g_{ln}; \\
		&  \varphi_{ijk}\varphi_{abl} g^{ia}=g_{jb}g_{kl}-g_{jl}g_{kb}+\psi_{jkbl}; \quad
		\varphi_{ijk}\psi_{abmn} g^{ia}g^{jb}=4\vphi_{kmn}; \\
		&\varphi_{ijk}\psi_{abmn} g^{ia}=g_{jb}\vphi_{kmn}-g_{jm}\vphi_{kbn}+g_{jn}\vphi_{kbm}-g_{kn}\vphi_{jmn} \\
		&\ \ \ \ \ \ \ \ \ \ \ \ \ \ \ \ \ \ \ +g_{km}\vphi_{jbn}-g_{kb}\vphi_{jmn};\\
		&\psi_{ijkl}\psi_{abmn} g^{ia}g^{jb}= 2\psi_{klmn}+4(g_{km}g_{ln}-g_{kn}g_{lm}).
	\end{aligned}
\end{equation*}
\end{lemma}
\begin{proof}
	The proof can be found in \cite{BR05}*{Section 2.4} and \cite{Kar09}*{Lemma A.8-Lemma A.10} and we omit it here.
\end{proof}

Now we clarify inner products and fix notations and conventions.

\begin{defn}\label{defn:innerproduct}
The inner product of any two $k$-forms $\omega, \overline{\omega}$ with respect to the metric $g=g_{ij}dx^i \otimes dx^j$ is defined as
$$g(\omega, \overline{\omega})\doteq \frac{1}{k!}\omega_{i_1\cdots i_k}\overline{\omega}_{j_1\cdots j_k}g^{i_1j_i}\cdots g^{i_kj_k}.$$
\end{defn}
As a result, we have 
$$\omega\wedge *\overline{\omega}= \overline{\omega}\wedge*\omega=g(\omega, \overline{\omega})\vol_g,\ \ g(\vphi,\vphi)=g(\psi,\psi)=7.$$

%Assume that $\overline{\omega}=A_{i_i\cdots i_k}dx^{i_1\cdots i_k}$ is another $k$-form, whose coefficients $A_{i_i\cdots i_k}$ are possibly not skew-symmetric. We still have
%     \begin{align*}
%g(\omega, \overline{\omega})=A_{j_1\cdots j_k}\omega_{i_1\cdots i_k} g^{i_1j_1}\cdots g^{i_kj_k}.
% \end{align*}The proof goes as the following.
%We denote the permutation group of $\{1,2,\cdots k\}$ by $S_k$ and rewrite $\overline{\omega}$ as 
% $$ \overline{\omega}=\left[ \frac{1}{k!}\sum_{\sigma \in S_k } (-1)^{|\sigma|}A_{\sigma(i_1)\cdots \sigma(i_k)}\right] dx^{i_1\cdots i_k} .$$
%By definition 
%$$g(\omega, \overline{\omega})=\frac{1}{k!}\omega_{i_1\cdots i_k}\sum_{\sigma \in S_k } (-1)^{|\sigma|}A_{\sigma(j_1)\cdots \sigma(j_k)} g^{i_1j_1}\cdots g^{i_kj_k},$$
% which is equal to
%     \begin{align*}
%&\frac{1}{k!}A_{j_1\cdots j_k}\sum_{\sigma \in S_k } (-1)^{|\sigma|}\omega_{\sigma(i_1)\cdots \sigma(i_k)} g^{i_1j_1}\cdots g^{i_kj_k} \\
%=&\frac{1}{k!}A_{j_1\cdots j_k}(k!\omega_{i_1\cdots i_k}) g^{i_1j_1}\cdots g^{i_kj_k} 
%=A_{j_1\cdots j_k}\omega_{i_1\cdots i_k} g^{i_1j_1}\cdots g^{i_kj_k}.
% \end{align*}
%
% 
\begin{defn}
For any $(0,k)$-tensors $S=S_{i_1\cdots i_k}dx^{i_1}\otimes \cdots \otimes dx^{i_k}$ and $ \overline{S}=\overline{S}_{i_1\cdots i_k}dx^{i_1}\otimes \cdots \otimes dx^{i_k}$, we define the inner product for tensors to be
 $$ g(S,\overline{S})\doteq \frac{1}{k!}S_{i_1\cdots i_k}\overline{S}_{j_1\cdots j_k}g^{i_1j_i}\cdots g^{i_kj_k}.$$
 \end{defn}
 
We see that these two definitions of inner products, one for forms and one for tensors, coincide when $S,\bar{S}$ are $(0,k)$-skew-symmetric tensors. We now clarify this. If $S, \overline{S}$ are skew-symmetric tensors, then $S_{i_1\cdots i_k}$ is skew-symmetric. We associate the tensor $S$ with a form
  $$ \omega_S=\frac{1}{k!}S_{i_1\cdots i_k}\sum_{\sigma\in S_k}[(-1)^{|\sigma|}dx^{\sigma(i_1)}\otimes \cdots \otimes  dx^{\sigma(i_k)}]=\frac{1}{k!}S_{i_1\cdots i_k}dx^{i_1\cdots i_k}.$$
  Similarly we have $\omega_{\overline{S}}=\frac{1}{k!}\overline{S}_{i_1\cdots i_k}dx^{i_1\cdots i_k}.$
 Accordingly, we see that, the inner product between tensors is now exactly the inner product for their associated forms. Here the definition for the inner product of tensors is $k!$ times of that in \cite{MR3613456}.
%%%%%%%%%%%%%%%%%of %%%%%%%%%%%%%%%%%%%%%%%%%%%%%%%%%%%%%%%%%%%%%%%%%%%%%%%%%%%%%%%%%%%%%%%%%%%%%%%%%%%%%%%%%%%%%%%%%%%%%%%%%%%%%%%%%%%%%%%%%%%%%%%%%%%%%%%%%%%%%
\subsection{$G_2$-module decompositions}\label{sec:2.3}
Let $\Omega^k$ be the set of all smooth $k$-forms on $M$. As shown in \cite{BR87}*{Page 541-542}, the $G_2$-module decompositions implies
\begin{align}\label{Hodge decomposition}
\Om^2=\Om^2_7\oplus\Om^2_{14},\quad\Om^3=\Om^3_1\oplus\Om^3_7\oplus\Om^3_{27}.
\end{align}
The subspaces are characterised as follows
\begin{equation}
\left\{
\begin{aligned}
   %\Om^2_7&=\left\{ x\lrcorner \vphi \mid x\in \Gamma(TM) \right\} = \left\{ * (\a \wedge \psi) \mid \a\in \Omega^1 \right\} \label{eq:de1} \\
    \Om^2_7& =\left\{ u\lrcorner \vphi \mid u\in \Gamma(TM) \right\} = \left\{ * (\a \wedge \psi) \mid \a\in \Omega^1 \right\}   \\
    &=\left\{ \beta \in \Lambda^2  \mid *(\vphi \wedge \beta )=2\beta \right\} \label{eq:de1}  ;\\
   \Om^2_{14}&=\left\{ \beta \in \Om^2 \mid *(\psi \wedge \beta )=0 \right\} =\left\{ \beta \in \Om^2 \mid *(\vphi \wedge \beta )=-\beta \right\};
\end{aligned}
\right.
\end{equation}
and
\begin{equation}\label{eq:de2}
\left\{
\begin{aligned}
   \Om^3_1&=\left\{ f\vphi \mid f\in C^{\infty}(M) \right \};\\
  % \Om^3_7&=\left\{ x\lrcorner \psi \mid x\in \Gamma(TM) \right\} = \left\{ *(\a \wedge \vphi) \mid \a\in    \Om^1 \right\} \\
    %  &=\left\{ \eta \in    \Om^3 \mid *(\vphi \wedge *(\vphi \wedge \eta)) =4\eta \right\} ; \label{eq:de2}\\
       \Om^3_7 &= \left\{ u\lrcorner \psi \mid u\in \Gamma(TM) \right\}=\left\{ *(\a \wedge \vphi) \mid \a\in    \Om^1 \right\}\\ 
       &=\left\{ \eta \in    \Om^3 \mid *(\vphi \wedge *(\vphi \wedge \eta)) =4\eta \right\} ; \\
   \Om^3_{27}&=\left\{ \eta \in    \Om^3 \mid \eta \wedge \vphi=0, \eta \wedge \psi=0 \right\}.
\end{aligned}
\right.
\end{equation}
The decompositions
$$\Om^5=\Om^5_7\oplus\Om^5_{14},\quad\Om^4=\Om^4_1\oplus\Om^4_7\oplus\Om^4_{27}$$
are given by the Hodge duality. We will use $\pi^p_q$ to denote the corresponding projection of $\Omega^p$ to the component $\Omega^p_q$. The following identities clarify the projection maps.

 By the decomposition \eqref{eq:de2}, for any $3$-form $\eta$, there is a smooth function $f_0$, an $1$-form $f_1$ and a $3$-form $f_3\in \Omega_{27}^3$ such that
 \begin{align}\label{eq:decompo_3form}
 	\eta=\pi_1^3\eta+\pi_7^3\eta+\pi_{27}^3\eta=3f_0 \vphi+*(f_1\wedge \vphi)+f_3.
 	\end{align}
 \begin{lemma}\label{lem:inner_vphi27part}
 	For any $3$-form $\eta$, we have
 	 $$ g(\pi_7^3\eta,\vphi)=g(\pi_{27}^3\eta,\vphi)=0.$$
 \end{lemma}
\begin{proof}
	First, by \eqref{eq:decompo_3form} we can find an $1$-form $f_1$ such that $\pi_7^3\eta=*(f_1\wedge\vphi)$. Since $\vphi$ is a $3$-form, we compute that 
	$$g(*(f_1\wedge\vphi),\vphi))\vol=f_1\wedge \vphi\wedge \vphi=0.$$
	Hence we see $ g(\pi_7^3\eta,\vphi)=0.$
	Besides, we see from \eqref{eq:de2} that $\pi_{27}^3\eta \wedge\psi=0$, which means 
	$$0=\pi_{27}^3\eta\wedge *\vphi=g(\pi_{27}^3\eta,\vphi)\vol.$$
	We then complete the proof.
\end{proof}

\begin{prop}\label{prop:decomI}
For any $2$-form $\a$ and $3$-form $\eta$ on $M$, it holds:
\begin{equation*} \label{prop:propjectionmap}
\left.
\begin{aligned}
  &\  \pi_7^2 \a= \frac{\a+*(\a \wedge \vphi)}{3}, &\ \pi_{14}^2 \a= \frac{2\a-*(\a \wedge \vphi)}{3};\\
    &\ \pi_1^3 \eta= \frac{g(\vphi,\eta)}{7}\vphi, &\  \pi_7^3 \eta= \frac{*[\vphi \wedge *(\vphi \wedge \eta)]}{4}.
\end{aligned}
\right.
\end{equation*}
\end{prop}
\begin{proof}
	The proof of the first two identities are given in \cite[Page 492]{Kar09}(we choose the opposite orientation so that the $*$ operator here differs to that in \cite{Kar09} by a minus). We now give a proof for the later two identities. By the $G_2$ module decomposition, $\eta=\pi_1^3 \eta+\pi_7^3\eta+\pi_{27}^3\eta$ with $\pi_1^3\eta=3f_0\vphi$ for some $f_0$. According to \lemref{lem:inner_vphi27part}, we see $g(\vphi, \pi_7^3\eta)=0,\ g(\vphi, \pi_{27}^3\eta)=0$, thus
	$$g(\vphi, \eta)=g(\vphi, \pi_1^3\eta+\pi_7^3\eta+\pi_{27}^3\eta)=g(\vphi, 3f_0\vphi)=21f_0,$$
from which we conclude the third identity.

For the fourth identity, note that
\begin{align*}
\pi_1^3\eta \wedge \vphi=3f_0\vphi\wedge\vphi=0
	\end{align*}
and $\pi_{27}^3\eta \wedge \vphi=0$ by $G_2$ module decomposition \eqref{eq:de2}. 
Hence we get 
\begin{align*}
		\eta\wedge \vphi=( \pi_1^3\eta+\pi_7^3\eta+\pi_{27}^3\eta)\wedge \vphi=\pi_7^3\eta\wedge\vphi.
\end{align*}
Again by \eqref{eq:de2}, $*[\vphi \wedge *(\vphi\wedge \pi_7^3\eta) ]=4\pi_7^3\eta.$
Then we have
 $$\pi_7^3\eta=\frac{*[\vphi \wedge *(\vphi \wedge \pi_7^3\eta) ]}{4}=\frac{*[\vphi \wedge *(\vphi \wedge \eta)]}{4}.$$
\end{proof}

%We will also use $d^2_7:\Om^1\rightarrow\Om_7^2$, $d^2_{24}:\Om^1\rightarrow\Om_7^{24}$ and $d_1,d_7,d_{27}$ to denote the corresponding projection of the $d$ operator.

%%%%%%%%%%%%%%%%%%%%%%%%%%%%%%%%%%%%%%%%%%%%%%%%%%%%%%%%%%%%%%%%%%%%%%%%%%%%%%%%%%%%%%%%%%%%%%%%%%%%%%%%%%%%%%%%%%%
\subsection{The \texorpdfstring{$\i$}{i}-operator and \texorpdfstring{$\j$}{j}-operator}\label{sec:2.4}
Bryant \cite{BR05} introduced two operators $\i_\vphi$ and $\j_\vphi$, which will be frequently used. 
\begin{defn} \label{def:operator_i_j}
The operator $\i_\vphi:S^2(T^*M)\rightarrow \Omega^3(M)$ is defined as
$$\i_\vphi (h) (u,v,w)=\vphi((u\lrcorner h)^\sharp,v,w)+\vphi(u,(v\lrcorner h)^\sharp,w)+\vphi(u,v,(w\lrcorner h)^\sharp),$$
for any symmetric $2$-tensor $h=h_{ij}dx^i \otimes dx^j$ and any vector fields $\ u,\ v,\ w\in \Gamma(TM).$

Here $u\lrcorner h$ refers to the contraction of the vector $u$ with the tensor $h$ by the first factor, and $u^\sharp$ is the dual vector of $1$-form $u$ with respect to the $G_2$ metric $g$.
In the local coordinates, it says
$$(\i_\vphi (h))_{ijk}=h_{il}g^{lm}\vphi_{mjk}+h_{jl}g^{lm}\vphi_{imk}+h_{kl}g^{lm}\vphi_{ijm}.$$
\end{defn}
\begin{defn}
The operator $\j_\vphi:\Omega^3(M)\rightarrow S^2(T^*M) $ is defined as
\begin{align*}
\j_\vphi(\eta)(u,v)=*(u\lrcorner \vphi\wedge v\lrcorner \vphi\wedge \eta) 
\end{align*}
for any 3-form $\eta\in \Om^3$ and $u,v\in \Gamma(TM)$.
The local expression of $\j_\vphi$ will be studied in \propref{prop:local_ex_j}.
\end{defn}
It is direct to see 
\begin{align}\label{eq:i(fh)}
	\i_\vphi (fh)=f\i_\vphi h,\ \  \j_\vphi(f\eta)=f\j_\vphi\eta
	\end{align}
hold for any function $f$ on $M$ since $\i$ and $\j$ are defined point-wise.
\begin{rem}
Note that the operator $\i$ defined above is $1/2$ times of the one defined in \cite{BR05}*{equation (2.15) and (2.17) }.
\end{rem}
The following properties can be found in \cite[Section 2.6, Page 79]{BR05}.
\begin{prop} \label{prop:operator_i and_B}
The operator $\i_\vphi$ and $\j_\vphi$ satisfy:
\begin{enumerate}
\item $\i_\vphi (g)=3\vphi $, and $\j_{\vphi}(\vphi)=6g$;
\item $\i_\vphi (h) \in \Omega^3_{27}, \ \forall h\in S^2_0(T^*M)$, the space of all trace-free symmetric $(0,2)$-tensor;
\item $\j_{\vphi}(\eta)=\j(*(f_1\wedge\vphi))=0, \ \forall \eta=*(f_1\wedge\vphi)\in \Om_7^3$; 
\item  $\j_{\vphi}(\eta)\in S^2_0(T^*M) ,\ \forall \eta\in \Om_{27}^3$;
\item $\j_{\vphi}(\i_{\vphi}(h))=4h, \ \forall h \in S^2_0(T^*M) $;
\item $ \i_{\vphi}(\j_{\vphi}(\eta))=4\eta$,  $\ \forall \eta \in \Om^3_{27}$.
\end{enumerate}
\end{prop}

We further compute the local expression of the $\j$ operator, which is used to generalize the definition of $\j$ later in this section.
\begin{prop} \label{prop:local_ex_j}
For any 3-form $\eta=\frac{1}{6}\eta_{ijk}dx^{ijk}$, we have
\begin{equation*}	
\j_{\vphi}(\eta)_{im}=\frac{1}{2}\left(\eta_{ijk}{\vphi_m}^{jk}+\eta_{mjk}{\vphi_i}^{jk}\right).
 \end{equation*}
\end{prop}
\begin{proof}
	Due to the $G_2$-module decomposition \eqref{eq:decompo_3form}, we write
	\begin{align}\label{eq:decomposition_X}
		\eta=\pi_1^3\eta+\pi_7^3\eta+\pi_{27}^3\eta=3f_0 \vphi+\ast_\vphi (f_1\wedge \vphi)+f_3.
	\end{align}
In the proof of \propref{prop:decomI}, we have $$f_0=\frac{1}{21}g(\eta,\vphi).$$
	By (1), (3), (5) of \propref{prop:operator_i and_B}, we compute that
	\begin{equation}\label{eq:ij}
	\begin{aligned}
		\i_{\vphi}(\j_{\vphi}(\eta))=&	\i_{\vphi}(\j_{\vphi}(3f_0\vphi))+	\i_{\vphi}(\j_{\vphi}(*(f_1\wedge\vphi)))+	\i_{\vphi}(\j_{\vphi}(f_3))\\
		=&54f_0\vphi+4f_3.
	\end{aligned}
\end{equation}
	In the local coordinate, it is written as
	\begin{align*}
	 \frac{1}{6}\left(\j_{\vphi}(\eta)_{ia}g^{al}\vphi_{ljk}+\j_{\vphi}(\eta)_{ja}g^{al}\vphi_{ilk}+\j_{\vphi}(\eta)_{ka}g^{al}\vphi_{ijl} \right)  =9f_0\vphi_{ijk}+\frac{2}{3}(f_3)_{ijk}.
	\end{align*}
	Taking contraction with ${\vphi_m}^{jk}$ on the both sides and using the contraction properties of $\vphi$ given in \lemref{lem:contraction}, we have
	\begin{align}\label{eq:con11}
		\frac{2}{3}\j_{\vphi}(\eta)_{im}+\frac{1}{3}g^{jk}\j_{\vphi}(\eta)_{jk}g_{ im} =54f_0g_{im}+\frac{2}{3}(f_3)_{ijk}{\vphi_m}^{jk}.
	\end{align}
	We further take contraction with $g^{im}$. Using the fact that $g_{im}g^{im}=7$ and $g(\vphi,f_3)=0$ in \lemref{lem:inner_vphi27part}, we get
	\begin{align*}
		3g^{jk}\j_{\vphi}(\eta)_{jk}=378f_0+4g(\vphi, f_3)=378f_0.
	\end{align*}
Substituting the above back into \eqref{eq:con11}, we get
	\begin{align} \label{eq:op_B_1}
		\j_{\vphi}(\eta)_{im} =18f_0g_{im}+(f_3)_{ijk}{\vphi_m}^{jk}.
	\end{align}
	
	We now calculate the last part of the above formula. The local expression of \eqref{eq:decomposition_X} is
		$$\eta_{ijk}=3f_0 \vphi_{ijk}+\ast_\vphi (f_1\wedge \vphi)_{ijk}+(f_3)_{ijk}.$$
		 Making contraction with ${\vphi_m}^{\ jk}$ and again using the contraction properties in \lemref{lem:contraction}, we have
	\begin{align}\label{index_X}
		\eta_{ijk}{\vphi_m}^{jk}= 18f_0 g_{im}+(\ast (f_1\wedge \vphi))_{ijk}{\vphi_m}^{jk}+(f_3)_{ijk}{\vphi_m}^{jk}.
	\end{align}
	We rewrite $\ast_\vphi (f_1\wedge \vphi)$ as
	\begin{align*}
		(\ast (f_1\wedge \vphi))_{ijk}{\vphi_m}^{jk}=4(f_1)^l\psi_{lijk}{\vphi_m}^{jk}=-4(f_1)^l\vphi_{lim},
	\end{align*}
	which is skew-symmetric on index $(i,m)$. While, \eqref{eq:op_B_1} implies that $(f_3)_{ijk}{\vphi_m}^{jk}$ is symmetric on $(i,m)$. Hence by taking the symmetric part of the both sides of \eqref{index_X}, we get
	\begin{align*}
		\frac{1}{2}\left(\eta_{ijk}{\vphi_m}^{jk}+\eta_{mjk}{\vphi_i}^{jk}\right)= 18f_0 g_{im}+(f_3)_{ijk}{\vphi_m}^{jk}.
	\end{align*}
	Comparing this formula with \eqref{eq:op_B_1}, we thus complete the proof of this lemma.
\end{proof}
Now we give the inner product of the $\i$ operator. The next lemma can be found in \cite{Kar09}*{Proposition 2.9}, and we give a proof for the readers' convenience.
\begin{lem}\label{lem:noem_i}
	Let $h_1, \ h_2$ be two symmetric $2$-tensors. It holds that
	\begin{align*}
		g\left( \i_{\vphi}(h_1),\i_{\vphi}(h_2)\right)=4g(h_1,h_2)+\Tr_{g}(h_1)\Tr_{g}(h_2).
	\end{align*}
\end{lem} 
\begin{proof}
	It follows by direct computation. By Definition \ref{def:operator_i_j} of $\i$, we see that
	$$ \i_\vphi h=\frac{1}{3!}(h_{i}^{\ m}\vphi_{mjk}+h_{j}^{\ m}\vphi_{imk}+h_{k}^{\ m}\vphi_{ijm})dx^{ijk} =\frac{1}{2!}h_{i}^{\ m}\vphi_{mjk}dx^{ijk}.$$
	Then by using the \defnref{defn:innerproduct} of inner product, we have
	\begin{align*}
		&	g\left( \i_{\vphi}(h_1),\i_{\vphi}(h_2)\right)\\ =&\frac{1}{3!}\left[ (h_1)_{i}^{\ m}\vphi_{mjk}+(h_1)_{j}^{\ m}\vphi_{imk}+(h_1)_{k}^{\ m}\vphi_{ijm}\right]\\
		&\left[ (h_2)_{a}^{\ r}\vphi_{rbc}+(h_2)_{b}^{\ r}\vphi_{arc}+(h_2)_{c}^{\ r}\vphi_{abr}\right] g^{ia}g^{jb}g^{kc}\\
		=&\frac{1}{2!}\left[ (h_1)_{i}^{\ m}\vphi_{mjk}+(h_1)_{j}^{\ m}\vphi_{imk}+(h_1)_{k}^{\ m}\vphi_{ijm}\right] (h_2)_{a}^{\ r}\vphi_{rbc}g^{ia}g^{jb}g^{kc}.
	\end{align*}
Recall the contraction properties in \lemref{lem:contraction}, which say
$$ \vphi_{ijk}\vphi_{abc}g^{ia}g^{jb}=6g^{kc},\ \ \vphi_{ijk}\vphi_{abc}g^{ia}=g_{jb}g_{kc}-g_{jc}g_{kb}+\psi_{jkbc}.$$ 
Plugging these two identities in the calculation above, we obtain that
\begin{align*}
&	g\left( \i_{\vphi}(h_1),\i_{\vphi}(h_2)\right)\\
=&\frac{6}{2!}(h_1)_{i}^{\ m} (h_2)_{a}^{\ r}g_{mr}g^{ia}+\frac{1}{2!} (h_1)_{j}^{\ m} (h_2)_{a}^{\ r}g^{ia}g^{jb}[g_{ir}g_{mb}-g_{ib}g_{mr}+\psi_{imrb}]\\
&+\frac{1}{2!} (h_1)_{k}^{\ m} (h_2)_{a}^{\ r}g^{ia}g^{kc}[g_{ir}g_{mc}-g_{ic}g_{rm}+\psi_{imrc}]\\
=&\frac{6}{2!}(h_1)_{i}^{\ m} (h_2)_{a}^{\ r}g_{mr}g^{ia}+ \frac{1}{2!}(h_1)^{bm} (h_2)^{ir}[g_{ir}g_{mb}- g_{ib}g_{mr}+\psi_{imrb}]\\
&+\frac{1}{2!} (h_1)^{cm} (h_2)^{ir}[g_{ir}g_{mc}-g_{ic}g_{rm}+\psi_{imrc}].
\end{align*}Note that $\psi_{imrb}$ is skew-symmetric in index $i,\ m,\ r,\ b$ and that $(h_1)^{bm}$ is symmetric in index $b, \ m$. We  have $(h_1)^{bm} (h_2)^{ir}\psi_{imrb}=0.$ Hence we conclude that
	\begin{align*}
			g\left( \i_{\vphi}(h_1),\i_{\vphi}(h_2)\right)
		=&6g(h_1,h_2)+\frac{1}{2}\Tr_g( h_1)\Tr_g (h_2)-g(h_1,h_2)\\
		&+\frac{1}{2}\Tr_g (h_1)\Tr_g (h_2)-g(h_1,h_2)\\
		=&4g(h_1,h_2)+\Tr_g (h_1)\Tr_g (h_2),
	\end{align*}
and the proof is completed.
\end{proof}

Similarly we have an identity for the $\j$ operator.
\begin{lem}\label{lem:norm_j}
	Let $\eta_1,\ \eta_2$ be two $3$-forms, it holds that
	$$g(\j_\vphi \eta_1, \j_\vphi \eta_2)=3g(\eta_1,\eta_2)+ \frac{1}{4}(\eta_1)_a^{\ jk}(\eta_2)^{abc}\psi_{jkbc}
	+\frac{1}{4}\vphi_{ajk}(\eta_1)^{ijk} \vphi_{ibc}(\eta_2)^{abc}.$$
	\end{lem}
\begin{proof}
 By using \propref{prop:local_ex_j} and the inner product given in \defnref{defn:innerproduct}, we see that
 \begin{align*}
 g(\j_\vphi \eta_1, \j_\vphi \eta_2)=&\frac{1}{8}\left[\vphi_{ijk}(\eta_1)_a^{\ jk}+\vphi_{ajk}(\eta_1)_i^{\ jk} \right]\left[ \vphi^i_{\ bc}(\eta_2)^{abc}+\vphi^a_{\ bc}(\eta_2)^{ibc}  \right]\\
 =&\frac{1}{4}\left[\vphi_{ijk}(\eta_1)_a^{\ jk}+\vphi_{ajk}(\eta_1)_i^{\ jk} \right]\vphi^i_{\ bc}(\eta_2)^{abc}.
 \end{align*}
Substituting the following contraction identity stated in \lemref{lem:contraction}
$$\vphi_{ijk}\vphi_{abc}g^{ia}=g_{jb}g_{kc}-g_{jc}g_{kb}+\psi_{jkbc}$$
into the calculation, we get
 \begin{align*}
	g(\j_\vphi \eta_1, \j_\vphi \eta_2)
	=&\frac{1}{4}(\eta_1)_a^{\ jk}(\eta_2)^{abc}[g_{jb}g_{kc}-g_{jc}g_{kb}+\psi_{jkbc}]
	+\frac{1}{4}\vphi_{ajk}(\eta_1)_i^{\ jk} \vphi^i_{\ bc}(\eta_2)^{abc}\\
	=&3g(\eta_1,\eta_2)+ \frac{1}{4}(\eta_1)_a^{\ jk}(\eta_2)^{abc}\psi_{jkbc}
	+\frac{1}{4}\vphi_{ajk}(\eta_1)^{ijk} \vphi_{ibc}(\eta_2)^{abc}.
\end{align*}
The proof is completed here.
\end{proof}
Based on these lemmas, we give the norms of the projection of $3$-forms on $G_2$ module decomposition \eqref{eq:de2}.
\begin{prop}\label{prop:norm_i_j}
	Let $\eta$ be a $3$-form, we have that
	$$|\pi_1^3\eta|^2=\frac{1}{7}g^2(\eta,\vphi),$$
	and that
	$$|\pi_7^3\eta|^2=\frac{1}{4}|\eta|^2+\frac{7}{2}|\pi_1^3 \eta|^2-\frac{1}{16}[\eta_a^{\ jk}\eta^{abc}\psi_{jkbc}
	+\vphi_{ajk}\eta^{ijk} \vphi_{ibc}\eta^{abc}].$$
\end{prop}
\begin{proof}
	We use the decomposition \eqref{eq:decomposition_X} of $3$-form $\eta$ in \propref{prop:local_ex_j}
		\begin{align*}
		\eta=\pi_1^3\eta+\pi_7^3\eta+\pi_{27}^3\eta=3f_0 \vphi+\ast_\vphi (f_1\wedge \vphi)+f_3,
	\end{align*}
with $f_0=\frac{1}{21}g(\vphi,\eta).$
So, it holds that 
\begin{align*}
|\pi_1^3 \eta|^2=9f_0^2|\vphi|^2=63f_0^2=\frac{1}{7}g^2(\vphi,\eta).
\end{align*}

Recall that \eqref{eq:ij} in the proof of \propref{prop:local_ex_j} says
$$\i_\vphi\j_\vphi \eta=54f_0\vphi+4f_3.$$
Then we have
$$\i_\vphi\j_\vphi \eta=18\pi_1^3\eta+4\pi_{27}^3\eta$$
and moreover
\begin{align}\label{eq:norm_IJI}
	|\i_\vphi\j_\vphi \eta|^2=324|\pi_1^3\eta|^2+16|\pi_{27}^3\eta|^2=16|\eta|^2+308|\pi_1^3\eta|^2-16|\pi_7^3\eta|^2,
	\end{align}
since the $G_2$ decomposition \eqref{eq:de2} is a direct-sum decomposition with respect to the $G_2$ metric.

Now we compute the left hand side part of \eqref{eq:norm_IJI} in another method. Due to \lemref{lem:noem_i}, we have
\begin{align}\label{eq:ijnorm}
	|\i_\vphi\j_\vphi \eta|^2=&g(\i_\vphi\j_\vphi \eta, \i_\vphi\j_\vphi \eta)=4g(\j_\vphi \eta,\j_\vphi \eta)+\Tr_g (\j_\vphi \eta) \Tr_g (\j_\vphi \eta).
\end{align}
Recalling the local expression of $\j_\vphi$ defined in \propref{prop:local_ex_j}, we see that
$$\Tr_g (\j_\vphi \eta) =6g(\vphi,\eta).$$
We further have from \lemref{lem:norm_j} that
$$4g(\j_\vphi \eta,\j_\vphi \eta)=12g(\eta,\eta)+\eta_a^{\ jk}\eta^{abc}\psi_{jkbc}
+\vphi_{ajk}\eta^{ijk} \vphi_{ibc}\eta^{abc}.$$
We substitute the above two identities into \eqref{eq:ijnorm} and obtain that
$$	|\i_\vphi\j_\vphi \eta|^2=12g(\eta,\eta)+\eta_a^{\ jk}\eta^{abc}\psi_{jkbc}
+\vphi_{ajk}\eta^{ijk} \vphi_{ibc}\eta^{abc}+36g^2(\vphi,\eta).$$
The proof follows by comparing it with \eqref{eq:norm_IJI}.
\end{proof}

%\begin{cor}\label{rem:B(f_3)}
%\begin{align*}
%\j_\vphi(\pi^3_{27}\eta )_{im} =\frac{1}{2}\left(\eta_{ijk}{\vphi_m}^{jk}+\eta_{mjk}{\vphi_i}^{jk}\right)-\frac{6}{7}g(\vphi,\eta)g_{im}.
%\end{align*}
%\end{cor}
%\begin{proof}
%By the decomposition of $\Omega^3$ given in \eqref{Hodge decomposition}, we have that
%$$ \j_\vphi(\pi^3_{27}\eta )_{im} =\j_\vphi(\eta )_{im} -\j_\vphi(\pi^3_{1}\eta )_{im} -\j_\vphi(\pi^3_{7}\eta )_{im} .$$
%Then the proof follows directly form \propref{prop:operator_i and_B}, \propref{prop:local_ex_j} and \propref{prop:propjectionmap}.
%\end{proof}

In analogy to $\i_\vphi$, $\j_\vphi$, we further define $\i_\omega$, $\j_\omega$ for any $p$-form $\omega$ by using their local expressions.
\begin{defn}\label{def:operator_i general}
Let $\omega$ be a $p$-form.
We define the $\i_\omega$-operator 
\begin{align*}
\i_\omega: S^2(T^*M)\rightarrow\Omega^p(M),\quad \i_{\omega}h=\frac{1}{(p-1)!}h_{i_1l}g^{ls}\omega_{si_2\cdots i_p} dx^{i_1\cdots i_p}.
\end{align*}
\end{defn}
\begin{defn}\label{def:operator_j general}
Let $\omega$ be a $p$-form. We define the $\j_\omega$-operator
	\begin{align*}
	 &\j_\omega: \Omega^p(M) \rightarrow S^2(T^*M),\\
	&\j_{\omega}\omega_2=\frac{1}{2}[\omega_{ia_2\cdots a_p}(\omega_2)_j^{\ a_2\cdots a_p} +\omega_{ja_2\cdots a_p}(\omega_2)_i^{\ a_2\cdots a_p}]dx^i\otimes dx^j.
	\end{align*}
\end{defn}

We now study the properties for the generalized operators $\i$ and $\j$, which will be used in the construction of Levi-Civita connections in \secref{sec:4.1} and \secref{sec:5.1}. Firstly, we see that  
\begin{lem}
 Let $\omega, \ \omega_1$ be any two $p$-forms, we have  $\i_{\omega}g=p \omega$,  $\j_{\omega}\omega_1=\j_{\omega_1}\omega$ and $\Tr_g (\j_\omega \omega_1)=p!g(\omega,\omega_1)  $. 
\end{lem} 
\begin{proof}
The first two identities can be seen directly from the definition of $\i$ and $\j$. The third one follows since
\begin{align*}
\Tr_g (\j_\omega \omega_1)=&g^{ij}(\j_\omega \omega_1)_{ij}=\frac{1}{2}g^{ij}[\omega_{ia_2\cdots a_p}(\omega_1)_j^{\ a_2\cdots a_p} +\omega_{ja_2\cdots a_p}(\omega_1)_i^{\ a_2\cdots a_p}]\\
=&g^{ij}\omega_{ia_2\cdots a_p}(\omega_1)_j^{\ a_2\cdots a_p},
\end{align*}
which is $p!g(\omega,\omega_2)$ by definition of inner product for forms.
\end{proof}
Furthermore, two lemmas are given below, which state the adjoint properties of $\i$ and $\j$ with respect to the $G_2$ metric $g$.
\begin{lem} \label{lem:adjoint_i_operator}
Let $\omega_1, \omega_2$ be any two $p$-forms and $h$ be a symmetric $2$-tensor. We have that
$$ g(\i_{\omega_1}h,\omega_2)=g(\i_{\omega_2}h,\omega_1).$$
\end{lem}
\begin{proof}
By the definition of the $\i$ operator above and the inner product given by \defnref{defn:innerproduct}, we have
\begin{align*}
g(\i_{\omega_1}h,\omega_2) =&\frac{1}{(p-1)!}h_{i_1l}g^{ls}(\omega_1)_{si_2\cdots i_p}(\omega_2)_{j_1\cdots j_p}g^{i_1j_1}g^{i_2j_2}\cdots g^{i_pj_p}.
\end{align*}
Exchanged the index $i_1$ to $s$, $j_1$ to $l$, it becomes
\begin{align*}
	g(\i_{\omega_1}h,\omega_2)
=&\frac{1}{(p-1)!}h_{i_1l}g^{i_1j_1}(\omega_2)_{j_1\cdots j_p} (\omega_1)_{si_2\cdots i_p}g^{ls}g^{i_2j_2}\cdots g^{i_pj_p} \\
=&\frac{1}{(p-1)!}h_{sj_1}g^{sl}(\omega_2)_{lj_2\cdots j_p} (\omega_1)_{i_1i_2\cdots i_p}g^{j_1i_1}g^{i_2j_2}\cdots g^{i_pj_p}.
\end{align*}
Again by the definition of the $\i$ operator and the inner product, we see the formula is identified with $g(\i_{\omega_2}h,\omega_1)$ and hence complete the proof.
\end{proof}
 
\begin{lemma}\label{lem:ex_j}
	Let $\omega_1, \  \omega_2$ be any two $p$-forms and $ \eta_1,\ \eta_2$ be any two $q$-forms. We have that
	$$g(\i_{\omega_1}\j_{\eta_1}\eta_2, \omega_2)=\frac{2}{(p-1)!}g(\j_{\eta_1}\eta_2, \j_{\omega_1}\omega_2)=\frac{(q-1)!}{(p-1)!}g(\eta_1,\i_{\eta_2}\j_{\omega_1}\omega_2).$$
\end{lemma}
\begin{proof}
We have from the definition of the $\i$ operator in \defnref{def:operator_i general} that
	\begin{align*}
		\i_{\omega_1}\j_{\eta_1}\eta_2= \frac{1}{(p-1)!}(\j_{\eta_1}\eta_2)_{ij}g^{jl}(\omega_1)_{l\cdots b_p}dx^{ib_2\cdots b_p}.
\end{align*}
We further compute the inner product as
	\begin{align}\label{eq:exj1}
		g(\i_{\omega_1}\j_{\eta_1}\eta_2, \omega_2)
		=& \frac{1}{(p-1)!} (\j_{\eta_1}\eta_2)_{ij}g^{jl}(\omega_1)_{lb_2\cdots b_p}(\omega_2)^{ib_2\cdots b_p} \notag \\
		=&\frac{1}{(p-1)!}(\j_{\eta_1}\eta_2)_{ij}(\omega_1)_{lb_2\cdots b_p}(\omega_2)_m^{\ b_2\cdots b_p}g^{jl}g^{im}.
	\end{align}
Since $(\j_{\eta_1}\eta_2)_{ij}$ is symmetric, we have by exchanging the index $i$ to $j$ and $l$ to $m$ that
\begin{align*}
		g(\i_{\omega_1}\j_{\eta_1}\eta_2, \omega_2)=&\frac{1}{(p-1)!}(\j_{\eta_1}\eta_2)_{ij}(\omega_2)_{lb_2\cdots b_p}(\omega_1)_m^{\ b_2\cdots b_p}g^{jl}g^{im}.
\end{align*}
Adding \eqref{eq:exj1} to the above on both sides and dividing by 2, we then get
\begin{align*}
		g(\i_{\omega_1}\j_{\eta_1}\eta_2, \omega_2)
	=&\frac{1}{2(p-1)!}(\j_{\eta_1}\eta_2)_{ij}\\
	&[(\omega_1)_{lb_2\cdots b_p}(\omega_2)_m^{\ b_2\cdots b_p}+(\omega_2)_{lb_2\cdots b_p}(\omega_1)_m^{\ b_2\cdots b_p}]g^{jl}g^{im}.
\end{align*}
Recalling the definition of $\j$ in \defnref{def:operator_j general} and the inner products in \defnref{defn:innerproduct}, we have
\begin{align*}
	g(\i_{\omega_1}\j_{\eta_1}\eta_2, \omega_2)=\frac{1}{(p-1)!}(\j_{\eta_1}\eta_2)_{ij}(\j_{\omega_1}\omega_2)_{lm}g^{jl}g^{im}
	=&\frac{2}{(p-1)!}g(\j_{\eta_1}\eta_2, \j_{\omega_1}\omega_2).
\end{align*}
Finally, since the forms $\eta_1,\ \eta_2,\ \omega_1,\ \omega_2$ are chosen arbitrarily, we exchange $\eta_i$ ($i=1,\ 2$) by $\omega_i$ and thus get
	$$g(\eta_1,\i_{\eta_2}\j_{\omega_1}\omega_2)=\frac{2}{(q-1)!}g(\j_{\eta_1}\eta_2, \j_{\omega_1}\omega_2).$$
	Hence
	$$g(\i_{\omega_1}\j_{\eta_1}\eta_2, \omega_2)=\frac{(q-1)!}{(p-1)!}g(\eta_1,\i_{\eta_2}\j_{\omega_1}\omega_2),$$
	which completes the proof.
\end{proof}

%test
%%%%%%%%%%%%%%%%%%%%%%%%%%%%%%%%%%%%%%%%%%%%%%%%%%%%%%%%%%%%%%%%%%%%%%%%%%%%%%%%%%%%%%%%%%%%%%%%%%%%%%%%%%%%%%%%%%%
\subsection{Torsion and curvatures}\label{Torsion and curvatures}\label{sec:2.5}
In \cite[Proposition 1]{BR05}, Bryant introduced a quadruple of torsion forms $$ \tau_0\in \Omega^0(M),\  \tau_1\in\Omega^1(M),\  \tau_2\in \Omega^2_{14}(M),\  \tau_3\in\Omega^3_{27}(M),$$
 so that 
\begin{align*}
	d\vphi=\tau_0\vphi+3\tau_1\wedge \vphi+*\tau_3, \quad
	d\psi=4\tau_1\wedge \psi+\tau_2\wedge \vphi.
\end{align*}

We only concentrate on the case when $\vphi$ is a closed $G_2$ structure, namely 
$$d\vphi=0. $$
Then $\tau_0,\tau_1$ and $\tau_3$ vanish. We now rewrite $\tau_2$ as $\tau$ for simplicity.
Since $\tau\in \Omega^2_{14}(M)$, the $G_2$ module decomposition \eqref{eq:de1} gives
\begin{align} d\psi=\tau \wedge \vphi=-* \tau,
\end{align}which implies 
$\delta \vphi=-* d* \vphi=-* d\psi=\tau.$
So the Hodge Laplacian of $\vphi$ is
\begin{align*}
\Delta \vphi =d\delta \vphi =d\tau.
\end{align*}

In the next Proposition, we give decomposition for $d\a$ when $\a\in \Om_{14}^2$.
\begin{prop}\label{prop:tangent_normalized2}
	Let $\vphi \in \mathcal M$ be a closed $G_2$ structure and $\a \in \Omega_{14}^2$ be a $2$-form satisfying $\delta\a=0$. We then have 
	$$\pi_7^3 d\a=0,\quad g(d\a,\vphi)=g(\a,\tau),\quad \pi_1^3 d\a=\frac{1}{7} g(\a, \tau)\vphi.$$
\end{prop}
\begin{proof}
	Firstly, using $\a\in\Omega_{14}^2$, we have by \eqref{eq:de1} that $$*\a=*(-*(\alpha\wedge\vphi))=-\vphi\wedge \a=-\a\wedge \vphi.$$
	Taking exterior derivative on both sides, it becomes
	$$d*\a=-d\a \wedge \vphi-\a\wedge d\vphi=-d\a \wedge \vphi.$$
	By $\delta \a=*d*\a=0,$ we see that $d\a \wedge \vphi=0$, which infers $\pi_7^3 d\a=0$ by the proof of \propref{prop:decomI}.
	
	Secondly, let $f$ be any arbitrary smooth function on $M$. We now consider the following integral
	\begin{align*}
		\int_M fg(d\a,\vphi)\vol=\int_M g(d\a, f\vphi)\vol=\int_M g(\a, \delta(f\vphi))\vol.
	\end{align*}
	Since $f$ is a function and $\vphi$ is a $3$-form, we compute that 
	\begin{align*}
		\delta(f\vphi)=&-*d*(f\vphi)=-*d(f\psi)=-*(df\wedge \psi+fd\psi)\\
		=&-*(df\wedge \psi)-*f(-*\tau)=-*(df\wedge \psi)+f\tau.
	\end{align*}
	Note that $\a\in\Omega_{14}^2$ and $*(df\wedge \psi)\in \Omega_7^2$ by \eqref{eq:de1}. The integration becomes
	\begin{align*}
		\int_M fg(d\a,\vphi)\vol=& \int_M g(\a, -*(df\wedge \psi)+fd\tau)\vol \\
		=&\int_M g(\a, f\tau)\vol
		=\int_M fg(\a,\tau)\vol. 
	\end{align*}
	Since $f$ is chosen arbitrarily, we have $g(d\a,\vphi)=g(\a, \tau)$ and hence complete the proof by using \propref{prop:decomI}.
\end{proof}
\begin{rem}
	When $\vphi$ is torsion free, these identities were derived in Proposition 2.1 of Donaldson's paper \cite{MR3959094}. We extend these identities to the case when $\vphi$ is a closed $G_2$ structure.
\end{rem}

A direct corollary is obtained by plugging $\a=\tau$ in Proposition \ref{prop:tangent_normalized2}.
\begin{cor}\label{cor:pro_tau}
	Let the $G_2$ structure be closed, i.e. $d\vphi=0$. Then we have that
	$$\pi_7^3d\tau=0,\quad g(\vphi,d\tau)=g(\tau,\tau),\ \pi_1^3 d\tau=\frac{1}{7}g(\tau,\tau)\vphi.$$
\end{cor}
\begin{rem}
The identities in this corollary were obtained by Lotay and Wei in \cite{MR3613456}*{Page 117}.
\end{rem}

In \cite{BR05}, Bryant calculated the Ricci curvature and scalar curvature of closed $G_2$ structure $\vphi$:
\begin{align}\label{eq:Ricci}
	Ric(g)=\frac{1}{4}|\tau|^2_{g}\cdot g +\frac{1}{8}\j_\vphi (2d\tau-*(\tau \wedge \tau)),\quad
Scal(g)=-\frac{1}{2}|\tau|^2_{g}.
\end{align}
Here we denote $|\tau|^2 \doteq g(\tau,\tau).$

We give an estimate of the eigenvalues of tensor $\j_\tau \tau=\tau_i^{\ l}\tau_{lj} e^i\otimes e^j$, which will be used later.
\begin{thm}\label{thm:estimate_tau*tau}
 Let $\tau$ be the torsion  tensor for a closed $G_2$ structure $\vphi$. Then the tensor $-\j_\tau \tau=\tau_i^{\ l}\tau_{lj} e^i\otimes e^j$ is semi-negative definite with maximal eigenvalue $0$, and minimal eigenvalue $\lambda_{min}$ satisfying
	$$\lambda_{min} \geq -\frac{2}{3}|\tau|^2.$$
Besides, we have that
	$ \tau_i^{\ l}\tau_{lj}\tau^{ir}\tau_r^{\ j}= 4|\tau|^4.$
\end{thm}
\begin{proof}
	The idea of this proof comes from \cite{BR05}*{Equation (2.20)}, which says any $2$-form $\beta \in \Omega^2_{14}$ can be written locally as $$\lambda_1 e^{23}+\lambda_2 e^{45}-(\lambda_1+\lambda_2)e^{67}$$ 
	 by using the normalised metric $g_{\vphi}=\Sigma_i e^i\otimes e^i$. Since $\tau\in \Omega_{14}^2$, we use the local expression of $\tau$ to compute that
	$$|\tau|^2 \doteq g(\tau,\tau)={\lambda_1}^2+{\lambda_2}^2+(\lambda_1+\lambda_2)^2$$ and
	\begin{align*}
	-\tau_i^{\ l}\tau_{lj}dx^i\otimes dx^j=& 2{\lambda_1}^2(e^2\otimes e^2+ e^3\otimes e^3)+2{\lambda_2}^2(e^4\otimes e^4+ e^5\otimes e^5)\\
	&+2(\lambda_1+\lambda_2)^2(e^6\otimes e^6+ e^7\otimes e^7).
	\end{align*}
	We see from the identities above that the tensor $-\j_\tau\tau$ is semi-negative definite and has 0 eigenvalue. Besides, the other eigenvalues are
	$$ \{-2{\lambda_1}^2, -2{\lambda_1}^2, -2{\lambda_2}^2, -2{\lambda_2}^2, -2(\lambda_1+\lambda_2)^2, -2(\lambda_1+\lambda_2)^2 \}.  $$
Now we give an estimate for the eigenvalues. When $|\tau|^2\neq 0$, note that the minimum of
	$$ \left\{  \frac{-2{\lambda_1}^2}{{\lambda_1}^2+{\lambda_2}^2+(\lambda_1+\lambda_2)^2}, \frac{-2{\lambda_2}^2}{{\lambda_1}^2+{\lambda_2}^2+(\lambda_1+\lambda_2)^2}, \frac{-2(\lambda_1+\lambda_2)^2}{{\lambda_1}^2+{\lambda_2}^2+(\lambda_1+\lambda_2)^2}     \right\}$$
	 is $-\frac{2}{3}$. We see that the minimal eigenvalue satisfies
	 $$\frac{\lambda_{min}}{|\tau|^2}\geq -\frac{2}{3}.$$ Thus we have
	 $$ \lambda_{min} \geq-\frac{2}{3}|\tau|^2 .$$
	  When $|\tau|^2=0,$ it is trivial that all eigenvalue of $-\j_\tau\tau$ is 0 and thus the minimal eigenvalue $\lambda_{min}\geq -\frac{2}{3}|\tau|^2.$
	 
	Finally, by definition of the inner products in \defnref{defn:innerproduct}, we compute that
\begin{align*}
		&g(\tau_i^{\ l}\tau_{lj}dx^i\otimes dx^j,\tau_i^{\ l}\tau_{lj}dx^i\otimes dx^j)\\	=&\frac{1}{2}[(2{\lambda_1}^2)^2+2({\lambda_1}^2)^2+(2{\lambda_2}^2)^2+(2{\lambda_2}^2)^2+(2(\lambda_1+\lambda_2)^2)^2+(2(\lambda_1+\lambda_2)^2)^2] \\
	=&4[{\lambda_1}^4+{\lambda_2}^4+(\lambda_1+\lambda_2)^4]
%	=8[{\lambda_1}^4+{\lambda_2}^4+2{\lambda_1}^3\lambda_2+3{\lambda_1}^2{\lambda_2}^2+2{\lambda_2}^3\lambda_1] \\
%	=&8({\lambda_1}^2+{\lambda_2}^2+\lambda_1\lambda_2)^2     
   =2[{\lambda_1}^2+{\lambda}^2+(\lambda_1+\lambda_2)^2]^2=2|\tau|^4.
	\end{align*}
The proof is completed.
\end{proof}

%%%%%%%%%%%%%%%%%%%%%%%%%%%%%%%%%
As an application of \thmref{thm:estimate_tau*tau}, we obtain an estimate of the Bakry-Emery Ricci curvature of gradient $G_2$-Laplacian solitons.

A gradient $G_2$-Laplacian soliton is a closed $G_2$ structure satisfying 
 $$ \Delta_\vphi \vphi=\lambda \vphi+d(\Na f\lrcorner \vphi),$$
 where $\lambda$ is a real number and $f$ is a smooth function on $M$. We denote them by $(\vphi, \lambda, \Na f)$.
 Lotay and Wei \cite{MR3613456} computed the Ricci tensor for the $G_2$-Laplacian soliton as the following:
 $$Ric_{ij}=-\frac{2\lambda+|\tau|^2}{6}g_{ij}-\frac{1}{2}\tau_i^{\ k}\tau_{kj}-\Na_i\Na_j f.$$
  The Bakry-Emery Ricci curvature is defined as
  $$Ric_f\doteq Ric+\Hess f.$$
Locally, it says
	 $$(Ric_f)_{ij}=-\frac{2\lambda+|\tau|^2}{6}g_{ij}-\frac{1}{2}\tau_i^{\ k}\tau_{kj}.$$
	 \begin{cor}\label{BER curvature}
	 	The following estimate holds 
	 	$$ -\frac{2\lambda+|\tau|^2}{6}g_{ij} \leq (Ric_f)_{ij}\leq -\frac{2\lambda-|\tau|^2}{6}g_{ij}.$$
	 	In other words,
	 	 	$$ -\frac{\lambda-Scal(g)}{3}g_{ij} \leq (Ric_f)_{ij}\leq -\frac{\lambda+Scal(g)}{3}g_{ij}.$$
	 \end{cor}
 \begin{proof}
 	This is a direct corollary of \thmref{thm:estimate_tau*tau} by applying
 	$$-\frac{2}{3}|\tau|^2g_{ij}\leq \tau_{il}\tau^l_{\ j}\leq0.$$
 \end{proof}
	 
	% The above Proposition yields a  
%%%%%%%%%%%%%%%%%%%%%%%%%%%%%%%%%%%%%%%%%%%%%%%%%%%%%%%%%%%%%%%%%%%%%%%%%%%%%%%%%%%%%%%%%%%%%%%%%%%%%%%%%%%%%%%%%%%%
\section{The space \texorpdfstring{$\mathcal{M}$}{kk} of closed \texorpdfstring{$G_2$}{kk}-structures}\label{sec3}

In this section, we first obtain various variations formula.  Then we characterise the tangent space of $\mathcal M$ and define abstract metrics on $T\mathcal M$. 

%%chechttest2022-10-27, 15:14
%%%%%%%%%%%%%%%%%%%%%%%%%%%%%%%%%%%%%%%%%%%%%%%%%%%%%%%%%%%%%%%%%%%%%%%%%%%%%%%%%%%%%%%%%%%%%%%%%%%%%%%%%%%%%%%%%%%%
\subsection{Variation formulas}\label{Variation formulas}\label{sec:3.1}

In this subsection, we present variations of various quantities and operators for further study. We let $\vphi(t)$ be a one-parameter family of $G_2$-structures and $g(t)$ be the induced $G_2$ metric on the manifold $M$.
We write $g(t)$, $ \vol_{\vphi(t)}$, $\ast_{\vphi(t)}$, $\delta_{\vphi(t)}$, $\Delta_{\vphi(t)}$ as $g$, $\vol$, $\ast$, $\delta$, $\Delta$ for simplicity.

We now apply the $G_2$-module decomposition of 3-forms given in \eqref{Hodge decomposition} associated to each point $\vphi(t)$. For convenience we denote the subspace in \eqref{eq:de2} associated to $\vphi(t)$ simply as $\Omega_p^q$, but it should be pointed out that the subspace $\Omega_p^q$ varies when $\vphi(t)$ changes. There exist three differential forms $f_0\in C^{\infty}(M), f_1\in\Om^1, f_3\in\Om^3_{27}$ for each $\vphi(t)$ such that
\begin{align}\label{eq:decomposition_3form}
 \vphi_t\doteq \frac{\p}{\p t}\vphi =3f_0 \vphi+* (f_1\wedge \vphi)+f_3.
 \end{align}
By virtue of \propref{prop:decomI}, the terms $f_0$ and $f_1$ are given as
	\begin{align}\label{eq:f_0,f_1}
		f_0=\frac{1}{21}g( \vphi_t ,\vphi),\quad f_1=\frac{1}{4}\ast(\vphi_t\wedge\vphi).
	\end{align}
%Applying \propref{prop:operator_i and_B}, we can rewrite \eqref{eq:decomposition_3form} as
% \begin{align}\label{eq:decomposition_3form_1}
% \vphi_t\doteq\frac{\p}{\p t}\vphi = * (f_1\wedge \vphi)+\i_\vphi(h)
% \end{align}
%with $h=h(t)= f_0 g+\frac{1}{4}\j_\vphi(f_3)$.
%\begin{rem}
%We will also use the notations $f^t_0, f^t_1, f^t_3$ instead of $f_0, f_1, f_3$ to emphasis their dependence of the $t$-derivative of $\vphi$.
%\end{rem}
The following lemma states the variations of the corresponding metric and volume form.
\begin{lem}\label{lem:var_g}
\cite[Proposition 4, Equation(6.2),(6.4)]{BR05} The variation of the associated metric is given as
\begin{align}\label{eq:metric variation}
\frac{\p}{\p t}g =\frac{1}{2}\j_\vphi(f_3) +2 f_0 g
\end{align}
and the variation of the volume form is
\begin{align}\label{eq:volume element variation}
	\frac{\p}{\p t}\vol=\frac{1}{2}\tr_{g}(\frac{\p}{\p t}g)\vol=7f_0\vol.
\end{align}
\end{lem}

%\textcolor[rgb]{1,0,0}{\begin{rem}\label{rem:def_f_i}
%We can define operator $f_0, f_1, f_3: \Omega^{3,+}(M)\times \Omega^{3}(M)\rightarrow \Omega^3(M)$ such that $f_i(\vphi, \gamma)$ refers to the $\Omega_1^3, \Omega_7^3, \Omega_{27}^3$ part of $\gamma$ with respect to $G_2$ structure $\vphi$.
%Also, we define $h(\vphi, \gamma)$ as
%$$h(\vphi, \gamma)=\frac{1}{4}\j_\vphi(\gamma) -\frac{7}{2} f_0(\vphi, \gamma).$$
%\end{rem}}

%%%%%%%%%%%%%%%%%%%%%%%%%%%%%%%%%%%%%%%%%%%%%%%%%%%%%%%%%%%%%%%%%%%%%%%%%%%%%%%%%%%%%%%%%%%%%%%%%%%%%%%%%%%%%%%%%%%%%%%%
We then calculate the variation of the Hodge star operator along the family of the $G_2$-structures $\vphi(t)$. The notation $*_t \omega(t)$ will be frequently used, which is defined by
\begin{align}\label{eq:def_*_t}
 *_t \omega(t)= \frac{d}{dt}[*\omega(t)]-*[\frac{d}{dt}\omega(t)].
\end{align}

\begin{prop}[Variation of the Hodge star operator]\label{lem:var_* fix}
Let $\omega$ be a fixed $p$-form. The variation of the Hodge star operator associated to the $G_2$ metric $g$ is given by
\begin{align}\label{eq:var_*_general}
**_t\omega =(7-2p)f_0\omega-\frac{1}{2}\i_{\omega}\j_\vphi(f_3) 
            =(7+7p)f_0\omega-\frac{1}{2}\i_{\omega}\j_\vphi(\vphi_t).
\end{align}
Especially, when $\omega$ is a fixed $3$-form, we have
\begin{align}\label{eq:var_*_3form}
**_t\omega &=f_0\omega-\frac{1}{2}\i_{\omega}\j_\vphi(f_3) =28f_0\omega-\frac{1}{2}\i_{\omega}\j_\vphi(\vphi_t).
\end{align}
\end{prop}
\begin{proof}
Let ${\overline{\omega}}$ be an arbitrary fixed $3$-form on $M$, which means  $\omega_t={\overline\omega}_t=0$. So we have $*_t \omega=(*\omega)_t $ and $*_t {\overline{\omega}}=(*{\overline{\omega}})_t $ and hence compute that
 \begin{align*}
  g(**_t\omega  , {\overline{\omega}}) \vol &=*_t\omega\wedge\overline{\omega}= (*\omega)_t  \wedge {\overline{\omega}} 
  =(*\omega \wedge \overline{\omega})_t  \\
 & =[g(\omega ,{\overline{\omega}})\vol]_t 
  =[(g(\omega ,{\overline{\omega}}))_t+7f_0g(\omega,{\overline{\omega}})] \vol.
  \end{align*}
The last identity follows from \eqref{eq:volume element variation}.
Then
 \begin{align}\label{lem:var_* fix middle equation}
 g(**_t\omega,{\overline{\omega}})=(g(\omega ,{\overline{\omega}}))_t+g(7f_0\omega,{\overline{\omega}}).
   \end{align}
We now compute $(g(\omega ,{\overline{\omega}}))_t$. Under the local coordinate, we have
\begin{align*}
g(\omega,{\overline{\omega}})=\frac{1}{p!}\omega_{i_1\cdots i_p } {\overline{\omega}}_{j_1 \cdots j_p}g^{i_1j_1}\cdots g^{i_p j_p} ,
\end{align*}
and hence by symmetry
\begin{align*}
	[g(\omega,{\overline{\omega}})]_t=&\sum_{k=1}^p\frac{1}{p!}\omega_{i_1\cdots i_p } {\overline{\omega}}_{j_1 \cdots j_p}g^{i_1j_1}\cdots(g^{i_kj_k})_t\cdots  g^{i_p j_p}\\
	=&\frac{1}{(p-1)!}\omega_{i_1\cdots i_p } {\overline{\omega}}_{j_1 \cdots j_p}(g^{i_1j_1})_tg^{i_2j_2}\cdots g^{i_p j_p} .
\end{align*}
Recalling that 
$$\frac{\p}{\p t}g^{ij}=-\frac{\p}{\p t}g_{ab}g^{ia}g^{jb},$$
we then make use of the variation formula \eqref{eq:metric variation} of the $G_2$ metric $g$ to get
\begin{align*}
\frac{\p}{\p t}(g^{i_1j_1})
=&   \left[ -2f_0g^{i_1j_1}-\frac{1}{2}(\j_\vphi(f_3))_{ab}g^{ai_1}g^{bj_1}\right] .
\end{align*}
Hence we see
\begin{align*}
	[g(\omega,{\overline{\omega}})]_t
	=& \frac{-2f_0}{(p-1)!}\omega_{i_1\cdots i_p } {\overline{\omega}}_{j_i \cdots j_p} g^{i_1j_1}g^{i_2j_2}\cdots g^{i_p j_p}\\ &-\frac{1}{2}\frac{1}{(p-1)!}\omega_{i_1\cdots i_p } {\overline{\omega}}_{j_i \cdots j_p} (\j_\vphi(f_3))_{ab}g^{ai_1}g^{bj_1}g^{i_2j_2}\cdots g^{i_p j_p}   .
\end{align*}
Recalling \defnref{defn:innerproduct}, we see the first term is
\begin{align*}
 \frac{-2pf_0}{p!}\omega_{i_1\cdots i_p } {\overline{\omega}}_{j_i \cdots j_p} g^{i_1j_1}\cdots g^{i_p j_p}=-2pf_0g(\omega, \overline{\omega}).
\end{align*}
Again by \defnref{defn:innerproduct} and \defnref{def:operator_i general}, the second term is
\begin{align*}
&-\frac{1}{2}\left[\frac{1}{(p-1)!}\j_\vphi(f_3)_{ab}g^{ai_1}\omega_{i_1\cdots i_p } {\overline{\omega}}_{j_i \cdots j_p} g^{bj_1} g^{i_2j_2}\cdots g^{i_p j_p}\right] \\
&=-\frac{1}{2}g(\frac{1}{(p-1)!}\j_\vphi(f_3)_{ab}g^{ai_1}\omega_{i_1\cdots i_p }dx^{bi_2\cdots i_p} ,{\overline{\omega}}) \\
&= -\frac{1}{2}g(i_\omega \j_\vphi(f_3), \overline{\omega}).
\end{align*}

Putting these identities back to \eqref{lem:var_* fix middle equation}, we have
 \begin{align*}
 g(**_t\omega,{\overline{\omega}})=(7-2p)f_0g(\omega,{\overline{\omega}})-\frac{1}{2}g(\i_{\omega}{\j_\vphi(f_3)} ,{\overline{\omega}}).
 \end{align*}
Note that ${\overline{\omega}}$ is chosen arbitrarily, so we conclude that
  $$ **_t\omega=(7-2p)f_0 \omega-\frac{1}{2}\i_{\omega}\j_\vphi(f_3) .$$
  
 By the decomposition of $\vphi_t$ in \eqref{eq:decomposition_3form}, the above formula can be rewritten as
  \begin{align*}
  **_t\omega &=(7-2p)f_0 \omega-\frac{1}{2}\i_{\omega}\j_\vphi[\vphi_t-*(f_1\wedge \vphi)-3f_0\vphi].
  \end{align*}
From \propref{prop:operator_i and_B}, we have
$\j_\vphi(*(f_1\wedge \vphi))=0,\ \ \j_\vphi\vphi=6g,$ and $\i_\omega g=p\omega$.
Then we see that
  \begin{align*}
	**_t\omega 
            &=(7-2p)f_0 \omega-\frac{1}{2}\i_{\omega}\j_\vphi(\vphi_t)+9f_0\i_\omega g  \\
             &=(7+7p)f_0 \omega-\frac{1}{2}\i_{\omega}\j_\vphi(\vphi_t).
  \end{align*}
 Finally, \eqref{eq:var_*_3form} holds by letting $p=3$ in \eqref{eq:var_*_general}.
\end{proof}
Replacing $f_0$ by the first equation of \eqref{eq:f_0,f_1} in \eqref{eq:var_*_general}, we immediately have 
\begin{cor}\label{lem:var_*_general}
	\begin{align*}
		**_t\omega =\frac{1+p}{3}g(\vphi,\vphi_t)\omega-\frac{1}{2}\i_{\omega}\j_\vphi(\vphi_t).
	\end{align*}
\end{cor}

\propref{lem:var_* fix} leads to the variation of the dual $4$-form $\psi=*\vphi$ which was already obtained by different methods in \cite{J96B}*{Lemma 3.1.1} and \cite{Kar09}*{Theorem 3.4}.

\begin{cor}\label{cor:variation dual form}
The variation of the associated dual form $*\vphi$ is
\begin{align*}
\frac{\p}{\p t}(\ast\vphi)=4f_0\psi +f_1\wedge\vphi-* f_3 
&=*[\frac{4}{3}\pi_1^3\vphi_t+\pi_7\vphi_t-\pi_{27}^3\vphi_t].
\end{align*}
\end{cor}
\begin{proof}
We choose $\eta=\vphi$ in \eqref{eq:var_*_3form} and immediately have the following identity from (6) in \propref{prop:operator_i and_B}
 $$**_t\vphi= f_0\vphi -\frac{1}{2}\i_\vphi \j_\vphi f_3=f_0\vphi -2f_3.$$
Inserting it into
$
\ast\frac{\p}{\p t}(\ast \vphi)=*(*_t \vphi+*\vphi_t)=**_t\vphi +\vphi_t 
$
and using \eqref{eq:decomposition_3form}, we get
\begin{align*}
 \frac{\p}{\p t}(\ast \vphi)&= ***_t\vphi + *\vphi_t \\
&=*(f_0\vphi -2 f_3 )+*[ 3f_0\vphi +*(f_1\wedge \vphi )+f_3]\\
&=4f_0\psi +f_1\wedge\vphi- *f_3,
\end{align*}
which completes the proof.
\end{proof}

%
%\begin{cor}\label{cor prop var_*_general}
%\begin{align}\label{var_*_general_flow}
%*(*\vphi_t)_t-\vphi_{tt} =28f_0\vphi_t-\frac{1}{2}\i_{\vphi_t}\j_\vphi(\vphi_t).
%\end{align}
%Furthermore, under the local coordinate we have
%\begin{align}\label{equation cor prop var_*_general}
%&\i_{\vphi_t}\j_\vphi(\vphi_t)=\frac{1}{4}\left[(\vphi_t)_{ijk}{\vphi}^{ljk}{{(\vphi_t)}}_{lpq}+{{(\vphi_t)}^l}_{jk}{\vphi_i}^{jk}{{(\vphi_t)}}_{lpq}\right]dx^{ipq}.
%\end{align}
%\end{cor}

%%%%%%%%%%%%%%%%%%%%%%%%%%%%%%%%%%%%%%%%%%%%%%%%%%%%%%%%%%%%%%%%%%%%%%%%%%%%%%%%%%%%%%%%%%%%%%%%%%%%%%%%%%%%%%%%%%%%%%%%

We next compute the variation of the codifferential operator $\delta$ acting on $p$-forms.
%Given any $p$-form $\omega$, under the local coordinate, we write
%\begin{align*}
%d\omega &= \frac{1}{p!}(\nabla_{i_{p+1}} \omega)_{i_1\cdots i_p}dx^{i_{p+1} i_1\cdots i_p},\quad
%\delta \omega&= \frac{1}{p-1!}g^{ml }(\nabla_m \omega)_{li_2\cdots i_p}dx^{i_2\cdots i_p}.
%\end{align*}

%\begin{lem}\label{delta product rule}
%Let $\omega$ be a 3-form and $f$ be a smooth function. Then it holds that
%\begin{align}
%\delta(f\omega)=-\nabla f \lrcorner \omega+f\delta \omega.
%\end{align}
%\end{lem}
%\begin{proof}
%By direct computation, we get
%\begin{align*}
%\delta(f\omega) &=-*d*(f\omega)
%           =-*d(f*\omega)
%           =-*(df \wedge *\omega+ f d(*\omega)) \\
%           &=-\nabla f \lrcorner \omega- f*d*\omega
%           =-\nabla f \lrcorner \omega+f\delta \omega.
%\end{align*}
%\end{proof}

\begin{prop}[Variation of the co-differential $\delta$]\label{prop:var_delta}
For any $p$-form $\omega$, we have
\begin{equation}\label{eq:var_delta_I}
	\begin{aligned}
	\delta_t\omega
	=& -7pf_0\delta \omega+(7+7p)\delta(f_0\omega)+\frac{1}{2}\i_{\delta \omega}\j_{\vphi}(\vphi_t)-\frac{1}{2}\delta [\i_{\omega}\j_{\vphi}(\vphi_t)].
	\end{aligned}
\end{equation}
\end{prop}
\begin{proof}
Note that $\delta=-*d*$ and $**=(-1)^{p(7-p)}=1$ on $p$-forms. We have
\begin{align*}
\delta_t \omega &=-(\ast d\ast)_t \omega=-\ast_td\ast \omega-\ast d\ast_t \omega \\
&=-\ast_t d\ast \omega-\ast d\ast\ast\ast_t \omega\\
&=\ast_t\ast\delta \omega+\delta \ast\ast_t \omega.
\end{align*}
By considering the variation of the two sides of $**=1$, we see that $\ast_t\ast=-\ast\ast_t.$ Inserting it into the identity above, we get
\begin{align}\label{eq:var_delta_III}
\delta_t \omega = -\ast\ast_t \delta \omega + \delta \ast\ast_t\omega .
\end{align}
Recalling that $\delta \omega$ is a $(p-1)$-form, it holds by \propref{lem:var_* fix} that
$$ **_t\omega =(7+7p)f_0\delta \omega-\frac{1}{2}\i_{\omega}\j_{\vphi}(\vphi_t),$$
  and that
$$**_t(\delta \omega)=(7+7(p-1))f_0\delta \omega-\frac{1}{2}\i_{\delta \omega}\j_{\vphi}(\vphi_t)=7pf_0\delta \omega-\frac{1}{2}\i_{\delta \omega}\j_{\vphi}(\vphi_t). $$
Substituting them into \eqref{eq:var_delta_III}, we thus complete the proof.
\end{proof}

\begin{cor}\label{cor:var delta_vphi}
Assume that the family of $G_2$-structures $\vphi(t)$ is closed. Then we have
\begin{align*}
\delta_t\vphi=&-g(\vphi,\vphi_t)\delta \vphi+\frac{1}{3}\delta [g(\vphi,\vphi_t)  \vphi] +\frac{1}{2}\i_{\delta \vphi}\j_{\vphi}\vphi_t-2\delta \vphi_t+\frac{1}{2}\delta *[\vphi\wedge*(\vphi\wedge\vphi_t)].
\end{align*}
\end{cor}
\begin{proof}
According to \eqref{eq:var_delta_I}, we have
\begin{align}\label{eq:mid1}
\delta_t\vphi&=-21f_0\delta \vphi+28\delta (f_0  \vphi)+\frac{1}{2}\i_{\delta \vphi}\j_{\vphi}\vphi_t-\frac{1}{2}\delta (\i_{\vphi}\j_{\vphi}\vphi_t) .
\end{align}
Since $*(f_1\wedge \vphi) \in \Omega_7^3$, by (3) of \propref{prop:operator_i and_B}, we have
$$\j_\vphi \vphi=6g,\ \ \ \j_\vphi(*(f_1\wedge \vphi))=0.$$
As a result, we put them into the expression of $\j_\vphi\vphi_t$ to get
$$ \j_\vphi\vphi_t= \j_\vphi(3f_0\vphi+*(f_1\wedge \vphi)+f_3)=18f_0g+\j_\vphi f_3.$$
Then we further conclude from \propref{prop:operator_i and_B} that
$$\i_\vphi \j_\vphi \vphi_t=\i_\vphi(18f_0g+\j_\vphi f_3)=54f_0\vphi+4f_3.$$
Recalling \propref{prop:decomI}, we see 
$$\pi_1^3\vphi_t=3f_0\vphi=\frac{1}{7}g(\vphi,\vphi_t)\vphi;\ \ \pi_7^3\vphi_t=*(f_1\wedge \vphi)=\frac{1}{4}*[\vphi\wedge*(\vphi\wedge\vphi_t)].$$
Inserting the decomposition of $\vphi_t$ and these two identities to $\i_\vphi \j_\vphi \vphi_t$, we get
 $$\i_\vphi \j_\vphi \vphi_t=54f_0\vphi+4(\vphi_t-\pi_1^3\vphi_t-\pi_7^3\vphi_t)=4\vphi_t+2g(\vphi,\vphi_t)\vphi-*[\vphi\wedge*(\vphi\wedge\vphi_t)].$$
Putting these identities back into \eqref{eq:mid1}, we have
\begin{align*}
 \delta_t\vphi=&-g(\vphi,\vphi_t)\delta \vphi+\frac{4}{3}\delta [g(\vphi,\vphi_t)  \vphi]+\frac{1}{2}\i_{\delta \vphi}\j_{\vphi}\vphi_t\\
 &-\frac{1}{2}\delta \{4\vphi_t+2g(\vphi,\vphi_t)\vphi-*[\vphi\wedge*(\vphi\wedge\vphi_t)]\}\\
 =& -g(\vphi,\vphi_t)\delta \vphi+\frac{1}{3}\delta [g(\vphi,\vphi_t)  \vphi] +\frac{1}{2}\i_{\delta \vphi}\j_{\vphi}\vphi_t-2\delta \vphi_t\\
 &+\frac{1}{2}\delta *[\vphi\wedge*(\vphi\wedge\vphi_t)].
\end{align*}
Hence we complete the proof.
\end{proof}

\begin{cor}[Variation of the torsion $\tau$]\label{cor:Var_T}
	When the family of $G_2$-structures $\vphi(t)$ is closed, the derivative of the torsion $\tau$ is given below.
\begin{align*}
		\tau_t=&-g(\vphi,\vphi_t)\delta \vphi+\frac{1}{3}\delta [g(\vphi,\vphi_t)  \vphi] +\frac{1}{2}\i_{\delta \vphi}\j_{\vphi}\vphi_t-\delta \vphi_t+\frac{1}{2}\delta *[\vphi\wedge*(\vphi\wedge\vphi_t)].
\end{align*}
	\end{cor}
\begin{proof}
	We differentiate $\tau=\delta \vphi$ to get
	$$\tau_t=\delta_t\vphi+\delta\vphi_t.$$
	So the conclusion is a direct corollary of \corref{cor:var delta_vphi}.
\end{proof}

At last, we compute the variation of the Hodge-Laplacian operator $\Delta$, which is equal to $$\Delta_t\omega=(\Delta \omega)_t-\Delta \omega_t.$$ 
\begin{prop}[Variation of Hodge Laplacian $\Delta$]\label{cor:var Lap_vphi}
For any closed $p$-form $\omega(t)$, we have
\begin{align*}\Delta_t \omega=&d[\delta **_t\omega -**_t\delta \omega] \\
=&d\delta[(7+7p)f_0\omega-\frac{1}{2}\i_\omega\j_\vphi \vphi_t]-d[7pf_0(\delta \omega)-\frac{1}{2}\i_{\delta \omega}\j_\vphi\vphi_t].
\end{align*}
\end{prop}
\begin{proof}
	Since $\omega$ is closed, we see 
	$\Delta \omega=d\delta \omega.$
	Hence
		$$\Delta_t \omega=(d\delta \omega)_t-d\delta\omega_t=d\delta_t\omega.$$
The proof follows by using \eqref{eq:var_delta_I} and \eqref{eq:var_delta_III} in \propref{prop:var_delta}.
\end{proof}

\begin{cor}\label{cor:var Lap vphi_along_Lap}
Assume that $\vphi(t)$ is the Laplacian flow \eqref{Laplacian flow} starting with a closed $G_2$ structure $\vphi(0)$. Then the evolution of $\Delta\vphi$ along $\vphi(t)$ is
\begin{align*}
(\Delta\vphi)_t =-\Delta^2\vphi  +\frac{1}{3}d\delta[g(\tau,\tau)\vphi]+\frac{1}{2}d[\i_\tau\j_\vphi(d\tau)] -d[g(\tau,\tau)\tau].
\end{align*}
\end{cor}
\begin{proof}
	First, by definition we have $(\Delta \vphi)_t=\Delta_t\vphi+\Delta \vphi_t.$
	Since $\vphi(t)$ remains closed along the Laplacian flow, we can apply \propref{cor:var Lap_vphi} by letting $\omega=\vphi$ and $\vphi_t=\Delta \vphi$. Thus we see $p=3$, $\delta\omega=\delta \vphi=\tau$ and  $$f_0^t=\frac{1}{21}g(\vphi,\vphi_t)=\frac{1}{21}g(\vphi,\Delta\vphi).$$
By substituting the above into \propref{cor:var Lap_vphi}, we get
	\begin{align}\label{eq:var_Lapfolw}\Delta_t \vphi
		=&d\delta[\frac{4}{3}g(\vphi,\Delta\vphi)\vphi-\frac{1}{2}\i_\vphi\j_\vphi (\Delta\vphi)]-d[g(\vphi,\Delta\vphi)\tau-\frac{1}{2}\i_\tau\j_\vphi(d\tau)].
	\end{align}
By \corref{cor:pro_tau}, we have that $g(\vphi,\Delta\vphi)=g(\tau,\tau)$ and $\pi_7^3(\Delta \vphi)=0$.
We further have by the $G_2$ module decomposition
$$\Delta \vphi=\pi_1^3(\Delta \vphi)+\pi_7^3(\Delta \vphi)+\pi_{27}^3(\Delta \vphi)=\frac{1}{7}g(\tau,\tau)\vphi+\pi_{27}^3(\Delta \vphi).$$
Applying \propref{prop:operator_i and_B}, we see that
\begin{align*}
\i_\vphi\j_\vphi (\Delta \vphi)=&\i_\vphi\j_\vphi (\frac{1}{7}g(\tau,\tau)\vphi+\pi_{27}^3(\Delta \vphi))
=\frac{18}{7}g(\tau,\tau)\vphi+4\pi_{27}^3(\Delta \vphi) \\
=&4\Delta\vphi+2g(\tau,\tau)\vphi.
\end{align*}
Substituting it into \eqref{eq:var_Lapfolw}, we get
	\begin{align*}
		\Delta_t \vphi
	=&d\delta[\frac{4}{3}g(\tau,\tau)\vphi-2\Delta\vphi-g(\tau,\tau)\vphi]-d[g(\vphi,\Delta\vphi)\tau-\frac{1}{2}\i_\tau\j_\vphi(d\tau)]\\
	=&\frac{1}{3}d\delta[g(\tau,\tau)\vphi]-2d\delta\Delta\vphi-d[g(\tau,\tau)\tau-\frac{1}{2}\i_\tau\j_\vphi(d\tau)].
\end{align*}
Finally, we conclude that
	\begin{align*}
		(\Delta \vphi)_t =&\Delta_t\vphi+\Delta^2\vphi\\
	=&-\Delta^2\vphi  +\frac{1}{3}d\delta[g(\tau,\tau)\vphi]+\frac{1}{2}d[\i_\tau\j_\vphi(d\tau)] -d[g(\tau,\tau)\tau].
\end{align*}
The proof is completed here.
\end{proof}
\begin{rem}
	The leading term of the linearisation of $\Delta_\vphi\vphi$ is computed in \cite{arXiv:1101.2004}*{ Section 2.4} to show the short time existence of the Laplacian flow. In \cite{MR3613456}*{Equation (3.13)}, Lotay and Wei compute the variation of torsion tensor $T=-\frac{1}{2}\tau$, with leading term $\Delta T$ where $\Delta$ is the "analyst's Laplacian". Since $\Delta_\vphi \vphi=d\tau$ when $\vphi$ is closed. By taking exterior derivative, we see their result coincides with \corref{cor:var Lap vphi_along_Lap}.
\end{rem}

Finally, we prove the following lemma, which states that $**_t$ is a self-adjoint operator.
\begin{lemma}\label{lem:adjoint_*_t}
Let $\omega_1=\omega_1(t),\ \omega_2=\omega_2(t)$ be two $p$-forms along $\vphi(t)$.	The following formula holds true:
	\begin{align*}
		g( \omega_1 ,**_t \omega_2)=g(**_t \omega_1,\omega_2).
	\end{align*}
\end{lemma}
\begin{proof}
	We start the proof by observing that
	\begin{align*}
		 \omega_1\wedge *\omega_2=g(\omega_1,\omega_2)\vol=*\omega_1\wedge \omega_2.
	\end{align*}
	Now we take derivatives about $t$ on both sides. The left hand side is
	\begin{align*}
		\frac{\p}{\p t}[ \omega_1\wedge *\omega_2 ]
		=\p_t \omega_1\wedge *\omega_2+ \omega_1\wedge *_t \omega_2+\omega_1\wedge *(\p_t \omega_2)
	\end{align*}
	and the right hand side is
	\begin{align*}
		\frac{\p}{\p t}[*\omega_1\wedge \omega_2 ]
		= &*_t \omega_1\wedge \omega_2+*(\p_t \omega_1)\wedge \omega_2+*\omega_1\wedge \p_t \omega_2\\
		=& *_t \omega_1\wedge \omega_2+\p_t \omega_1\wedge * \omega_2+\omega_1\wedge *(\p_t \omega_2).
	\end{align*}
	Comparing these two parts, we immediately prove this lemma.
\end{proof}
%\begin{rem}
%	With analogy to \lemref{lem:adjoint_*_t}, we also have that
%	\begin{align*}
%		\int_M X(t)\wedge (* \Delta)_t Y(t)=\int_M (*\Delta)_t X(t)\wedge Y(t),
%	\end{align*}
%	which can be rewritten as
%	\begin{align*}
%		\int_M X(t)\wedge * \Delta_t Y(t)=\int_M *\Delta_t X(t)\wedge Y(t)+\int_M (*_t\Delta-\Delta *_t)X\wedge Y.
%	\end{align*}
%	%It is easy to check that Hodge Laplacian operator in above formula can be replaced by Green operator. In fact,
%\end{rem}

%\begin{cor}\label{cor:var Lap vphi_along_Lap}
%Assume that $\vphi(t)$ is the solution to the Laplacian flow \eqref{Laplacian flow} with initial value being a closed $G_2$ structure. Let
%\begin{align}
%u &=D\vphi .
%\end{align}Then the variation of the Laplacian flow is
%\begin{align}
%\frac{\p u}{\p t}=D(\Delta_\vphi\vphi)=
%\end{align}
%\end{cor}
%\begin{proof}
%
%\end{proof}

%%%%%%%%%%%%%%%%%%%%%%%%%%%%%%%%%%%%%%%%%%%%%%%%%%%%%%%%%%%%%%%%%%%%%%%%%%%%%%%%%%%%%%%%%%%%%%%%%%%%%%%%%%%%%%%%%%%%%%%%

%%%%%%%%%%%%%%%%%%%%%%%%%%%%%%%%%%%%%%%%%%%%%%%%%%%%%%%%%%%%%%%%%%%%%%%%%%%%%%%%%%%%%%%%%%%%%%%%%%%%%%%%%%%%%%%%%%%%%%%%
\subsection{Tangent space \texorpdfstring{$T\mathcal M$ of $\mathcal M$}{TM of M}}\label{sec:3.2}

%Let $\vphi_H$ be the unique harmonic representative of $c$. According to Hodge theory, for any $\vphi\in \mathcal M$, there exists a 2-form $\eta$ such that
%\begin{align}
%\vphi=\vphi_H+d\eta.
%\end{align}
%Actually, $\eta=\delta G\vphi$.
We first study the tangent space of $\mathcal{M}$. Note that $[\vphi]$ is an infinite dimensional affine space, with vectors in $d\Omega^2(M)$. Due to \eqref{eq:G2metric}, we see that for any $\vphi\in \mathcal M$, any $d\a \in d\Omega^2(M)$ and any sufficiently small $\epsilon$, $\vphi+\epsilon d\a$ is still a $G_2$ structure in $\mathcal{M}$. So $\mathcal{M}$ is an open subset of $[\vphi]$.
The above discussion yields the following conclusion.
\begin{prop}\label{prop:312}
The tangent space of $\mathcal M$ at any point $\vphi\in \mathcal M$ is identified to the space of all d-exact $3$-forms, i.e.
\begin{align}\label{tangent space}
T_{\vphi}\mathcal M=d\Om^2.
\end{align}
\end{prop}
 
An element $X\in T\mathcal M$ can be presented by
\begin{align}\label{eq:X d alpha}
X=d\alpha, \text{ for some }\alpha\in \Om^2.
\end{align}
However, the choice of $\alpha$ is in general not unique. We will study several canonical expressions of $X\in T_\vphi \mathcal{M}$ below, which essentially depend on $\vphi$.

 %%%%%%%%%%%%%%%%%%%%%%%%%%%%%%%%%%%%%%%%%%%%%%%%%%%%%%%%%%%%%%%%%%%%%%%%%%%%%%%%%%%%%%%%%%%%%%%%%%%%%%%%%%%%%%%%%%%%%%%%
%\subsubsection{Potential  \texorpdfstring{$3$}{3}-form \texorpdfstring{$u$}{u}}

 We will introduce the potential form $u$ of $X\in T\mathcal M$. Since elements in $T_\vphi\mathcal M$ are $d$-exact forms, we will further show that the potential form $u$ can be chosen as a unique $d$-exact 3-form.
 
Recall that on an oriented compact Riemannian manifold $(M,g)$, the classical Hodge decomposition of the space $\Om^p$ of $p$-forms states that
 \begin{align*}
 	\Om^p= d\Om^{p-1}\oplus\delta \Om^{p+1}\oplus H^p.
 \end{align*} 
Here $H^p$ is the space of harmonic $p$-forms. As a result, the elliptic equation $\Delta_g u=X$ has a solution $u\in \Om^p$ if and only if the $p$-form $X$ is orthogonal to $H^p$. The Green operator $G$ is defined to be the inverse of the Hodge Laplacian operator on the orthogonal of $H^p$. We denote the Hodge decomposition determined by $g$ as 
 \begin{align}\label{eq:Hodegdecomposotion}
 	\a=\pi_d\a+\pi_\delta\a+\pi_H \a,
 \end{align}
 for any $p$-form $\a$ such that $\pi_d,\ \pi_\delta$ and $\pi_H$ refer to the projection map from $\Omega^p$ to $d\Omega^{p-1},\ \delta\Omega^{p+1}$ and$\ H^p$ respectively.
 
\begin{prop}\label{delta u=X}
Let $X$ be an element in the tangent space $T_\vphi\mathcal M$. Then there exists a unique $d$-exact 3-form $u$ such that
\begin{align}\label{eq:potention}
\Delta_\vphi u=X.
\end{align}
\end{prop}

In the converse direction, if we are given a $d$-exact 3-form $u$, we have $\Delta_\vphi u=d\delta u$, which is clearly an element in the tangent space $T_{\vphi}\mathcal M$ by \propref{prop:312}.
\begin{defn}\label{potential form}
	We call $u$ the potential form of $X$ and write it as
	$u=G X.$ 
\end{defn}
Consequently, the tangent space is identified as the set of the potential forms:
$$T_\vphi\mathcal M \widetilde{=} \{u \in d\Omega^2(M)| \exists X\in T_\vphi\mathcal M \text{ s.t. }\Delta_\vphi u=X\}. $$

At last, we prove \propref{delta u=X}, by using the following lemma of the Hodge Laplacian operator and the Green operator, which are direct applications of the Hodge decomposition.

   \begin{lemma} \label{lem:Green}
   For any $d$-exact p-form $d\b$ on $M$, there is a unique $d$-exact p-form $d\mu$ on $M$ such that
\begin{align*}
\Delta d\mu=d\b.
\end{align*}
       Hence the Green operator $G$ is restricted on $d\Omega^{p-1}(M)$, the space of all $d$-exact p-forms.
   \end{lemma}
\begin{proof}
By the Hodge decomposition theorem \eqref{eq:Hodegdecomposotion}, there exists $\beta_1, \ \beta_2$ and $\beta_3 $ such that
$\b=d\b_1+\delta \b_2+\b_3.$
Hence we have $d\b =d(d\b_1+\delta \b_2+\b_3)=d \delta \b_2.$ Applying the Hodge decomposition theorem to $\beta_2$, we obtain $\mu$ such that $\delta \b_2=\delta d\mu$.
Note that
$\Delta d\mu=(d\delta+\delta d)d\mu=d\delta d\mu=d\delta\beta_2=d\b,$ we complete the proof of the existence part.

For the uniqueness part, we assume that $d\overline{\mu}$ also satisfies $\Delta d \overline{\mu}=\b$. Then we have $\Delta d(\mu-\overline{\mu})=0$, which means that the $d$-exact form $d(\mu-\overline{\mu})$ is a harmonic form. By the Hodge decomposition it must vanish. Hence we see $d\mu=d\overline{\mu}$ and therefore the choice of $d\mu$ is unique.
\end{proof}

\begin{rem}
In the above proof, we define the action of the Green operator $G$ on $d\b=d\delta d\mu$ to be
$
Gd\b=d\mu.
$
 This definition can be extended to $\delta \Omega^p(M)$ for all $\delta$-exact $p$-forms, and thus be extended to $d\Omega^{p-1}(M)\oplus \delta \Omega^{p+1}(M)$. We write the potential forms of elements in  $d\Omega^p(M)\oplus \delta \Omega^p(M)$ as
\begin{align*}
G(d\delta d\mu_1+\delta d \delta\mu_2)=d\mu_1+\delta \mu_2   .
\end{align*}
\end{rem}

The following standard formulas of the projection map $\pi_d$ is an application of the Hodge decomposition, which will be used frequently. We include it here for the readers' convenience.
\begin{lemma}\label{Hodge2}
	Let $G$ be the Green operator defined in \lemref{lem:Green}. Then for any $p$-form $\beta$, we have 
\begin{align} \label{pi_d}
\pi_d \b=G d\delta \b=dG\delta\beta.
\end{align}

\end{lemma}
\begin{proof}
According to the Hodge decomposition of p-forms, we have $\b=d\b_1+\delta\b_2+\b_3$ such that $\Delta \b_3=0$. Then $\pi_d  \b=d\b_1$. Now we check the first identity of \eqref{pi_d},
$
G d\delta \b = G d\delta d\b_1
=G (d\delta +\delta d)d\b_1
=G\Delta d\b_1 .
$
From \lemref{lem:Green},
$G \Delta d\b_1
=d\b_1
=\pi_d  \b. $
Hence we obtain $Gd\delta\b=\pi_d \b$.

The second equality holds by using
$dG\delta \b=dG\delta d\beta_1.$
Due to the Hodge decomposition, there are $\mu_1,\ \mu_2$ and $ \mu_3$ such that $\Delta\mu_3=0$ and $ \beta_1=d\mu_1+\delta\mu_2+\mu_3.$ Therefore, 
$ dG\delta \b=dG\delta d\delta \mu_2=dG(d\delta+\delta d)\delta\mu_2=d\delta\mu_2=d\beta_1=\pi_d \b.$
The proof is completed.
\end{proof}
Similarly, we also have
\begin{align}\label{eq:pi-delta}
	\pi_\delta \b=G\delta d\b=\delta Gd\b.
\end{align}

We now compute the variation of the Green operator $G$ and the projection map $\pi_d$.
\begin{prop}\label{prop:var_G}
	The variation of the Green operator $G$ is given as
	\begin{align}\label{eq:var_G}
		G_t X =&Gd[**_t(\delta u)]-\pi_d (**_t u),	
	\end{align}
where $u$ is the potential form of $X$, i.e. $u=GX.$
\end{prop}
\begin{proof}
By taking derivative on both sides of $\Delta G X=X$, we get
	$$\Delta_t GX+\Delta G_t X+X_t=X_t.$$
	We see from the above formula that $\Delta G_t X=-\Delta_t GX$, thus
	$$G_t X=-G\Delta_t GX=-G\Delta_t u.$$
	Using the variation of $\delta$ operator given in \propref{prop:var_delta} and  \propref{Hodge2}, we have
	$$G_t X=-Gd\delta_t u=-Gd[-**_t\delta u+\delta (**_tu)]=Gd[**_t(\delta u)]-\pi_d (**_t u).$$
	The proof is completed.
\end{proof}

\begin{prop}\label{prop:var_pi_d}
	The variation of the projection map $\pi_d$ is
		\begin{align*}
		&(\pi_d)_tX
		=-\pi_d **_t\pi_dX+\pi_d **_tX=\pi_d**_t(X-\pi_d X).
	\end{align*}
\end{prop}
\begin{proof}
	From \lemref{Hodge2}, we have $\pi_d=Gd\delta$. Hence we compute that
	\begin{align*}
		&(\pi_d)_t X=(Gd\delta)_tX=(G)_td\delta X+Gd(\delta)_tX .
		\end{align*}
	Then using the variation of the Green operator and Hodge star operator given in \propref{prop:var_G} and \propref{prop:var_delta}, we have that
	\begin{align*}
		(\pi_d)_tX
		=&Gd(**_t(\delta Gd\delta X))-\pi_d(**_t Gd\delta X)+Gd(-**_t\delta X+\delta **_t X)\\
		=&Gd(**_t \delta X)-\pi_d(**_t Gd\delta X)-Gd(**_t\delta X)+\pi_d\delta **_t X)\\
		=&-\pi_d **_t\pi_dX+\pi_d **_tX .
	\end{align*}
	In the above calculation, we use  $\delta Gd\delta X=\delta \pi_d X=\delta X$ for any $d$-exact form $X$.
\end{proof}

 %%%%%%%%%%%%%%%%%%%%%%%%%%%%%%%%%%%%%%%%%%%%%%%%%%%%%%%%%%%%%%%%%%%%%%%%%%%%%%%%%%%%%%%%%%%%%%%%%%%%%%%%%%%%%%%%%%%%%%%%

%%%%%%%%%%%%%%%%%%%%%%%%%%%%%%%%%%%%%%%%%%%%%%%%%%%%%%%%%%%%%%%%%%%%%%%%%%%%%%%%%%%%%%%%

\subsection{General metrics on closed \texorpdfstring{$G_2$}{G2} spaces} \label{sec:3.3}
 In this section, we define general metrics on $\mathcal{M}$ and the corresponding compatible connections. 

 To begin with, we let $A_\vphi$ be a linear operator as
 $$A_\vphi : d\Omega^2 \rightarrow d\Omega^2.$$
 At each $\vphi\in \mathcal M$, we define a bilinear form on $T_\vphi\mathcal{M}$ as
 \begin{align}\label{eq:def_metric_A}
 \langle X, Y\rangle^A|_\vphi=\int_M g(A_\vphi X, Y)\vol,\  X,Y\in d\Omega^2(M).
 \end{align}
 \begin{defn}
 The bilinear form $\langle X, Y\rangle^A$ is called a metric on $\mathcal M$ if the following two properties hold true:
 \begin{align}\label{eq:cond_A_I}
 &\int_M g(A_\vphi X, Y)\vol=\int_M g(A_\vphi Y, X)\vol,\\
\label{eq:cond_A_II}
 &\int_M g(A_\vphi X, X)\vol > 0,\  \forall X\neq0.
 \end{align}

 \end{defn}
\begin{rem}
	We will see several examples later. In \secref{section:Dirichletmetric},  we study the Dirichlet metric by letting $A_\vphi$ be the Green operator $G_\vphi$. In \secref{L2metric}, we study the Laplacian metric and the $L^2$ metric by letting $A_\vphi$ be the identity operator and $G^2_\vphi $ respectively.
	
	We observe that any map $A_\vphi$ satisfying the condition \eqref{eq:cond_A_II} must be an 1-1 map on $T_\vphi \mathcal M$. Hence its inverse $(A_\vphi)^{-1}$ exists.
\end{rem}

 Now we let $\vphi(t)$ be a smooth path on $\mathcal{M}$ and $\vphi_t(t)=\frac{\p}{\p t}\vphi(t)$ be the $t$-derivative of $\vphi(t)$. We let $X(t), Y(t)$ be two vector fields along $\vphi(t)$ and $X_t(t), Y_t(t)$ be the $t$-derivative along $\vphi(t)$ as 
 $$X_t(t)=\frac{d X(t)}{dt},\ \  Y_t(t)=\frac{d Y(t)}{dt}\ .$$
  The variation of the operator $A_\vphi$ along $\vphi(t)$ is given as
 \begin{align}\label{eq:der_A}
 (A_\vphi)_t X(t) \doteq \frac{d}{dt}[A_{\vphi(t)} X(t)]-A_{\vphi(t)} X_t(t).
 \end{align}

 \begin{defn}
 Let $P$ be a map from $\mathcal{M} \times d\Omega^2(M) \times d\Omega^2(M)$ to $d\Omega^2(M)$, which is linear with respect to the latter two variables. We define
 \begin{align}\label{eq:def_con}
 D_t^{P}(X)(t)=X_t(t)+P(\vphi(t), \vphi_t(t),X(t)).
 \end{align}
\end{defn}
\begin{prop}
 Let $f(t)$ be a smooth function along the curve $\vphi(t)$. The operator $D^{P}$ satisfies
\begin{align*}
	D_t^{P}(X+Y)(t) =&D_t^{P}(X)(t)+D_t^{P}(Y)(t), \\
	D_t^{P}(fX)(t) =&f_t(t)X(t)+f(t) D_t^{P}(X)(t).
\end{align*}
 \end{prop}
\begin{proof}
 It follows from direct computation that
  \begin{align*}
& D_t^{P}(X+Y)(t)=(X+Y)_t(t)+ P(\vphi(t), \vphi_t(t),X(t)+Y(t)) \\
 =&X_t(t)+Y_t(t)+ P(\vphi(t), \vphi_t(t),X(t))+P(\vphi(t), \vphi_t(t),Y(t))\\
 =&D_t^{P}(X)(t)+D_t^{P}(Y)(t)
 \end{align*}
 and
 \begin{align*}
 D_t^{P}(fX)(t)=&(fX)_t(t)+ P(\vphi(t), \vphi_t(t),f(t)X(t)) \\
 =&f_t(t)X(t)+f(t)[(X)_t(t)+ P(\vphi(t), \vphi_t(t),X(t))]\\
 =&f_t(t)X(t)+f(t) D_t^{P}(X)(t).
 \end{align*}
Consequently, the operator $D^P$ given in \eqref{eq:def_con} actually defines a connection.
 \end{proof}
 \begin{defn}\label{defn:torsion}
  Let $\vphi(s,t)$ be a smooth two-parameter family of closed $G_2$-structures in $\mathcal{M}$ and $D$ be a connection on $T\mathcal{M}$. We define the torsion tensor to be $D_t \vphi_s-D_s \vphi_t$. The connection $D$ is called torsion free, if the torsion vanishes, i.e.
 $D_t \vphi_s-D_s \vphi_t=0$ for any family $\vphi(s,t)$.
 \end{defn}
It is direct to compute that the torsion of the $D^{P}$ connection is
$$D^{P}_t \vphi_s -D^{P}_s \vphi_t=P(\vphi, \vphi_t,\vphi_s)-P(\vphi, \vphi_s, \vphi_t).$$
\begin{lemma}\label{lem:tf-cond}
The connection $ D^{P}$ is torsion-free if and only if
 \begin{align*}
 P(\vphi, X, Y)=P(\vphi, Y, X)
 \end{align*}
 holds for all $\vphi \in \mathcal{M}$ and $X, Y \in T\mathcal M$.
\end{lemma}

  \begin{defn}[Geodesic]\label{defn:DPgeo}
	A path $\vphi(t)$ is called a $D^{P}$ geodesic, if
	$$ D^{P}_t(\vphi_t)=0,$$
	which can be further written as
	$$\vphi_{tt}=-P(\vphi, \vphi_t, \vphi_t).$$
\end{defn}

\begin{defn}
	We say the connection $D^{P}$ is compatible with respect to the metric $\langle \cdot, \cdot \rangle^A$, if the following identity holds true for all curves $\vphi(t)$ and vector fields $Y(t),\ Z(t)$ along $\vphi(t)$.
	 \begin{align}\label{eq:conp-I}
		\frac{d}{dt}\langle Y, Z\rangle^A=\langle D^{P}_tY, Z\rangle^A+\langle Y,D^{P}_tZ\rangle^A.
	\end{align}
\end{defn}

 \begin{thm}\label{thm:comp_General}
 The connection  $D^{P}$ is compatible with respect to $\langle\cdot ,\cdot\rangle^A$ if and only if $P$ satisfies the following identity.
  \begin{align}\label{eq:comp_III}
 & \langle P(\vphi, \vphi_t,Y), Z\rangle^A +\langle Y, P(\vphi, \vphi_t, Z)\rangle^A \notag \\
 =&\frac{1}{2}\left[\int_M g((A_\vphi)_tY),Z)\vol+\int_M g((A_\vphi)_tZ, Y)\vol\right] \notag\\
 & +\frac{1}{2}[\langle Y, \pi_d(**_tZ)\rangle^A+\langle Z, \pi_d(**_tY)\rangle^A].
 \end{align}
 \end{thm}
 \begin{proof}
 We insert the definition of the connection $D^P$ into the right part of the compatible condition \eqref{eq:conp-I}.
 \begin{align*}
 &\langle D^{P}_tY, Z \rangle^A+\langle Y,D^{P}_tZ\rangle^A \\
 =& \langle Y_t+P(\vphi, \vphi_t,Y), Z\rangle^A+\langle Y, Z_t+P(\vphi, \vphi_t, Z)\rangle^A\\
 =& \langle Y_t , Z\rangle^A+\langle P(\vphi, \vphi_t,Y), Z\rangle^A+\langle Y, Z_t\rangle^A +\langle Y, P(\vphi, \vphi_t, Z)\rangle^A.
 \end{align*}
Meanwhile, we insert the variation of $A_\vphi$ in \eqref{eq:der_A} into the left part of \eqref{eq:conp-I},
 \begin{align*}
  \frac{d}{dt}\langle Y, Z\rangle^A=& \frac{d}{dt}\int_M A_\vphi Y\wedge *Z \\
   =&\int_M [(A_\vphi)_tY+A_\vphi Y_t)]\wedge *Z+ A_\vphi Y\wedge (*_tZ+*Z_t) 
   \end{align*}
We further write the left part by using the definition of the metric  $\langle\cdot ,\cdot\rangle^A$
\begin{align*}
   \frac{d}{dt}\langle Y, Z\rangle^A=& \int_M g((A_\vphi)_tY),Z)\vol +\langle Y_t, Z \rangle+\langle Y, \pi_d(**_tZ+Z_t)\rangle.
\end{align*}

 By symmetry, we switch between $Y$ and $Z$ to derive
  \begin{align*}
  \frac{d}{dt}\langle Y, Z\rangle^A_\vphi  =& \int_M g((A_\vphi)_tZ, Y)\vol +\langle Z_t, Y \rangle+\langle Z, \pi_d(**_tY+Y_t)\rangle.
 \end{align*}
Adding them together and dividing by 2, we have that the left hand side part of \eqref{eq:conp-I} becomes
  \begin{align*}
	\frac{d}{dt}\langle Y, Z\rangle^A_\vphi =&\frac{1}{2}\left[\int_M g((A_\vphi)_tY,Z)\vol+\int_M g((A_\vphi)_tZ, Y)\vol\right]\\
	& +\frac{1}{2}[\langle Y, \pi_d(**_tZ)\rangle^A+\langle Z, \pi_d(**_tY)\rangle^A]+\langle Z_t, Y \rangle^A+\langle Z, Y_t)\rangle^A.
\end{align*}
Substituting them in \eqref{eq:conp-I}, we thus obtain \eqref{eq:comp_III}.
 \end{proof}

In \secref{sec:4.1} and \secref{sec:5.1}, we will use different expressions of 
 \begin{align*}
 	\int_M g\left( (A_\vphi)_tY,Z\right)\vol,\ \ \int_M g\left( (A_\vphi)_tZ,Y\right)\vol
 \end{align*}
  to find compatible connections obeying the compatible conditions \eqref{eq:comp_III}. 
    %which are listed below in \eqref{eq:comp_conn_sum}.
 %Unfortunately, none of them naturally satisfy the torsion-free condition \eqref{eq:tf-cond}.
% \begin{align}\label{eq:comp_conn_sum}
% & D_t^{P_I^A}(X)=     \notag  \\
% & D_t^{P_{II}^A}(X)=    \notag  \\
% & D_t^{P_{III}^A}(X)=    \notag  \\
% & D_t^{P_{IV}^A}(X)=    \notag  \\
% & D_t^{P_V^A}(X)=   .
 %\end{align}

%%%%%%%%%%%%%%%%%%%%%%%%%%%%%%%%%%%%%%%%%%%%%%%%%%%%%%%%%%%%%%%%%%%%%%%%%%%%%%%%%%%%%%%%%%%%%%%%%%%%%%%%%%%%%%%%%%%%%%%%
\section{Geometry of Dirichlet metric \texorpdfstring{$\mathcal G^D$}{hh}}\label{section:Dirichletmetric}
 In this section, we study the geometry of the infinite dimensional space $\mathcal M$ equipped with a Dirichlet-type metric $\mathcal G^D$. We first define various compatible connections and then calculate their torsion tensors. At last, we introduce a compatible connection for the Dirichlet metric $\mathcal G^D$ and write down the associate geodesic equation.
 
Let $X, Y$ be any two tangent vectors of $T_\vphi \mathcal{M}$, which are two $d$-exact $3$-forms. Due to Proposition \ref{delta u=X}, we let $u, v$ be the corresponding potential forms of $X, Y$, which means they are $d$-exact forms satisfying
	\begin{align}\label{eq:potential functions}
		u=G  X,\quad v=G  Y.
	\end{align}
Here $G $ is the Green operator of the Hodge Laplacian $\Delta  \doteq d\delta+\delta d$ regarding to the $G_2$ structure $\vphi$.

\begin{defn}\label{Dirichlet metric}

We define the \textit{Dirichlet metric} ${\mathcal G}^D$ in the infinite dimensional space $\mathcal M$ in terms of the potential forms as
    \begin{align} \label{eq:metric1}
     {\mathcal G}^D(X,Y)\doteq\int_M g(\delta  u, \delta  v)  \vol.
    \end{align}Note that the operator $\delta$ depends on the $G_2$ structure $\vphi$.
%
%    Moreover, restricted in the gauge-fixed space $\widetilde{\mathcal M}$, it becomes 
%        \begin{align} \label{eq:metric1 normal}
%     {\mathcal G}^D(X,Y)\doteq\int_M g(\alpha,\beta)  \vol,\quad \alpha,\beta \in T\widetilde{\mathcal M}.
%    \end{align}
%    Here, $\alpha=\delta u$ and $\beta=\delta v$, according to Proposition \ref{alpha=delta u}. Thus, we see that the Dirichlet metric is a $L^2$ metric in terms of $\alpha$ in the gauge-fixed space.
\end{defn}
 Our definition is equal to the bilinear form introduced by Bryant-Xu in \cite{arXiv:1101.2004 }, which reads
    \begin{align} \label{eq:metric1 BX}
    \langle X,Y\rangle \doteq\int_M g (G X,Y) \vol  , \quad \forall X,Y\in T_{\vphi}\mathcal M.
    \end{align}
Thanks to \lemref{lem:Green}, we see that the bilinear form \eqref{eq:metric1 BX} is well-defined. We show that \eqref{eq:metric1} in our Definition \ref{Dirichlet metric} coincides with the bilinear form \eqref{eq:metric1 BX}.
\begin{prop}\label{gradient = BX} 
	We have the identity
\begin{align}\label{eq:metric1 gradient}
 \langle X,Y\rangle ={\mathcal G}^D(X,Y), \quad \forall X,Y\in T_{\vphi}\mathcal M.
\end{align}
So, the bilinear form \eqref{eq:metric1 BX} is positive definite and symmetric about $X, Y$.
\end{prop}
\begin{proof}
We use \eqref{eq:potential functions} directly to get
\begin{align*}
 \langle X,Y\rangle =\int_M g (u,\Delta v) \vol .
\end{align*}
Recalling that $u,v$ are $d$-exact forms, by making use of \eqref{Hodge Laplacian} and integration by parts, we have that
\begin{align*}
 \langle X,Y\rangle &=\int_M g  ( u, d\delta v) \vol   =\int_M g  (\delta u, \delta v) \vol  ,
\end{align*}
which equals to ${\mathcal G}^D(X,Y)$. 
\end{proof}

As a result, we also call the bilinear form \eqref{eq:metric1 BX} the \textit{Dirichlet metric} in $\mathcal M$. Definition \eqref{eq:metric1} is motivated from the author's study of the Dirichlet metric in K\"ahler geometry \cites{MR3576284,MR3412344}. We now discuss some direct observations of the Dirichlet metric. Firstly, we see that the Dirichlet metric is also written in terms of wedge product,
    \begin{align} \label{eq:metric1 wedge}
{\mathcal G}^D(X,Y)=\int_M GX\wedge * Y .
    \end{align}
These equivalent definitions will be used frequently in this article.
%%%%%%%%%%%%%%%%%%%%%%%%%%%%%%%%%%%%%%%%%%%%%%%%%%%%%%%%%%%%%%%%%%%%%%%%%%%%%%%%%%%%%%%%%%%%%%%%%%%%%%%%%%%%%%%%%%%%%%%%

%\subsubsection{Volume functional and the Laplacian flow}
 In order to find the torsion free $G_2$-structures, Hitchin \cite{MR1871001} initialed the study of the variational structure of the volume functional of closed positive $3$-forms. The Dirichlet metric is related to the volume functional and the Laplacian flow, which is defined to be
\begin{align}\label{Laplacian flow}
\frac{\p}{\p t}\vphi=\Delta  \vphi.
\end{align}
The following observation was shown in Bryant \cite{BR05}*{Section 6.2.1} and Bryant-Xu \cite{arXiv:1101.2004}*{Section 1.5}.
\begin{prop}\label{prop:GD_Vol}
The Laplacian flow \eqref{Laplacian flow} is the gradient flow of the volume functional up to a constant factor of $\frac{1}{3}$ under the Dirichlet metric ${\mathcal G}^D$.
\end{prop}
\begin{proof}We include the argument here for readers' convenience. It is actually seen from a direct computation. By definition, we have that
	$$\frac{d}{dt} \Vol(\vphi(t))= \int_M \frac{d}{dt}\vol(\vphi(t)).$$
	Using \eqref{eq:f_0,f_1} and \eqref{eq:volume element variation}, we see
	$$ \frac{d}{dt}\vol(\vphi(t))=\frac{1}{3}g(\vphi,\vphi_t)\vol,$$
	and thus
\begin{align*}
\frac{d}{dt} \Vol(\vphi(t))=\frac{1}{3}\int_M  g (\vphi_t , \vphi)\vol 
=\frac{1}{3}\int_M  g (\vphi_t , \pi_d  \vphi)\vol .
\end{align*}
The last identity holds since $\vphi_t$ is $d$-exact. 
We then compute that
 $$G  \Delta  \vphi=G  \Delta  (\pi_d  \vphi+\pi_H  \vphi)= \pi_d  \vphi.$$ 
 So, we have seen from \propref{gradient = BX} that
 $$\frac{d}{dt} \Vol(\vphi(t)) 
 =\frac{1}{3}\int_M  g (\vphi_t , G\Delta \vphi)\vol=\frac{1}{3}\mathcal G^D(X,\Delta \vphi),$$
and proved this proposition.
 \end{proof} 
The Laplacian flow has gained substantial attention in the literature e.g. \cites{MR4295856,MR3613456,MR3934598,MR3951021} and references therein.
%%%%%%%%%%%%%%%%%%%%%%%%%%%%%%%%%%%%%%%%%%%%%%%%%%%%%%%%%%%%%%%%%%%%%%%%%%%%%%%%%%%%%%%%%%%%%%%%%%%%%%%%%%%%%%%%%%%%%%%%

%%%%%%%%%%%%%%%%%%%%%%%%%%%%%%%%%%%%%%%%%%%%%%%%%%%%%%%%%%%%%%%%%%%%%%%%%%%%%%%%%%%%%%%%%%%%%%%%%%%%%%%%%%%%%%%%%%%%%%%%

\subsection{Compatible connections}\label{sec:4.1}
In this section, we will define different compatible connections for the Dirichlet metric ${\mathcal G}^D$ and compare their differences. We will use $ \langle\cdot, \cdot \rangle$ to denote the Dirichlet metric ${\mathcal G}^D$ as well. 

Let  $\vphi(t)$ be a smooth curve on $\mathcal M$ and $X(t)$, $Y(t)$ be any two vector fields along the curve $\vphi(t)$. Then $X(t), \ Y(t)$ are $d$-exact $3$-forms. The potential forms of $X(t)$, $Y(t)$ are given as
\begin{align}\label{eq:potential_along_curve}
	u(t)=G_{\vphi(t)} X(t), \ v(t)=G_{\vphi(t)}Y(t).
\end{align}
We will omit $\vphi(t)$ in the operators for convenience.

 Now we recall the settings in \secref{sec:3.3}. We see that the Dirichlet metric is exactly the metric \eqref{eq:def_metric_A} with $A_\vphi =G$. 

In the following lemma we will use different formulas of the Dirichlet metric shown above, aiming to get its different compatible connections.
\begin{lemma}\label{lem:G_tI}
	\begin{align*}
		\int_Mg(G_tX, Y)\vol=\int_Mg(Gd[**_t(\delta u)], Y)\vol-\int_Mg(\pi_d (**_t u), Y)\vol.
	\end{align*}
\end{lemma}
\begin{proof}
	This is a direct corollary of \propref{prop:var_G}, which gives the variation of the Green operator as
	$$G_tX=Gd(**_tGX)-\pi_d[**_t(GX)]=Gd[**_t(\delta u)]-\pi_d(**_t u). $$ 
	Putting it into the left part of the formula in the lemma we immediately complete the proof. 
\end{proof}
Now we deal with the two terms in the lemma above respectively.
\begin{lemma}\label{lem:G_tII}
	$$\int_Mg(Gd[**_t(\delta u)],Y)\vol=\langle d [**_t (\delta u)], Y \rangle=\langle d [**_t (\delta v)], X \rangle.$$
\end{lemma}
\begin{proof}
	First, using the equivalent definition of the $\mathcal G^D$ metric shown by \propref{gradient = BX}, we see that
	\begin{align*}
		\int_Mg(Gd[**_t(\delta u)],Y)\vol=\langle d [**_t (\delta u)], Y \rangle.
		\end{align*}
	Now we use integration by parts and the adjoint property of $*_t$ operator in \lemref{lem:adjoint_*_t} to see that
		\begin{align*}
		\int_Mg(Gd[**_t(\delta u)],Y)\vol=&\int_M g(d[**_t(\delta u)],GY)\vol =\int_M g(**_t(\delta u),\delta v)\vol \\
		=&\int_M g(\delta u,**_t(\delta v))\vol =\int_M g(GX,d[**_t(\delta v)])\vol \\
		=&\langle d [**_t(\delta v)], X \rangle.
	\end{align*}
The proof is completed.	
\end{proof}
\begin{lemma}\label{lem:G_tIII}
$$\int_Mg(\pi_d (**_t u), Y)\vol=\langle X ,\pi_d( **_t Y) \rangle=\langle d\delta (**_t u), Y  \rangle.$$   
\end{lemma}
\begin{proof}
	Since $Y$ is $d$-exact, we see that
	$$	\int_Mg(\pi_d (**_t u), Y)\vol=\int_Mg(**_t u , Y)\vol. $$
	Hence we have from \lemref{lem:adjoint_*_t} and \propref{gradient = BX} that
	\begin{align*}
	\int_Mg(\pi_d (**_t u), Y)\vol=&\int_Mg( u , **_t Y)\vol 
	=\int_Mg( u ,\pi_d( **_t Y) )\vol\\
	 =& \langle X ,\pi_d( **_t Y) \rangle.
\end{align*}
Also, recalling that $\pi_d=Gd\delta,$ we can rewrite the left part as
\begin{align*}
	\int_Mg(\pi_d (**_t u), Y)\vol=\int_Mg(Gd\delta (**_t u), Y)\vol =\langle d\delta (**_t u), Y \rangle,
\end{align*}
which finishes the proof.
\end{proof}

\begin{lem}\label{lem:G_tIV}
	It holds that
	\begin{align} \label{eq:Var_GIII}
		\int_Mg(\pi_d (**_t u), Y)\vol=&28\langle d\delta(f_0^tu), Y \rangle- \frac{1}{2} \langle\pi_d(\i_Y\j_\vphi \vphi_t), X \rangle.
	\end{align}
\end{lem}
\begin{proof}
	We start by using \propref{lem:var_* fix} to rewrite the left side of \eqref{eq:Var_GIII} as
	\begin{align}\label{eq:Var_GII}
		\int_Mg(\pi_d (**_t u), Y)\vol=&\int_Mg(**_t u, Y)\vol =28\int_Mg( f_0^tu-\frac{1}{2} \i_u\j_\vphi \vphi_t, Y)\vol \notag\\
		=&28\int_M  f_0^tg(u, Y)\vol-\frac{1}{2}\int_M g(\i_u\j_\vphi \vphi_t, Y)\vol.
	\end{align}
Similar to the proof of the former two lemmas, the first part is
$$28\int_M  f_0^tg(u, Y )\vol=28\langle X, \pi_d(f_0^tY)\rangle .$$
Note that $\j_\vphi \vphi_t$ is a symmetric $2$-tensor.	By \lemref{lem:adjoint_i_operator}, the second part is
$$	-\frac{1}{2}	\int_M g( \i_u\j_\vphi \vphi_t, Y)\vol=-\frac{1}{2}	\int_M g( \i_Y\j_\vphi \vphi_t, u)\vol=-\frac{1}{2}\langle d\delta(\i_Y\j_\vphi \vphi_t), X\rangle .$$
The proposition follows by substituting the above two identities into \eqref{eq:Var_GII}.
\end{proof}

By combining \lemref{lem:G_tII}, \lemref{lem:G_tIII}, \lemref{lem:G_tIV} with \lemref{lem:G_tI}, we get the following proposition. 

\begin{prop}\label{prop:expr_G_t}
	It holds that
	\begin{align*}
		\int_M g(G_tX, Y)\vol=&\langle d [**_t (\delta u)], Y \rangle-\langle X ,\pi_d( **_t Y) \rangle\\
		=&\langle d [**_t (\delta u)], Y \rangle-\langle d\delta (**_t u), Y  \rangle\\
		=&\langle d [**_t (\delta u)], Y \rangle-28\langle d\delta(f_0^tu), Y \rangle+ \frac{1}{2} \langle\pi_d(\i_Y\j_\vphi \vphi_t), X \rangle. 
	\end{align*}
\end{prop}
\begin{proof}
	Recall the equalities in \lemref{lem:G_tII} and \lemref{lem:G_tIII} and then substitute them into \lemref{lem:G_tI}. We immediately get the first part of the proposition.
	The other parts follow similarly.
 \end{proof}

Now we use the above lemmas to give compatible connections of the $\mathcal G^D$ metrics. 
Now we let
\begin{align} 
P^A(\vphi,\vphi_t,Y)&=\frac{1}{2}d\left[**_t( \delta v)  \right] ; \label{eq:PA}\\
P^B(\vphi,\vphi_t,Y)&=P^A(\vphi,\vphi_t,Y)+\frac{1}{2}\{\pi_d(**_t Y)-d\delta(**_tv)\};\label{eq:PB} \\ 
% \\ =&\frac{1}{2}\{\pi_d(**_t Y)+d\left[**_t( \delta v)  \right]-d\delta(**_tv)\};
P^C(\vphi,\vphi_t,Y)&=P^B(\vphi,\vphi_t,Y)+\frac{1}{2}\left\{d\delta(**_tv)-28d\delta(f_0^tv)+\frac{1}{2}\pi_d(\i_Y\j_\vphi\vphi_t)\right\}\label{eq:PC}\\
=&P^A(\vphi,\vphi_t,Y)+\frac{1}{2}\left\{\pi_d(**_tY) -28d\delta(f_0^tv)+\frac{1}{2}\pi_d(\i_Y\j_\vphi\vphi_t)\right\}. \notag
%P^C(\vphi,\vphi_t,Y)=&\frac{1}{2}\left\{\pi_d(**_tY)+d\left[**_t( \delta v)\right] -28d\delta(f_0^tv)+\frac{1}{2}\pi_d(\i_Y\j_\vphi\vphi_t)\right\} \label{eq:PC}
\end{align}

We define a list of connections of $\mathcal M$ based on the formulas above.

 \begin{defn}\label{defn:con}
 We define the $D^{P^A},\quad D^{P^B},\quad D^{P^C}$ connections by
    \begin{align} \label{eq:con2}
     (D_t^{P^A}Y)(t)\doteq  & Y_t+P^A(\vphi,\vphi_t,Y),\\
	(D_t^{P^B}Y)(t)\doteq&   Y_t+P^B(\vphi,\vphi_t,Y), \label{eq:con1}\\
		(D_t^{P^C}Y)(t)\doteq&   Y_t+P^C(\vphi,\vphi_t,Y). \label{eq:con3}
\end{align}
\end{defn}

We now prove the following compatibility property.
\begin{prop}\label{compatible2}
The connections defined above are all compatible with the Dirichlet metric ${\mathcal G}^D$.
\end{prop}
\begin{proof}
	Recall the setting in \secref{sec:3.3}. Here the operator $A_\vphi$ is chosen to be the Green operator $G=G_\vphi$. Due to \propref{thm:comp_General}, a connection $D_t^PX=X_t+P(\vphi, \vphi_t,X)$ is compatible with respect to $\mathcal G^D$, if and only if the following identity holds:
\begin{align}\label{eq:Com_GD}
			\langle P(\vphi, \vphi_t,Y), X\rangle +\langle Y, P(\vphi, \vphi_t, X)\rangle
			&=\frac{1}{2}\int_M[ g(G_tY,X)+ g(G_tX, Y)]\vol\\
		&	+\frac{1}{2}\langle Y, \pi_d(**_tX)\rangle
			+\frac{1}{2}\langle X, \pi_d(**_tY)\rangle.\notag
\end{align}
Now we compute the right part of \eqref{eq:Com_GD} by using different expressions of 
$$\int_M g(G_tY,X)\vol,\quad \int_M g(G_tX,Y)\vol$$
 given in \propref{prop:expr_G_t}. Substituting the first equality of \propref{prop:expr_G_t}, the right part of \eqref{eq:Com_GD} becomes
\begin{align*}
	Right
		=&\frac{1}{2}\left[\langle d [**_t (\delta u)], Y \rangle-\langle X ,\pi_d( **_t Y) \rangle\right]\\
		&+\frac{1}{2}\left[\langle d [**_t (\delta v)], X \rangle-\langle Y ,\pi_d( **_t X) \rangle\right]\\
	&
	+\frac{1}{2}\langle Y, \pi_d(**_tX)\rangle+\frac{1}{2}\langle X, \pi_d(**_tY)\rangle\\
	=&\frac{1}{2}\left[\langle d [**_t (\delta u)], Y \rangle+\langle d [**_t (\delta v)], X \rangle\right] .	
	\end{align*}
Noticing the definition of $P^A$ in \eqref{eq:PA}, we see the right part equals to
$$	\langle P^A(\vphi, \vphi_t,Y), X\rangle +\langle Y, P^A(\vphi, \vphi_t, X)\rangle,$$
which is the left part of \eqref{eq:Com_GD} exactly.
Thus we see the connection $D^{P^A}Y=Y_t+P^A(\vphi,\vphi_t,Y)$ is compatible with the $\mathcal G^D$ metric.

The compatibility properties of the $D^{P^B}$ and the $D^{P^C}$ connections follow similarly by replacing the first equality of \propref{prop:expr_G_t} with the second and the third equalities of that in the proof.
\end{proof}

\begin{prop}
	The compatible connections defined in \defnref{defn:con} can be rewritten as
	\begin{align}
	D_t^{P^A}Y= &  Y_t+\frac{1}{2}d\left[g(\vphi,\vphi_t)\delta v-\frac{1}{2}\i_{\delta v}\j_\vphi \vphi_t \right], \label{eq:con2_I} \\
D_t^{P^B}Y=& D_t^{P^A}Y+\frac{1}{2}\left\{\pi_d\left[\frac{4}{3}g(\vphi,\vphi_t)Y-\frac{1}{2}\i_{Y}\j_\vphi \vphi_t\right]\right.\label{eq:con1_I}\\
&-\left. d\delta\left[\frac{4}{3}g(\vphi,\vphi_t) v-\frac{1}{2}\i_{ v}\j_\vphi \vphi_t\right]\right\}, \notag\\
D_t^{P^C}Y=&D_t^{P^A}Y+\frac{2}{3}\left\{\pi_d[g(\vphi,\vphi_t)Y] -d\delta[g(\vphi,\vphi_t)v]\right\}. \label{eq:con3_I}
	\end{align}
\end{prop}\label{prop:DP_explicit}
\begin{proof}
Note that $Y, v=GY$ are $3$-forms and $\delta v$ is a $2$-form. We see directly from \corref{lem:var_*_general} that
$$**_tY=\frac{4}{3}g(\vphi,\vphi_t)Y-\frac{1}{2}\i_Y\j_\vphi \vphi_t, \ \  **_tv=\frac{4}{3}g(\vphi,\vphi_t)v-\frac{1}{2}\i_v\j_\vphi \vphi_t,$$
and
$$**_t(\delta v)=g(\vphi,\vphi_t)\delta v-\frac{1}{2}\i_{\delta v}\j_\vphi \vphi_t.$$
Thus the conclusions follows from inserting these identities into \eqref{eq:PA}, \eqref{eq:PA} and \eqref{eq:PA}:
\begin{align*}
P^A(\vphi,\vphi_t,Y)=&\frac{1}{2}d\left[g(\vphi,\vphi_t)\delta v-\frac{1}{2}\i_{\delta v}\j_\vphi \vphi_t \right];\\
P^B(\vphi,\vphi_t,Y)=&P^A(\vphi,\vphi_t,Y)+\frac{1}{2}\left\{\pi_d\left[\frac{4}{3}g(\vphi,\vphi_t)Y-\frac{1}{2}\i_{Y}\j_\vphi \vphi_t\right]\right.\notag\\
&-\left. d\delta\left[\frac{4}{3}g(\vphi,\vphi_t) v-\frac{1}{2}\i_{ v}\j_\vphi \vphi_t\right]\right\};\\
P^C(\vphi,\vphi_t,Y)=&P^A(\vphi,\vphi_t,Y)+\frac{2}{3}\left\{\pi_d[g(\vphi,\vphi_t)Y] -d\delta[g(\vphi,\vphi_t)v]\right\}.
\end{align*}
\end{proof}

 We now construct a family of compatible connections by using the linear combination of the connections constructed above.
\begin{defn}
We define $D^{a,b,c}$ as
$$D^{a,b,c}_tY=aD^{P^A}_tY+bD^{P^B}_tY +cD^{P^C}_tY,$$
where $a,\ b,\ c$ are arbitrary real numbers.
\end{defn}
\begin{prop}\label{compatible combination}
When $a+b+c=1$, the combination $D^{a,b,c}$ defines a compatible connection for the Dirichlet metric.
\end{prop}
\begin{proof}
For any vector fields $Y(t),Z(t)$ along $\vphi(t)$, we have
\begin{align*}
&\ \ \langle D^{a,b,c}_t Y, Z\rangle+\langle D^{a,b,c}_t Z, Y\rangle \\
&=\langle aD^{P^A}_tY+bD^{P^B}_tY +cD^{P^C}_tY,Z \rangle +\langle aD^{P^A}_tZ+bD^{P^B}_tZ +cD^{P^C}_tZ,Y \rangle \\
&=a(\langle D^{P^A}_tY,Z\rangle+\langle D^{P^A}_tZ,Y \rangle)+b(\langle D^{P^B}_tY ,Z \rangle+\langle D^{P^B}_t Z ,Y \rangle )\\
&\ \ +c(\langle D^{P^C}_tY,Z \rangle+\langle D^{P^C}_tZ,Y \rangle )\\
&=a X\langle Y,Z\rangle+b X\langle Y,Z\rangle+c X\langle Y,Z\rangle \\
&=(a+b+c) X\langle Y,Z\rangle.
\end{align*}
When $a+b+c=1$, it becomes $X\langle Y, Z\rangle$. Thus $ D^{a,b,c}$ is a compatible connection.
\end{proof}
%%%%%%%%%%%%%%%%%%%%%%%%%%%%%%%%%%%%%%%%%%%%%%%%%%%%%%%%%%%%%%%%%%%%%%%%%%%%%%%%%%%%%%%%%%%%%%%%%%%%%%%%%%%%%%%%%%%%%%%%
Now we let $\vphi(s,t)$ be a smooth two parameter family of $G_2$-structures in $\mathcal{M}$. We use the notations given below:
\begin{align*}
X(t,s)=\frac{\p}{\p t}\vphi(t,s), \quad Y(t,s)=\frac{\p}{\p s}\vphi(t,s).
\end{align*}
Thus we have that
$$X_s=Y_t.$$
We now compute the torsion tensor of the connections defined in \defnref{defn:con}. We will denote the torsion of the $D^{P^A}$ connection as $T^{P^A}$, and the torsions for the $D^{P^B},\ D^{P^C}$ connections. As shown in \defnref{defn:torsion}, the torsion is
$$ T(X, Y)=D_t Y-D_s X. $$
Clearly, $T$ is skew-symmetric, i.e.
$$T(X,Y)=-T(Y,X).$$
\begin{prop}\label{prop:torsion_GDcompatible}
	Let $Z(t)$ be any vector field along $\vphi(s,t)$ and $\omega(t)=GZ(t)$ be the corresponding potential form. The torsion tensors of the connections mentioned above are 
  \begin{align*}
  \langle T^{P^A}(X,Y),Z\rangle &=\frac{1}{2}\int_M[ g(\vphi, X)g(\delta v,\delta w)-g(\vphi, Y)g(\delta u,\delta w)]\vol\\
&-\frac{1}{4}\int_M [ g(\i_{\delta v}\j_\vphi X, \delta w)-g(\i_{\delta u}\j_\vphi Y, \delta w)]\vol, \\
 \langle T^{P^B}(X,Y),Z \rangle &= \langle T^{P^A}(X,Y),Z\rangle \\
&+\frac{2}{3}\int_M[ g(\vphi,X)g(Y,w)-g(\vphi,Y)g(X, w)]\vol \\
&+\frac{1}{4}\int_M[ g(\i_{Y}\j_\vphi X,  w)-g(\i_{X}\j_\vphi Y,  w)]\vol\\
&-\frac{2}{3}\int_M[ g(\vphi,X)g(v,Z)-g(\vphi,Y)g(u, Z)]\vol\\
&-\frac{1}{4}\int_M[ g(\i_{ v}\j_\vphi X, Z)-g(\i_{ u}\j_\vphi Y, Z)]\vol,\\
\langle T^{P^C}(X,Y),Z \rangle &=\langle T^{P^A}(X,Y),Z\rangle +\frac{1}{2}\int_M[ f_0^tg(Y,w)-f_0^sg(X, w)\vol\\
&-\frac{1}{2}\int_M[ f_0^tg(v,Z)-f_0^sg(u, Z)]\vol.
  \end{align*}
\end{prop}
\begin{proof}
	This follows by direct calculation and integration by parts formula. We now give the first part of the proof and the other two follow similarly.
	 
	 First, by assumption we see $X_s=Y_t$. Using the explicit formula of the $D^{P^A}$ connection given by \propref{prop:DP_explicit}, we get
	 \begin{equation}\label{eq:TPAI}
	 \begin{aligned}
	  T^{P^A}(X,Y) =& Y_t+\frac{1}{2}d\left[g(\vphi,X)\delta v-\frac{1}{2}\i_{\delta v}\j_\vphi X \right] -X_s\\
	  &-\frac{1}{2}d\left[g(\vphi,Y)\delta u
	  -\frac{1}{2}\i_{\delta u}\j_\vphi Y \right]\\
	  =&\frac{1}{2}d\left[g(\vphi,X)\delta v- g(\vphi,Y)\delta u\right]-\frac{1}{4} d\left[\i_{\delta v}\j_\vphi X-\i_{\delta u}\j_\vphi Y  \right] .
	 	\end{aligned}
 	\end{equation}
	Taking the $G^D$ metric with $Z$ and using integration by parts, we see that
	\begin{align*}
	 \langle T^{P^A}(X,Y),Z\rangle 
=& \int_M g\left(\frac{1}{2}d\left[g(\vphi,X)\delta v- g(\vphi,Y)\delta u\right], GZ\right)\vol\\
&-\int_M g\left( \frac{1}{4} d\left[\i_{\delta v}\j_\vphi X-\i_{\delta u}\j_\vphi Y  \right] , GZ\right) \vol\\
=& \int_M g\left(\frac{1}{2}\left[g(\vphi,X)\delta v- g(\vphi,Y)\delta u\right], \delta w\right)\vol\\
&-\int_M g\left( \frac{1}{4} \left[\i_{\delta v}\j_\vphi X-\i_{\delta u}\j_\vphi Y  \right] , \delta w\right) \vol.
\end{align*}
Now since 
$$g(g(\vphi,X)\delta v ,\delta w)=g(\vphi,X)g(\delta v, \delta w),$$
 we complete the proof of the first part.
\end{proof}

%Now we compute the geodesics for the connections defined in \defnref{defn:con}. Recall that a curve $\vphi(t)$ on $\mathcal{M}_c$ is a geodesic for the connection \eqref{eq:con1} if it satisfies the equation
%\begin{align*}
%	D_{\vphi_t}\vphi_t =0.
%\end{align*}
%\begin{lem}\label{dg geodesics}
%	Let $Y=\vphi_t$ and $v=GY=G\vphi_t$. The geodesic equations are
%	\begin{align*}
%		0=(D_t^{P^A}\vphi_t)(t)= &  \vphi_{tt}+\frac{1}{2}d\left[3f_0^t\delta v-\frac{1}{2}i_{\delta v}j_\vphi f_3^t \right],  \\
%		0=(D_t^{P^B}\vphi_t)(t)=&   \vphi_{tt}+\frac{1}{2}\left\{\pi_d(f_0^t\vphi_t-\frac{1}{2}i_{\vphi_t}j_\vphi f_3^t)+d\left[3f_0^t\delta v-\frac{1}{2}i_{\delta v}j_\vphi f_3^t \right]\right.  \\
%		&\left. -d\delta(f_0^t v-\frac{1}{2}i_{ v}j_\vphi f_3^t)\right\}, \\
%		0=(D_t^{P^C}\vphi_t)(t)=& \vphi_{tt}+\frac{1}{2}\left\{\pi_d(f_0^t\vphi_t)+d\left[[3f_0^t\delta v-\frac{1}{2}i_{\delta v}j_\vphi f_3^t \right] -d\delta(f_0^tv)\right\}. 
%	\end{align*}
%\end{lem}

\subsection{$\mathcal G^D$-Levi-Civita Connection}\label{sec:4.2}
In this section, we will introduce a Levi-Civita connection for the Dirichlet metric $\mathcal G^D$ on $\mathcal M$. Recall the definition in finite dimensional Riemannian manifold. A $\mathcal G^D$-Levi-Civita connection refers to a connection which is both torsion-free and compatible with the $\mathcal G^D$ metric. In this section, we will construct a Levi-Civita connection of the $\mathcal G^D$ metric based on the torsion tensor and contorsion tensor of the compatible connection $D^{P^A}$ defined in the previous subsection. Also, recalling the contorsion tensor in the finite dimensional Riemannian geometry, we give the following definition.
\begin{defn}\label{defn:K}
The contorsion $K$ of the  $D^{P^A}$ connection is defined to be a $2$-tensor satisfying the following identity:
\begin{align*}
2\mathcal G^D\langle K(X,Y), Z\rangle =&\mathcal G^D\langle T^{P^A}(X,Y), Z\rangle -\mathcal G^D\langle T^{P^A}(Y,Z), X\rangle \\
&+\mathcal G^D\langle T^{P^A}(Z,X), Y\rangle . 
\end{align*}
\end{defn} 
We let $\vphi(s,t,r)$ be a family of $G_2$-structures on $\mathcal M$ and $X=\vphi_t, \ Y=\vphi_s, \ Z=\vphi_r$. Also, we let $u=GX,\ v=GY$ and $w=GZ$ be the potential forms. 
\begin{defn}\label{defn:S}
	We define a tensor $S$ as
	$$S(Y,Z)=\frac{1}{4}d[\i_{\delta w}\j_\vphi Y-2g(Y, \vphi)\delta w]+\frac{1}{2}d\delta[g(\delta v,\delta w)\vphi-\i_\vphi \j_{\delta v}\delta w].$$
\end{defn}
 To compute $K$, we need the following lemma.

\begin{lemma}\label{lem:S_operator}
	It holds that
	$$	\mathcal G^D\langle T^{P^A}(X,Y), Z\rangle =\mathcal G^D\langle S(Y,Z), X\rangle .$$

\end{lemma}
\begin{proof}
	We apply \propref{prop:torsion_GDcompatible} to get
	\begin{align*}
		\mathcal G^D \langle T^{P^A}(X,Y),Z\rangle =&\frac{1}{2}\int_M g(\vphi, X)g(\delta v,\delta w)\vol-\frac{1}{2}\int_M g(\vphi, Y)g(\delta u,\delta w)\vol\\
		&-\frac{1}{4}\int_M    g(\i_{\delta v}\j_\vphi X, \delta w)\vol+\frac{1}{4}\int_M  g(\i_{\delta u}\j_\vphi Y, \delta w)\vol.
	\end{align*}
	We rewrite each parts in the right side of the formula above. 
	
	Since $X$ is $d$-exact, the first part is
	\begin{align*}
		I=&	\frac{1}{2}\int_M g(\vphi, X)g(\delta v,\delta w)\vol=\frac{1}{2}\int_M g(g(\delta v,\delta w)\vphi,X)\vol\\
		=&\frac{1}{2}\int_M g(\pi_d[g(\delta v,\delta w)\vphi],X)\vol.
	\end{align*}
	Using $\pi_d=Gd\delta$ in \lemref{Hodge2} and the definition of the Dirichlet metric, we rewrite it as
	\begin{align*}
		I=\frac{1}{2}\int_M g(Gd\delta[g(\delta v,\delta w)\vphi],X)\vol
		=\frac{1}{2}\mathcal G^D\langle d\delta[g(\delta v,\delta w)\vphi] ,X \rangle .
	\end{align*}

	The second part follows similarly.
	\begin{align*}
		 II =&-\frac{1}{2}\int_Mg(\delta u,g(\vphi, Y)\delta w)\vol =-\frac{1}{2}\int_Mg( u,d[g(\vphi, Y)\delta w])\vol\\
		 =&-\frac{1}{2}\mathcal G^D\langle d[g(\vphi, Y)\delta w] ,X \rangle.
	\end{align*}
	
	Using the property of the operator $\j$ given in \lemref{lem:ex_j} with $p=2$ and $q=3$, we have that
	$$ III=-\frac{1}{4}\int_M  g(\i_{\delta v}\j_\vphi X, \delta w)\vol=-\frac{1}{2}\int_M  g(X, \i_\vphi \j_{\delta v}\delta w)\vol.$$
	Similar to the argument of the first part, we obtain
	$$III=-\frac{1}{2}\mathcal G^D\langle d\delta(\i_\vphi \j_{\delta v}\delta w) ,X \rangle.$$
	
	Finally, for the last part, \lemref{lem:adjoint_i_operator} says
	\begin{align*}
		IV=&\frac{1}{4}\int_M g(\i_{\delta u}\j_\vphi Y, \delta w)\vol=\frac{1}{4}\int_M g(\i_{\delta w}\j_\vphi Y, \delta u)\vol \\
		=&\frac{1}{4}\int_M g(d(\i_{\delta w}\j_\vphi Y),  u)\vol.
	\end{align*}
	Thus we conclude from the definition of the Dirichlet metric that
	$$IV=\frac{1}{4}\mathcal G^D\langle d(\i_{\delta w}\j_\vphi Y) ,X \rangle.$$
	Adding these four parts together, we finish the proof. 
\end{proof}

A direct corollary of \lemref{lem:S_operator} is 
\begin{prop}\label{porp:K}
	The contorsion tensor is given as
	$$K(X, Y)=\frac{1}{2}[T^{P^A}(X,Y)+S(X,Y)+S(Y,X)].$$
\end{prop}
\begin{proof}
	Exchanging $X, \ Y,\ Z$ in \lemref{lem:S_operator}, we get
	\begin{align*}
		\mathcal G^D\langle T^{P^A}(Y,Z), X\rangle=-	\mathcal G^D\langle T^{P^A}((Z,Y), X\rangle=- \mathcal G^D\langle S(Y, X), Z\rangle
		\end{align*}
	and
	\begin{align*}
	\mathcal G^D\langle T^{P^A}(Z,X), Y\rangle=\mathcal G^D\langle S(X,Y), Z\rangle.
	\end{align*}
By putting the above two identities into the right side of formula in \defnref{defn:K}, we have that
\begin{align*}
	2\mathcal G^D\langle K(X,Y), Z\rangle =&\mathcal G^D\langle T^{P^A}(X,Y), Z\rangle+ \mathcal G^D\langle S(Y, X), Z\rangle
	+\mathcal G^D\langle S(X,Y), Z\rangle\\
	=&\mathcal G^D\left\langle T^{P^A}(X,Y)+S(X,Y)+S(Y, X), Z\right\rangle. 
\end{align*}
We thus conclude this proposition from the formula above.
\end{proof}
Clearly, the skew-symmetric part of the tensor $K(X,Y)$ is the torsion part $\frac{1}{2}T^{P^A}(X,Y)$, and $\frac{1}{2}[S(X,Y)+S(Y,X)]$ is the symmetric part. We compute that
\begin{align}\label{eq:skewsym_K}
	K(X,Y)-K(Y,X)=T^{P^A}(X,Y).
\end{align}
Now we define a new connection based on $D^{P^A}$ and the contorsion tensor.
\begin{defn}\label{defn:DDcon}
	We define a connection $D^D$ on $\mathcal M$ as
$$D^D_t Y=D^{P^A}_t Y-K(\vphi_t,Y). $$ 
Recalling the definition of $D^P$ connection given in \defnref{defn:con}, we write the $D^D$ connection as 
$$D^D_tY=Y_t+P^D(\vphi, \vphi_t, Y)$$
with
$$P^D(\vphi, \vphi_t, Y)=P^A(\vphi, \vphi_t, Y)-K(\vphi_t,Y).$$
\end{defn}
\begin{prop}\label{prop:DDexplicit}
The explicit formulas of $K$ and $P^D$ are given below as
	\begin{align*}
	K(X,Y)=&-\frac{1}{4}d\bigg[-\i_{\delta u}\j_\vphi Y+2g(\vphi,Y)\delta u\bigg]+\frac{1}{2}d\delta\bigg[g(\delta u,\delta v)\vphi-\i_\vphi \j_{\delta u}\delta v\bigg],
\end{align*} 
and 
\begin{align*}
P^D(\vphi,\vphi_t,Y)=&\frac{1}{4}d\bigg[2g(\vphi,X)\delta v+2g(\vphi,Y)\delta u-\i_{\delta v}\j_\vphi \vphi_t -\i_{\delta u}\j_\vphi Y\bigg]\\
&-\frac{1}{2}d\delta\bigg[g(\delta u,\delta v)\vphi-\i_\vphi \j_{\delta u}\delta v\bigg].
\end{align*}
\end{prop}	
\begin{proof}
The proof follows by substituting the formula of the $D^{P^A}$ connection in \eqref{eq:TPAI} and $S$ operator in \defnref{defn:S} into the expression of $K$ operator in \propref{porp:K}. We have
	\begin{align*}
	K(X, Y)=&\frac{1}{2}[T^{P^A}(X,Y)+S(X,Y)+S(Y,X)]\\
	=&\frac{1}{4}d\left[g(\vphi,X)\delta v- g(\vphi,Y)\delta u\right]-\frac{1}{8} d\left[\i_{\delta v}\j_\vphi X-\i_{\delta u}\j_\vphi Y  \right] \\
&	+\frac{1}{8}d[\i_{\delta v}\j_\vphi X-2g(X, \vphi)\delta v]+\frac{1}{4}d\delta[g(\delta u,\delta v)\vphi-\i_\vphi \j_{\delta u}\delta v]\\
&	+\frac{1}{8}d[\i_{\delta u}\j_\vphi Y-2g(Y, \vphi)\delta u]+\frac{1}{4}d\delta[g(\delta v,\delta u)\vphi-\i_\vphi \j_{\delta v}\delta u]
\\
=&	-\frac{1}{4}d[-\i_{\delta u}\j_\vphi Y+2g(\vphi,Y)\delta u]+\frac{1}{2}d\delta[g(\delta u,\delta v)\vphi-\i_\vphi \j_{\delta u}\delta v].
	\end{align*}
Note that in  the last step of the above computation we use $\j_{\delta u}\delta v= \j_{\delta v}\delta u$, due to the symmetric property of $\j$ by \defnref{def:operator_j general}.

 Now using the definition of the $D^{P^A}$ connection given by \eqref{eq:con2_I} in \propref{prop:DP_explicit},
 we get
 	\begin{align*}
 P^D(\vphi, \vphi_t, Y)=&P^A(\vphi, \vphi_t, Y)-K(\vphi_t,Y)\\ =&\frac{1}{2}d\left[g(\vphi,X)\delta v-\frac{1}{2}\i_{\delta v}\j_\vphi \vphi_t \right]	+\frac{1}{4}d[-\i_{\delta u}\j_\vphi Y+2g(\vphi,Y)\delta u]\\
 	&-\frac{1}{2}d\delta[g(\delta u,\delta v)\vphi-\i_\vphi \j_{\delta u}\delta v]\\
 =&\frac{1}{4}d\left[2g(\vphi,X)\delta v+2g(\vphi,Y)\delta u-\i_{\delta v}\j_\vphi \vphi_t -\i_{\delta u}\j_\vphi Y\right]\\
 &-\frac{1}{2}d\delta[g(\delta u,\delta v)\vphi-\i_\vphi \j_{\delta u}\delta v],
 		\end{align*}
 	which completes the proof.
	\end{proof}
\begin{prop}\label{prop:DD-Levicivita}
	The $D^D$ connection is a Levi-Civita connection of the $\mathcal G^D$ metric.
\end{prop}
\begin{proof} 
First, the trosion freeness follows from the use of contorsion tensor.
\begin{align*}
 D^D_tY- D^D_sX =&D^{P^A}_tY-D^{P^A}_sX+D^D_t Y-D^{P^A}_tY-D^D_s X+D^{P^A}_s X \\
	=& T(X,Y)-K(X,Y)+K(Y,X)=T(X,Y)-T(X,Y)=0.
\end{align*}
For compatibility of the $D^D$ connection, since $D^{P^A}$ is compatible with the $\mathcal G^D$ metric, we have
 \begin{align*}
 D_t \mathcal G^D\langle Y, Z\rangle =&\mathcal G \langle D_t^{P^A}Y, Z\rangle
+\mathcal G^D \langle Y, D_t^{P^A}Z\rangle.
\end{align*}
Rewriting 
$$D_t^{P^A}Y= D_t^{P^A}Y-D_t^DY+D_t^D Y=K(X,Y)+D_t^D Y,$$
and similarly
$$D_t^{P^A}Z=K(X,Z)+D_t^D Z,$$
we have
\begin{align*}
	 D_t \mathcal G^D\langle Y, Z\rangle 
=&\mathcal G^D \langle K(X, Y), Z\rangle +\mathcal G^D \langle Y, K(X,Z)\rangle 
+\mathcal G^D \langle D_t^D Y, Z\rangle\\
& +\mathcal G^D \langle Y, D_t^D Z \rangle.
\end{align*}
Recalling \propref{porp:K}, we get
 \begin{align*}
&	2\mathcal G^D \langle K(X, Y), Z\rangle +2\mathcal G^D \langle Y, K(X,Z)\rangle \\
	=&\mathcal G^D \langle T^{P^A}(X,Y)+S(X,Y)+S(Y,X), Z\rangle \\
	&+\mathcal G^D \langle Y, T^{P^A}(X,Z)+S(X,Z)+S(Z,X)\rangle.
\end{align*}
Furthermore we see from \propref{defn:S}
 that
 \begin{align*}
	&2\mathcal G^D \langle K(X, Y), Z\rangle +2\mathcal G^D \langle Y, K(X,Z)\rangle \\
	= & \mathcal G^D \langle T^{P^A}(X, Y), Z\rangle -\mathcal G^D \langle T^{P^A}(Y,Z), X \rangle +\mathcal G^D \langle T^{P^A}(Z,X), Y\rangle \\ &+\mathcal G^D \langle T^{P^A}(X,Z), Y\rangle -\mathcal G^D \langle T^{P^A}(Z,Y), X\rangle +\mathcal G^D \langle T^{P^A}(Y,X), Z\rangle=0.
\end{align*}
As a result, we have
 \begin{align*}
	D_t \mathcal G^D\langle Y, Z\rangle
	=&\mathcal G^D \langle D_t^D Y, Z\rangle +\mathcal G^D \langle Y, D_t^D Z \rangle,
\end{align*}
which means $D^D$ is compatible with the Dirichlet metric $\mathcal G^D$.
\end{proof}
\begin{prop}\label{prop:DDgeo}
 	Let $\vphi(t)$ be a geodesic of the $D^D$ geodesic on $\mathcal{M}$ and $u=G\vphi_t$ be the potential form of $\vphi_t$. Then we have the geodesic equation
	\begin{align}\label{Geodesic equation Dirichlet}
		\vphi_{tt}=&-d[g(\vphi_t,\vphi)\delta G\vphi_t]+\frac{1}{2}d[\i_{\delta G\vphi_t}\j_\vphi \vphi_t] \notag \\
		&+\frac{1}{2}d\delta[g(\delta G\vphi_t, \delta G\vphi_t)\vphi-\i_\vphi \j_{\delta G\vphi_t}(\delta G\vphi_t)].	
	\end{align}
\end{prop}
\begin{proof}
 Recalling the definition of connection defined in \defnref{defn:con}, we have $D^D_t \vphi_t=0$. Using the explicit fomula of the $D^D$ connection given in \propref{prop:DDexplicit} with $X=Y=\vphi_t$ and $u=v=G\vphi_t$ in this case, we get 
	\begin{align*}
		D^D_t Y	=&Y_t+\frac{1}{4}d\left[2g(\vphi,\vphi_t)\delta G\vphi_t+2g(\vphi,\vphi_t)\delta G\vphi_t-\i_{\delta G\vphi_t}\j_\vphi \vphi_t -\i_{\delta G\vphi_t}\j_\vphi \vphi_t\right]\\
		&-\frac{1}{2}d\delta[g(\delta G\vphi_t,\delta G\vphi_t) \vphi-\i_\vphi \j_{\delta G\vphi_t}(\delta G\vphi_t)]\\
		=&\vphi_{tt}+d[g(\vphi,\vphi_t)\delta G\vphi_t]-\frac{1}{2}d\left[\i_{\delta G\vphi_t}\j_\vphi \vphi_t\right]\\
		&-\frac{1}{2}d\delta[g(\delta G\vphi_t,\delta G\vphi_t) \vphi-\i_\vphi \j_{\delta G\vphi_t}(\delta G\vphi_t)],
	\end{align*}
which completes the proof.
\end{proof}

%%%%%%%%%%%%%%%%%%%%%%%%%%%%%%%%%%%%%%%%%%%%%%%%%%

 %%%%%%%%%%%%%%%%%%%%%%%%%%%%%%%%%%%%%%%%%%%%%%%%%%%%%%%%%%%%%%%%%%%%%%%%%%%%%%%%%%%%%%%%%%%%%%%%%%%%%%%%%%%%%%%%%%%%%%%%

%%%%%%%%%%%%%%%%%%%%%%%%%%%%%%%%%%%%%%%%%%%%%%%%%%%%%%%%%%%%%%%%%%%%%%%%%%%%%%%%%%%%%%%%%%%%%%%%%%%%%%%%%%%%%%%%%%%%%%%%

%%%%%%%%%%%%%%%%%%%%%%%%%%%%%%%%%%%%%%%%%%%%%%%%%%%%%%%%%%%%%%%%%%%%%%%%%%%%%%%%%%%%%%%%%%%%%%%%%%%%%%%%%%%%%%%%%%%%%%%%

\section{Sobolev-type metrics and variational structures} \label{L2metric}
In this section, we will define two more metrics on space $\mathcal M$ and study their geometry properties. The motivation of these definitions originates from the study of Sobolev-type metrics in K\"ahler geometry \cite{MR3576284,MR2915541,MR3412344,MR3035976,MR1107281,MR1027070}.
We will find Levi-Civita connections for these metrics and compute the geodesic equations. Besides, we will define three energy functionals based on these metrics and study their variational structures.

\subsection{Laplacian metric $\mathcal G^L$- and $L^2$-metric $\mathcal G^M$}\label{sec:5.1}

Over $\mathcal M$, we can define a Laplacian metric $\mathcal G^L$ and a $L^2$-metric $\mathcal G^M$, which are analogues of the Calabi metric and the Mabuchi metric in K\"ahler geometry. See \cite{MR2915541}.

	Let $X,Y$ be any two vectors on $T_\vphi\mathcal M$ and let $u=GX,\ v=GY$ be the corrresponding potential forms given by \defnref{potential form}. 
 \begin{defn} [$\mathcal G^L$ metric and $\mathcal G^M$ metric]\label{defn:L2}
  We define 
\begin{align*}
 &\text{the Laplacian metric}: \mathcal G^L_{\vphi}(X,Y)=\int_M g_\vphi(X,Y)\vol_\vphi=\int_M X\wedge *_{\vphi}Y ,\\
&\text{the $L^2$ metric}: \mathcal G^M_\vphi(X,Y)=\int_M g_{\vphi}(u,v) \vol_{\vphi}=\int_M g_{\vphi}(G^2 X,Y) \vol_{\vphi}.
 \end{align*}
 
  Compared with \eqref{eq:def_metric_A} in \secref{sec:3.3}, the Laplacian metric and the $L^2$ metric correspond to the cases when $A_\vphi =id$ and $A_\vphi =G^2(=G_\vphi^2)$ respectively. We list these metrics in terms of the vector $X$, the 2-form $\a=\delta u$ and the potential $u=GX$ as shown in the next table.
  \begin{center}
  	\begin{tabular}{|c|c|c|c|}
  		\hline
  		\rule{0pt}{8pt} 
  		%\multirow{2}*{Definitions}&
  	%	\multicolumn{3}{c|}{Metrics} \\ \cline{2-4}
  		\rule{0pt}{13pt}& Laplacian metric $\mathcal G^L$ & Dirichlet metric $\mathcal G^D$ & $L^2$ metric $\mathcal G^M$ \\ \hline
  			\rule{0pt}{10pt}  $A_\vphi$ & id & $G$ & $G^2$ \\ \hline
  			\rule{0pt}{10pt} $X,\ Y$ 
  			&$ \int_M g(X, Y)\vol$ 
  			& $ \int_M g(X, GY)\vol$ 
  			&$ \int_M g(GX, GY)\vol$ \\ \hline
  		\rule{0pt}{10pt} $2$-forms
  		& $ \int_M g(d\a,  d\b)\vol$ 
  		& $ \int_M g(\a, \b)\vol$  
  		&$ \it_M g(G\a,  \b)\vol$  \\ \hline
  			\rule{0pt}{10pt} Potentials
  			&$\int_M g(\Delta u, \Delta v)\vol$ 
  			& $\int_M g(u, \Delta v)\vol$  
  			&$\int_M g(u,  v)\vol$ \\ \hline
  	\end{tabular}
  \end{center}
 \end{defn}
In order to study geometry of the space $\mathcal M $, we propose the following questions.
\begin{ques}["Optimal" connection and metric]
	Could we find a connection $D$ which is compatible with the Dirichlet metric $\mathcal G^D$ and satisfies one or more properties as following
	\begin{itemize}
		\item the volume functional is concave along the $D$ geodesics;
		\item the curvature regarding to the connection $D$ has a definite sign.
	\end{itemize}
What happens for the Laplacian metric and the $L^2$ metric (\defnref{defn:L2}) ?
\end{ques}
\begin{rem}
The convexity of energy functionals is one of the important problems in K\"ahler geometry, see \cite{MR3496771,MR4020314} for a complete references.
\end{rem}
%We have discussed the concavity  of the volume functional along compatible geodesics in \secref{section:uniqueness}. The computation of curvatures will appear in the sequel of this paper.
%
% With the potential forms $u,v$, we could interpret the Laplacian metric as
% \begin{align*}
%\mathcal G^L_{\vphi}(X,Y)=\int_M g_\vphi(X,Y)\vol_\vphi= \int_M g_\vphi(\Delta_\vphi u,\Delta_\vphi v)\vol_\vphi.
% \end{align*}

\begin{prop}\label{1st volume Laplacian metric}The first variation of the volume functional with respect to the metrics defined in \defnref{defn:L2} is given by
\begin{align*}
D_\vphi \Vol(X)=\frac{7}{3}\mathcal G^L(\pi_d(\vphi),X) =\frac{7}{3}{\mathcal G}_\vphi^M(X, \Delta^2 \vphi).
\end{align*}
Here $\pi_d$ is the projection map defined in \lemref{Hodge2} regarding to the $G_2$ metric $g=g_{\vphi}$.
\end{prop}
\begin{proof}
As we have already calculated in \eqref{gradient vol}, the variation of volume form towards direction $X$ at point $\vphi$ is
\begin{align*}
D_\vphi \Vol(X)= \frac{7}{3}\int_Mg(\vphi,X)\vol_{\vphi} =\frac{7}{3}\mathcal G^D( \Delta\vphi, X).
\end{align*}
By \defnref{defn:L2}, the right part is equal to
\begin{align*}
\frac{7}{3}\mathcal G^L(\pi_d(\vphi),X)=\frac{7}{3}{\mathcal G}^M( \Delta^2 \vphi,X).
\end{align*} Thus, we have obtained the conclusion.
\end{proof}
%The critical point of the volume functional satisfies the following equation
%\begin{align*}
%\pi^\vphi_d(\vphi)=0,
%\end{align*}
%which is equivalent to the torsion free equation
%\begin{align*}
%\delta \vphi=0,
%\end{align*}
%under the assumption that 3-form $\vphi$ is closed.
We then have two direct corollaries of \propref{1st volume Laplacian metric}.

\begin{cor}
The Euler-Lagrange equations of the volume functional are
\begin{align*}
\Delta\vphi=\pi_{d}\vphi=\Delta^2 \vphi=0
\end{align*}
under the metrics $\mathcal G^D$, $\mathcal G^L$, $\mathcal G^M$ respectively.
\end{cor}
\begin{cor}[$\mathcal G^L$ and $\mathcal G^M$ gradient flows]
The gradient flow of the volume functional is
\begin{align}
\frac{\p}{\p t} \vphi(t)&=\pi_{d}\vphi(t) \text{  under the Laplacian metric},\label{GL flow}\\
	\frac{\p}{\p t}\vphi(t)&=\Delta^2\vphi(t) \text{ under the $L^2$ metric}.\label{bilaplacian flow}
\end{align}
\end{cor}

Let $\vphi(t)$ be a curve on $\mathcal{M}$ and $Y(t), \ Z(t)$ be two smooth vector fields along $\vphi(t)$. Also, let $u=G\vphi_t,\ v=G Y, \ w=GZ$ be the potential forms. We now introduce Levi-Civita connections $D^L$ and $D^M$ for the Laplacian metric and the $L^2$ metric, respectively.
\begin{defn}\label{defn:L_con}
The $D^L$ and $D^M$ connections are defined below.
\begin{align*}
	D^L_t Y(t)=&Y_t+P^L(\vphi,\vphi_t,Y), \ \ 	D^M _t Y(t)=Y_t+P^M(\vphi,\vphi_t,Y),
\end{align*}
where $P^L$ and $P^M$ are given by
\begin{align*}
	P^L(\vphi,\vphi_t,Y)=&\frac{2}{3}\pi_d \left[g(\vphi_t,\vphi)Y +g(\vphi,Y)\vphi_t-g(\vphi_t,Y)\vphi\right] \\
	&-\frac{1}{4} \pi_d \left[\i_Y\j_\vphi \vphi_t+\i_{\vphi_t}\j_\vphi Y-\i_\vphi \j_{\vphi_t} Y\right],\\
	P^M(\vphi,\vphi_t,Y)=&\frac{1}{2}d\delta d[g(\vphi,\vphi_t)\delta Gv+g(\vphi, Y)\delta G u]-\frac{1}{4}d\delta d[\i_{\delta Gv}\j_\vphi\vphi_t+\i_{\delta Gu}\j_\vphi Y]\\
	&+\frac{1}{2}d[g(\vphi,\vphi_t)\delta v+g(\vphi, Y)\delta u]-\frac{1}{4}d(\i_{\delta v}\j_\vphi\vphi_t+\i_{\delta u}\j_\vphi Y) \\
	&-\frac{2}{3}d\delta[g(\vphi,\vphi_t)v+g(\vphi,Y)u]+\frac{1}{4}d\delta(\i_v\j_\vphi\vphi_t+\i_u\j_\vphi Y)\\
	&-\frac{1}{2}d\delta d\delta[2g(\delta u,\delta Gv)\vphi+2g(\delta v,\delta Gu)\vphi-\i_\vphi \j_{\delta u}(\delta Gv)-\i_\vphi \j_{\delta v}(\delta Gu)]\\
	&+d\delta d\delta[\frac{2}{3}g(u,v)\vphi-\frac{1}{4} d\delta d\delta(\i_\vphi \j_u v)] .
\end{align*}
\end{defn}

\begin{thm}[$D^L$ connection and $D^M$ connection]\label{prop:L_con}
The $D^L$ and $D^M$ connections are Levi-Civita connections of the $\mathcal G^L$ and $\mathcal G^M$ metrics respectively.
\end{thm}
\begin{proof}
	The proof will be given in \secref{sec:proof1}.
\end{proof} 

Now we give the geodesic equations for the $D^L$ and $D^M$ connections.
   \begin{prop}[The $D^L$ and $D^M$ geodesics]\label{prop:geodesic}
   	Let $\vphi(t)$ be a smooth curve on $\mathcal M$ and denote the potential of $\vphi_t$ by $u=G\vphi_t$.
   	\begin{itemize}
   		\item The $D^L$ geodesic equation is
   		\begin{equation}\label{Geodesic equation Laplacian}
   		\begin{aligned}
   			\vphi_{tt}=&-P^L(\vphi,\vphi_t,\vphi_t)  \\
   			=&-\frac{2}{3}\pi_d \left[2g(\vphi_t,\vphi)\vphi_t -g(\vphi_t,\vphi_t)\vphi\right] +\frac{1}{4} \pi_d \left[2\i_{\vphi_t}\j_\vphi \vphi_t-\i_\vphi \j_{\vphi_t} \vphi_t \right].
   		\end{aligned}
   	\end{equation}
   	% In the local coordinates, it reads
   	%	\begin{align}\label{Geodesic equation Laplacian}
%   			\vphi_{tt}=&-\frac{2}{3}\pi_d \left[2g(\vphi_t,\vphi)\vphi_t -g(\vphi_t,\vphi_t)\vphi\right] +\frac{1}{4} \pi_d \left[(\vphi_t)_{ijk}\vphi^{ajk}(\vphi_t)_{abc} \right.\notag \\
%   			&+(\vphi)_{ijk}(\vphi_t)^{ajk}(\vphi_t)_{abc} 
%   			\left. -(\vphi_t)_{ijk}(\vphi_t)^{ajk}\vphi_{abc}  \right]dx^{ibc}.
%   		\end{align}
   	\item  
   	The $D^M$ geodesic equation is
   	\begin{equation}\label{Geodesic equation L2}
   	\begin{aligned}
   		\vphi_{tt}=&-P^M(\vphi,\vphi_t,\vphi_t)\\
   		=&-d\delta d \left[ g(\vphi,\vphi_t)\delta Gu+\frac{1}{2} \i_{\delta Gu}\j_\vphi\vphi_t \right] -d\left[g(\vphi,\vphi_t)\delta u-\frac{1}{2}\i_{\delta u}\j_\vphi\vphi_t \right] \\
   		&-d\delta d\delta\left[g(\delta u,\delta Gu)\vphi-\frac{1}{2}\i_\vphi \j_{\delta u}(\delta Gu)
   		-\frac{2}{3}g(u,u)\vphi+\frac{1}{4}\i_\vphi \j_u u\right]  \\
   	&+d\delta\left[\frac{4}{3}g(\vphi,\vphi_t)u-\frac{1}{2}\i_u\j_\vphi\vphi_t\right].
   	\end{aligned}
   \end{equation}
   	\end{itemize}
      \end{prop}
  \begin{proof}
  	Recalling the geodesic equation given by \defnref{defn:DPgeo}, the $D^L$ geodesic equation is
  	$$0=D_t^L\vphi_t=\vphi_{tt}+P^L(\vphi,\vphi_t,\vphi_t)=0,$$
  and the $D^M$ geodesic equation is
  	$$0=D_t^M\vphi_t=\vphi_{tt}+P^M(\vphi,\vphi_t,\vphi_t)=0.$$
  The proof follows by substituting the formulas of $P^L$ and $P^M$ in \defnref{defn:L_con} to the above identities.	
  \end{proof}

A nature question is: 
\begin{ques}[Geodesics]
	Could we solve the Dirichlet problem or the Cauchy problem of the following geodesic equations we defined before? They are
	\begin{itemize}   
			\item the $D^L$ geodesic \eqref{Geodesic equation Laplacian} for the Laplacian metric;
		\item  the $D^D$ geodesic \eqref{Geodesic equation Dirichlet} for the Dirichlet metric;     
		\item the $D^M$ geodesic \eqref{Geodesic equation L2} for the $L^2$ metric.
	\end{itemize}
\end{ques}

%%%%
\subsection{Energy functionals}
We have found in this article that there are several geometric flows, which could serve as candidates for searching the $G_2$ manifolds. These flows occur as gradient flows of volume functional under different metrics on $\mathcal M$. Other than the volume functional, we introduce three associated energy functionals, the $\mathcal E^D$ energy functional, the $\mathcal E^L$ energy functional and the $\mathcal E^M$ energy functional, motivated from the Calabi energy in K\"ahler geometry. These functionals are the $L_2$-norms of critical equations of the volume functional under the Dirichlet metric, the $L^2$ metric and the Laplacian metric respectively.
\begin{defn}\label{E energies}
	We define the $\mathcal E^L,\ \mathcal E^D$ and the $\mathcal E^M$ energy functionals as
		\begin{align*}
				\mathcal E^L(\vphi)=&\mathcal G^L(\pi_d\vphi,\pi_d \vphi)=\int_M |\pi_{d}\vphi|^2_\vphi\vol;\\
		\mathcal E^D(\vphi)=&\mathcal G^D(\Delta \vphi,\Delta \vphi)=\int_M |\delta \vphi|^2_\vphi\vol=\int_M \tau\wedge * \tau ;\\
	\mathcal E^M(\vphi)=&\mathcal G^M(\Delta^2 \vphi,\Delta^2 \vphi)=\int_M |\Delta\vphi|^2_\vphi\vol=\int_M d\tau\wedge * d\tau  .
\end{align*}
The $\mathcal E^D$ functional is called the Dirichlet energy in existing literature \cites{MR2980500,MR2995206}.
\end{defn}
The next theorem gives the first order variation of these functionals. The proof will be  given in \secref{sec:pf2}.

\begin{thm}\label{prop:Var_Dirichlet energy}
	Let $\vphi(t)$ be a family of closed $G_2$-structures in $\mathcal{M}$. The first order derivatives of the energy functionals defined above reads 
\begin{align*}
\frac{d}{dt}\mathcal{E}^a(\vphi(t))
=\int_Mg_{\vphi(t)}\left(\frac{\p}{\p t}\vphi(t),F^a(\vphi(t))\right)\vol_{g(t)}, \quad a=L,D,M.
\end{align*}
	The $F^a$($a=L,D,M$) operator is
		\begin{align}
			F^L(\vphi)
			&=-\frac{4}{3}|\pi_d \vphi|^2_\vphi\vphi+\frac{1}{2}\i_\vphi\j_{\pi_d \vphi}(\pi_d \vphi)+[\frac{8}{3}\pi_1^3+2\pi_7^3-2\pi_{27}^3]\pi_d\vphi;\label{PL} \\
			F^D(\vphi)&=-2\Delta_\vphi \vphi-\frac{1}{3}|\tau|^2_{\vphi}\vphi +\i_\vphi \j_\tau\tau ;	\label{PD}\\
		F^M(\vphi)&=-2g(\tau,\delta d \tau)\vphi+\frac{2}{3}g(\vphi,d\delta d\tau)\vphi+2\i_\vphi\j_\tau(\delta d\tau)-2d\delta d\tau\notag\\ &+*[*(d\delta d\tau\wedge \vphi)\wedge \vphi] +\frac{4}{3}|d\tau|_\vphi^2\vphi-\frac{1}{2}\i_\vphi\j_{d\tau}(d\tau).\label{PM}
	\end{align}

\end{thm}
Their gradient flows can be read directly from \thmref{prop:Var_Dirichlet energy}. 
\begin{defn}\label{decreasing flows}
These gradient flows are listed in the following chart. We use the notations
\begin{align*}
 F^L_d= \pi_d F^L,\ F^D_d=\pi_d F^D, \quad F^M_d=\pi_d F^M.
\end{align*}
\begin{center}
\begin{tabular}{|c|c|c|c|}
\hline
\rule{0pt}{10pt} \multirow{2}*{Functionals}&
\multicolumn{3}{c|}{Metrics} \\ \cline{2-4}
\rule{0pt}{13pt} &$\mathcal G^L=\int_M g(\cdot,\cdot) \vol$ &$\mathcal G^D=\int_M g (G \cdot,\cdot) \vol $  & $\mathcal G^M=\int_M g ( G^2\cdot,\cdot) \vol$ \\ \hline

\rule{0pt}{10pt} $Vol(\vphi)$ &   $\vphi_t=\pi_d^\vphi \vphi$ &$\vphi_t=\Delta\vphi$  & $\vphi_t=\Delta^2\vphi$\\ \hline

\rule{0pt}{10pt} $\mathcal E^L(\vphi)$ & $\vphi_t=-F_d^L(\vphi)$ & $\vphi_t=-\Delta F_d^L(\vphi)$ & $\vphi_t=-\Delta^2 F_d^L(\vphi)$ \\ \hline

\rule{0pt}{10pt} $\mathcal E^D(\vphi)$ & $\vphi_t=-F_d^D(\vphi)$  & $\vphi_t=-\Delta F_d^D(\vphi)$ & $\vphi_t=-\Delta^2  F_d^D(\vphi)$ \\ \hline

\rule{0pt}{10pt} $\mathcal E^M(\vphi)$ & $\vphi_t=-F_d^M(\vphi)$ & $\vphi_t=-\Delta F_d^M(\vphi)$ & $\vphi_t=-\Delta^2 F_d^M(\vphi)$ \\ \hline
\end{tabular}
\end{center}

Especially, we focus on the following geometric flow equation
\begin{align}\label{eq:Drichletflow}
	\frac{\p}{\p t}\vphi(t)=-F_d^D(\vphi(t)),
	\end{align}
which is called the Dirichlet flow.
\end{defn}

\begin{rem}
	Without closedness condition, the Dirichlet energy functional $\mathcal E^D(\vphi)$ of general $G_2$-structures $\vphi$ and the corresponding gradient flow are studied in \cites{MR2980500,MR2995206}. However, our Dirichlet flow \eqref{eq:Drichletflow} preserves the closedness condition, which is different from theirs.
\end{rem}

It would be interesting to consider the short time existence of the flows listed in the table. It is known for the Laplacian flow \cite{arXiv:1101.2004} and the Dirichlet flow \cite{MR2980500}. It would be more challengeable to consider the optimal conditions to guarantee their long time existence. Convergence of these flows leads to the existence of torsion-free $G_2$-structures. We may consider the convergence problem, assuming the existence of torsion-free metrics and comparing with the results of the Calabi flow \cite{MR3010550,MR3833786,MR2811044}.

\begin{ques}[Distance]
	Could we find a metric $\mathcal G$ on $\mathcal M$ such that the distance measured by $\mathcal G$ between two Laplacian flows is decreasing? What are for the Dirichlet flow?
\end{ques}

%%%%%%%%%%%%%%%%%%%%%%%%%%%%%%%%%%%%%%%%%%%%%

\subsection{Proof of \thmref{prop:L_con}}\label{sec:proof1}
The proof is divided into three parts. First, we prove the torsion-freeness for the two connections. Then we check that the $D^L,\ D^M$ connections are compatible with respect to the $\mathcal G^L,\ \mathcal G^M$ metrics respectively.
\begin{proof}[Proof of the Torsion-freeness]
	First, we observe that the $P^L$ and $P^M$ operators in the two cases are symmetric with respect to the later two variables, i.e. 
	$$P^L(\vphi, X, Y)=P^L(\vphi, Y,X), \ \ \ P^M(\vphi, X, Y)=P^M(\vphi, Y,X).$$
	Then \propref{lem:tf-cond} implies that the $D^L$ and $D^M$ connections are torsion-free.
\end{proof}
\begin{proof}[Proof of the compatibility of the $D^L$ connection]	
	Next, we prove that the $D^L$ connection is compatible with the $\mathcal G^L$ metric.
	We apply \thmref{thm:comp_General} to the $\mathcal G^L$ metric and $D^L$ connection. Under this situation, $A_{\vphi} =id$ due to \defnref{defn:L2}, and 
	\begin{align*}
		P(\vphi, \vphi_t ,Y)=P^L(\vphi,\vphi_t,Y).
	\end{align*}
	Now it remains to check the condition \eqref{eq:comp_III} in \thmref{thm:comp_General} holds. Since $A_\vphi=id$,  $A_t=0$ and hence the right side of \eqref{eq:comp_III} becomes
	\begin{align*}
		Right=&\frac{1}{2}\left[\int_M g(Z, \pi_d(**_tY))+g(Y,\pi_d(**_t Z))\right] \vol\\
		=&\frac{1}{2}\int_M [g(Z, **_tY)+g(Y,**_t Z)]\vol.
	\end{align*}
Inserting the formula of $**_t$ form \corref{lem:var_*_general}, we get
	\begin{align*}
		Right=&\frac{1}{2}\int_M \left[ g\left( Z, \frac{4}{3}g(\vphi,\vphi_t)Y-\frac{1}{2}\i_Y\j_\vphi\vphi_t\right) +g\left( Y,\frac{4}{3}g(\vphi,\vphi_t)Z-\frac{1}{2}\i_Z\j_\vphi\vphi_t\right)\right]  \vol\\
		=&\frac{1}{2}\int_M \left\{ \frac{8}{3}g(\vphi,\vphi_t)g(Z, Y)-\frac{1}{2}[g(Z,\i_Y\j_\vphi\vphi_t)+g(Y,\i_Z\j_\vphi\vphi_t)]\right\} \vol.
	\end{align*}
	We further apply \lemref{lem:ex_j} to compute that
	\begin{align*}
		Right=&\int_M \left[ \frac{4}{3}g(\vphi,\vphi_t)g(Z, Y)-\frac{1}{2}g(\j_ZY,\j_\vphi\vphi_t)\right] \vol.
	\end{align*}

	Substituting $P^L(\vphi,Z,Y)$ in the left part, we get
	\begin{align*}
		&Left=\int_M \left[g(P^L(\vphi,\vphi_t, Z), Y)+g(P^L(\vphi,\vphi_t,Y),Z)\right] \vol \\
		=&\frac{2}{3}\int_M g\left(g(\vphi_t,\vphi)Z+g(\vphi,Z)\vphi_t-g(\vphi_t,Z)\vphi,Y\right)\vol\\
		&+\frac{2}{3}\int_M g\left(g(\vphi_t,\vphi)Y+g(\vphi,Y)\vphi_t-g(\vphi_t,Y)\vphi,Z\right)\vol\\
		&	-\frac{1}{4}\int_M \left[ g(\i_Z\j_\vphi\vphi_t+\i_{\vphi_t}\j_\vphi Z-\i_\vphi\j_{\vphi_t}Z, Y)+g(\i_Y\j_\vphi\vphi_t+\i_{\vphi_t}\j_\vphi Y-\i_\vphi\j_{\vphi_t}Y, Z)\right]\vol.
	\end{align*}
	Again, by using \lemref{lem:ex_j}, we have that
	\begin{align*}
		Left
		=&\frac{4}{3}\int_M g(\vphi,\vphi_t)g(Z,Y)\vol \\
		&	-\frac{1}{4}\int_M \left[ g(\j_ZY,\j_\vphi\vphi_t)+g(\j_{\vphi_t}Y,\j_\vphi Z)-g(\j_{\vphi_t}Z, \j_\vphi Y)+g(\j_\vphi\vphi_t,\j_ZY) \right. \\
		&+\left. g(\j_\vphi Y,\j_{\vphi_t}Z)-g(\j_\vphi Z, \j_{\vphi_t}Y)\right]\vol\\
		=&\frac{4}{3}\int_M g(\vphi,\vphi_t)g(Z,Y)-\frac{1}{2}\int_M g(\j_ZY,\j_\vphi\vphi_t)\vol=Right.
	\end{align*}
	Thus we see the $D^L$ connection is a Levi-Civita connection for the $\mathcal G^L$ metric, which is a direct corollary of \thmref{thm:comp_General}. 
\end{proof}
\begin{proof}[Proof of the compatibility of the $D^M$ connection]
	The compatible property of $D^M$ connection follows similarly by choosing $A_{\vphi} =G^2 $ and
	\begin{align*}
		P(\vphi, \vphi_t ,Y)=&P^M(\vphi,\vphi_t,Y).
	\end{align*}
The right hand side of \eqref{eq:comp_III} is then
\begin{align*}
	Right=&\frac{1}{2}\int_M g((G^2)_t Y, Z)\vol+\frac{1}{2}\int_Mg(Y, (G^2)_tZ)\vol \\
	& +\frac{1}{2}\int_M [g(G^2 Y, **_tZ)+g(G^2Z, **_tY)]\vol.
	\end{align*}
Using the variation of $G$ given in \lemref{prop:var_G}, we have
\begin{align*}
	(G^2)_tY=&G_tGY+GG_tY\\
	=&GGd[**_t(\delta v)]+Gd[**_t(\delta Gv)]-G\pi_d (**_t v)-\pi_d (**_t Gv).	
\end{align*}
Substituting it into the right part and integrating by parts, we see that the first term equals to
\begin{align*}
	&\int_M g((G^2)_t Y, Z)\vol\\
	=&\int_M g\left(GGd[**_t(\delta v)]+Gd[**_t(\delta Gv)]-G\pi_d (**_t v)-\pi_d (**_t Gv), Z\right)\vol\\
	=&\int_M g(**_t(\delta v), \delta Gw)\vol
	+\int_M g(**_t(\delta Gv),\delta w)\vol-\int_M g(**_tv,  w)\vol\\
	&-\int_M g(**_t Gv, Z)\vol.
\end{align*}
Exchanging $Y,\ Z$ in the above formula we immediately get
\begin{align*}
	\int_M g((G^2)_t Z, Y)\vol	=&\int_M g(**_t(\delta w), \delta Gv)\vol	+\int_M g(**_t(\delta Gw),\delta v)\vol
	\\&-\int_M g(**_tw,  v)\vol	-\int_M g(**_t Gw, Y)\vol.
\end{align*}
Recalling that \lemref{lem:adjoint_*_t} says $*_t$ is adjoint with respect to the inner product, we have
\begin{align*}
	Right=
	&\frac{1}{2}\left[\int_M g((G^2)_t Y, Z)\vol+	\int_M g((G^2)_t Z, Y)\vol\right]	\\
	&+\frac{1}{2}\int_M [g(Gv, **_tZ)+g(Gw, **_tY)]\vol\\
	=&\int_M g(**_t(\delta v), \delta Gw)\vol
	+\int_M g(**_t(\delta Gv),\delta w)\vol-\int_M g(**_tv,  w)\vol.
\end{align*}
Applying \corref{lem:var_*_general} and \lemref{lem:ex_j}, we have
\begin{align*}
	\int_M g(**_t(\delta v), \delta Gw)\vol=&\int_M g(g(\vphi,\vphi_t)\delta v-\frac{1}{2}\i_{\delta v}\j_\vphi\vphi_t, \delta Gw)\vol \\
	=&\int_M g(\vphi_t,\vphi)g(\delta v,\delta Gw)\vol-\int_M g(\j_\vphi\vphi_t, \j_{\delta v}\delta Gw)\vol.
\end{align*}
We deal with the other terms similarly to get
\begin{equation}\label{eq:left1}
\begin{aligned}
Right	=&\int_M g(\vphi_t,\vphi)g(\delta v,\delta Gw)\vol-\int_M g(\j_\vphi\vphi_t, \j_{\delta v}\delta Gw)\vol\\
	&+\int_M g(\vphi_t,\vphi)g(\delta w,\delta Gv)\vol-\int_M g(\j_\vphi\vphi_t, \j_{\delta w}\delta Gv)\vol\\
	&-\frac{4}{3}\int_M g(\vphi,\vphi_t)g(v,w)\vol+\frac{1}{2}\int_M g(\j_\vphi\vphi_t,\j_vw)\vol.
\end{aligned}
\end{equation}
Now we deal with the left part, which is
$$Left\doteq\int_M g(P^M(\vphi,X,Y), G^2Z)\vol+ \int_M g(P^M(\vphi,X,Z), G^2Y)\vol.$$
Substituting the explicit formula of $P^M$ in \defnref{defn:L_con} into the identity above and then integrating by parts, we get
\begin{align*}
I\doteq &\int_M g(P^M(\vphi,X,Y), G^2Z)\vol\\
	=&\frac{1}{2}\int_M g(\vphi,X)g(\delta Gv,\delta w)+g(\vphi, Y)g(\delta G u,\delta w) \vol\\
	&-\frac{1}{4}\int_M g(\i_{\delta Gv}\j_\vphi X+\i_{\delta Gu}\j_\vphi Y, \delta w)\vol\\
	&+\frac{1}{2}\int_Mg(\vphi,X)g(\delta v,\delta Gw)+g(\vphi, Y)g(\delta u,\delta Gw)\vol \\
	&-\frac{1}{4}\int_Mg(\i_{\delta v}\j_\vphi X+\i_{\delta u}\j_\vphi Y,\delta Gw)\vol \\
	&-\frac{2}{3}\int_M g(\vphi,X)g(v ,w)+g(\vphi,Y)g(u, w)\vol +\frac{1}{4}\int_Mg(\i_v\j_\vphi X+\i_u\j_\vphi Y, w)\vol \\
	&-\int_M g(\delta u,\delta Gv)g(\vphi ,Z)+g(\delta v,\delta Gu)g(\vphi, Z)\vol\\
	&+\frac{1}{2}\int_Mg(\i_\vphi \j_{\delta u}(\delta Gv)+\i_\vphi \j_{\delta v}(\delta Gu), Z)\vol\\
	&+\frac{2}{3}\int_M g(u,v)g(\vphi, Z)\vol -\frac{1}{4}\int_M  g(\i_\vphi \j_u v, Z)\vol.
\end{align*}
Using \lemref{lem:adjoint_i_operator} and \lemref{lem:ex_j} respectively, we have
$$\int_M g(\i_{\delta Gv}\j_\vphi X+\i_{\delta Gu}\j_\vphi Y, \delta w)\vol=\int_M [g(\i_{\delta w}\j_\vphi X,\delta Gv)+g(\i_{\delta w}\j_\vphi Y, \delta Gu)]\vol ,$$
$$\int_Mg(\i_\vphi \j_{\delta u}(\delta Gv)+\i_\vphi \j_{\delta v}(\delta Gu), Z)\vol=\frac{1}{2}\int_Mg(\i_{\delta u}\j_\vphi Z, \delta Gv)+g(\i_{\delta v} \j_{\vphi} Z, \delta Gu)\vol$$
and
$$\int_Mg(\i_v\j_\vphi X+\i_u\j_\vphi Y, w)\vol=\int_M g(\i_\vphi\j_v w,X)\vol+\int_M g(\i_\vphi\j_uw, Y)\vol.$$	
Putting these into the formula of $I$ and collecting the like terms, we have that
\begin{align*}
I	=&	\frac{1}{2}\int_M \left[ g(\vphi,X)g(\delta Gv,\delta w)+g(\vphi,X)g(\delta v,\delta Gw)+g(\vphi, Y)g(\delta G u,\delta w)\right. \\
	&+\left. g(\vphi, Y)g(\delta u,\delta Gw) -g(\delta u,\delta Gv)g(\vphi, Z)-g(\delta v,\delta Gu)g(\vphi, Z)\right] \vol \\
	&-\frac{1}{4}\int_M \left[ g(\i_{\delta w}\j_\vphi X,\delta Gv)+g(\i_{\delta v}\j_\vphi X,\delta Gw)+g(\i_{\delta w}\j_\vphi Y, \delta Gu) \right. \\
	&\left. +g(\i_{\delta u}\j_\vphi Y,\delta Gw)-g(\i_{\delta u}\j_\vphi Z, \delta Gv)-g(\i_{\delta v} \j_{\vphi} Z, \delta Gu) \right]\vol\\
	&-\frac{2}{3}\int_M g(\vphi,X)g(v ,w)+g(\vphi,Y)g(u, w)-g(\vphi, Z)g(u,v)\vol\\
	& +\frac{1}{4}\int_M g(\i_\vphi\j_v w, X)+g(\i_\vphi\j_u w,Y) - g(\i_\vphi \j_u v, Z)\vol.
	\end{align*}
Observe that each of the four integrating terms in the above formula can be divided into two parts, which are symmetric and skew-symmetric with respect to the pair $Y,\  Z$ respectively. Note that
$$II\doteq \int_M g(P^M(\vphi,X,Z), G^2Y)\vol $$
is exactly $I$ by replacing $Y, \ Z$. The symmetry leads to
 
\begin{align*}
Left=&	I+II=\int_M \left[ g(\vphi,X)g(\delta Gv,\delta w)+g(\vphi,X)g(\delta v,\delta Gw)\right] \vol \\
	&-\frac{1}{2}\int_M \left[ g(\i_{\delta w}\j_\vphi X,\delta Gv)+g(\i_{\delta v}\j_\vphi X,\delta Gw) \right]\vol\\
	&-\frac{4}{3}\int_M g(\vphi,X)g(v ,w)\vol +\frac{1}{2}\int_M g(\i_\vphi\j_v w, X)\vol.
\end{align*}
Finally, applying \lemref{lem:ex_j}, we rewrite it as
\begin{align*}
	Left=&\int_M \left[ g(\vphi,X)g(\delta Gv,\delta w)+g(\vphi,X)g(\delta v,\delta Gw)\right] \vol \\
	&-\int_M \left[ g(\j_\vphi X,\j_{\delta w}\delta Gv)+g(\j_\vphi X,\j_{\delta v}\delta Gw) \right]\vol\\
	&-\frac{4}{3}\int_M g(\vphi,X)g(v ,w)\vol +\frac{1}{2}\int_M g(\j_v w, \j_\vphi X)\vol.
\end{align*}
Comparing it with the right part in \eqref{eq:left1}, we complete the proof.
\end{proof}

%%%%%%%%%%%%%%%%%%%%%%%%%%%%%%%%%%%
\subsection{Proof of \thmref{prop:Var_Dirichlet energy} }\label{sec:pf2}
In this subsection, we prove \thmref{prop:Var_Dirichlet energy}.
We now compute the first order variations of the $\mathcal E^L, \ \mathcal E^D$ and $\mathcal E^M$ functionals.

\begin{proof}[Proof of the third part of \thmref{prop:Var_Dirichlet energy}]
	We start the proof with the calculation of the variation of $\mathcal{E}^L$ energy. By the definition of operator $\pi_d$, it holds
	$$\mathcal E^L(\vphi)\doteq\int_M |\pi_d\vphi|^2_\vphi\vol=\int_M g(\pi_d\vphi,\vphi)\vol.$$
	Then taking derivatives along the closed path $\vphi(t)$, we get
	\begin{align*}
		\frac{d}{dt}\mathcal E^L(\vphi) &= \frac{d}{dt}\int_M g(\pi_d\vphi,\vphi)\vol
		=\frac{d}{dt}\int_M \pi_d\vphi\wedge *\vphi \\
		&=\int_M (\pi_d\vphi)_t \wedge *\vphi+ \pi_d\vphi\wedge (*\vphi)_t \\
		&=\int_M (\pi_d)_t\vphi\wedge *\vphi+\pi_d\vphi_t \wedge *\vphi+ \pi_d\vphi\wedge (*\vphi)_t \\
		&\doteq I+II+III.
	\end{align*}
	
	We now deal with the three parts respectively. 
	Using the variation of the projection map $\pi_d$ given in \propref{prop:var_pi_d}, we have
	\begin{align*}
		I &=\int_M g\left( \pi_d**_t(\vphi-\pi_d\vphi) ,\vphi\right) \vol .
	\end{align*}
	Again using \corref{lem:var_*_general} and \lemref{lem:ex_j}, we have that
	\begin{align*}
		I&=\int_M g\left(\frac{4}{3}g(\vphi,\vphi_t)(\vphi-\pi_d \vphi)-\frac{1}{2}\i_{(\vphi-\pi_d \vphi)}\j_{\vphi}(\vphi_t),\pi_d\vphi\right) \vol \\
		&=\int_M g\left(\frac{4}{3}g(\vphi-\pi_d \vphi,\pi_d\vphi)\vphi-\frac{1}{2}\i_\vphi\j_{(\vphi-\pi_d \vphi)}(\pi_d \vphi),\vphi_t\right) \vol.
	\end{align*}
	
	The second part is
	\begin{align*}
		II =\int_M \pi_d\vphi_t \wedge *\vphi
		=\int_M g(\pi_d\vphi_t,\vphi) \vol
		=\int_M g(\vphi_t,\pi_d \vphi)\vol.
	\end{align*}
	
	Finally, by the variation of $\psi=*\vphi$ given by \corref{cor:variation dual form}, we have
	\begin{align*}
		III &=\int_M  \pi_d\vphi\wedge (*\vphi)_t
		=\int_M  g(\pi_d\vphi, *(*\vphi)_t)\vol \\
		&=\int_Mg_\vphi(\pi_d\vphi,\frac{4}{3}\pi_1^3\vphi_t+\pi_7^3\vphi_t-\pi_{27}^3\vphi_t) \vol\\
		&=\int_M  g((\frac{4}{3}\pi_1^3+\pi_7^3-\pi_{27}^3)\pi_d\vphi,\vphi_t)\vol.
	\end{align*}
	
	Taking summation of $I,II$ and $III$ and recalling the definition of $F^L$ given in \thmref{prop:Var_Dirichlet energy}, we get
	$$\frac{\p}{\p t}\mathcal E^L(\vphi(t))=I+II+III=\int_M g(F^L(\vphi), \vphi_t)\vol$$
	with
	$$F^L(\vphi)=\frac{4}{3}g(\vphi-\pi_d \vphi,\pi_d\vphi)\vphi-\frac{1}{2}\i_\vphi\j_{(\vphi-\pi_d \vphi)}(\pi_d \vphi)+\pi_d \vphi+(\frac{4}{3}\pi_1^3+\pi_7^3-\pi_{27}^3)\pi_d\vphi.$$
	Recalling \propref{prop:propjectionmap}, we have $\pi_1^3\pi_d\vphi=\frac{1}{7}g(\vphi,\pi_d\vphi)\vphi.$ 
	We then use \propref{prop:operator_i and_B} to get
	\begin{align*}
		\i_\vphi\j_\vphi\pi_d\vphi=&	\i_\vphi\j_\vphi[(\pi_1^3+\pi_7^3+\pi_{27}^3)\pi_d\vphi]
		=\frac{18}{7}g(\vphi,\pi_d\vphi)\vphi+4\pi_{27}^3\pi_d\vphi\\
		=&18\pi_1^3 \pi_d\vphi+4\pi_{27}^3\pi_d\vphi.
	\end{align*}
Substituting these to the formula of $F^L(\vphi)$, we get \eqref{PL} and
 complete the first part of the proof.
	\end{proof}
	\begin{proof}[Proof of the second part of \thmref{prop:Var_Dirichlet energy}]
		By \defnref{E energies}, the variation of this energy is
		\begin{align*}
			\frac{d}{dt}\mathcal E^D(\vphi(t))=&\frac{d}{dt} \int_M \tau\wedge *\tau =\int_M  \tau_t \wedge *\tau  +\tau  \wedge (*\tau )_t .
		\end{align*}
		Since $*_t\tau=(*\tau)_t-*\tau_t$, we have
		 $$\int_M \tau_t \wedge *\tau=\int_M  *\tau_t \wedge \tau =\int_M  (*\tau)_t  \wedge \tau-*_t \tau  \wedge \tau.$$ 
	Hence
\begin{align*}
			\frac{d}{dt}\mathcal E^D(\vphi(t)) =\int_M   2\tau \wedge  (*\tau )_t-*_t\tau  \wedge \tau .
		\end{align*}
		Now we calculate these two parts respectively. The first part follows by $*\tau=-d\psi$,
		\begin{align*}
	I=&	2\int_M   \tau  \wedge (*\tau )_t
			=2\int_M   *d\psi \wedge (d\psi)_t =2\int_M   *d\psi \wedge d \psi_t\\
			=&2\int_M   \Delta\vphi \wedge \psi_t =2\int_M g\left(\Delta\vphi, *\psi_t\right)  \vol.
		\end{align*}
		Recall that \corref{cor:variation dual form} says
		$ *\psi_t=\frac{4}{3}\pi_1^3(\vphi_t)+\pi_7^3(\vphi_t)-\pi_{27}^3(\vphi_t).$
So, we can rewrite the first part as
		\begin{align}\label{eq:Var_D_energy_I}
				I =&2\int_M g\left(\frac{4}{3}\pi_1^3\vphi_t +\pi_7^3\vphi_t-\pi_{27}^3\vphi_t, \Delta \vphi \right) \vol \notag \\
				=&2\int_M g\left(\frac{4}{3}\pi_1^3\Delta \vphi +\pi_7^3\Delta \vphi-\pi_{27}^3\Delta \vphi,  \vphi_t \right) \vol \notag \\
				=&2\int_M g\left(\frac{4}{3}\pi_1^3\Delta \vphi-\pi_{27}^3\Delta\vphi, \vphi_t \right)  \vol.
		\end{align}
		For the second part, we have
		\begin{align*}
			II=-\int_M *_t\tau \wedge \tau  =-\int_M g\left(**_t\tau , \tau \right) \vol. \end{align*}
		By the variation of Hodge star operator given in \corref{lem:var_*_general}, it holds that
	$**_t\tau =g(\vphi,\vphi_t)\tau -\frac{1}{2}\i_{\tau }\j_{\vphi}\vphi_t$.
	Hence we get
		\begin{align*} 
			II=& -\int_M g\left(g(\vphi,\vphi_t)\tau -\frac{1}{2}\i_{\tau }\j_{\vphi}\vphi_t, \tau \right) \vol \notag \\
			=&  -\int_M g(\tau ,\tau )g(\vphi_t,\vphi)\vol +\frac{1}{2}\int_Mg(\i_{\tau }\j_{\vphi}\vphi_t, \tau ) \vol.
		\end{align*}
		 Using \lemref{lem:ex_j}, we have that
		\begin{align*}
			\frac{1}{2}\int_Mg(\i_\tau\j_{\vphi}\vphi_t, \tau )  \vol
			=\int_M g(\vphi_t,\i_\vphi \j_\tau\tau )  \vol.
		\end{align*}
	Thus
	\begin{align}\label{eq:Var_D_energy_II}
		II	=&  \int_M g\left(\vphi_t,-g(\tau ,\tau )\vphi\right)\vol  +\int_M g(\vphi_t,\i_\vphi \j_\tau\tau )  \vol.
	\end{align}
		Adding \eqref{eq:Var_D_energy_I} to \eqref{eq:Var_D_energy_II}, we obtain that
		\begin{align*}
			\frac{d}{dt}\mathcal E^D(\vphi(t))=\int_Mg_{\vphi(t)}\left(\vphi_t,F^D(\vphi(t))\right) \vol,
		\end{align*} and
		\begin{align*}
			F^D(\vphi)&=\frac{8}{3}\pi_1^3\Delta \vphi-2\pi_{27}^3\Delta \vphi-g(\tau,\tau)\vphi+\i_{\vphi}\j_\tau \tau.
		\end{align*}
We see from \corref{cor:pro_tau} that  $\pi_1^3\Delta \vphi=\frac{1}{7}g(\tau,\tau)\vphi$ and  $\pi_7^3 \Delta\vphi=0$ for closed $\vphi$. We have
	$$ \pi_{27}^3\Delta \vphi=\Delta \vphi-\pi_1^3\Delta \vphi-\pi_7^3\Delta \vphi=\Delta \vphi-\frac{1}{7}g(\tau, \tau)\vphi.$$
	 Thus we can rewrite $F^D$ as
	\begin{align*}
	F^D(\vphi)&=\frac{8}{21}g(\tau,\tau)\vphi-2\left[\Delta \vphi-\frac{1}{7}g(\tau, \tau)\vphi\right]-g(\tau,\tau)\vphi+\i_{\vphi}\j_\tau \tau\\
			&=-\frac{1}{3}g(\tau,\tau)\vphi-2\Delta \vphi+\i_{\vphi}\j_\tau \tau.
		\end{align*}
		%Here we further explain $P(\vphi)$. In [MR3613456], it holds that
		%$$\Delta_\vphi=i(h)=$$
		Since $\frac{\p}{\p t}\vphi(t)$ is a d-exact form, we conclude that
		\begin{align*}
			\frac{d}{dt}\mathcal E^D&=\int_M g\left(\vphi_t,\pi_d F^D(\vphi) \right) \vol 
			={\mathcal G}_\vphi^D\left(\vphi_t,d\delta F^D(\vphi)\right) .
		\end{align*}
	\end{proof}

	\begin{proof}[Proof of the third part of \thmref{prop:Var_Dirichlet energy}]
	By definition, it holds that
		\begin{align*}
			\mathcal{E}^M(\vphi(t))=\int_M g_{\vphi}(d\tau, d\tau)\vol_g=\int_M d\tau\wedge *d\tau.
		\end{align*} 
Taking derivative with time $t$, we then have
		\begin{align*}
			\frac{d}{dt}\mathcal{E}^M(\vphi(t))
			=&2\int_M d\tau_t \wedge *d\tau +\int_M d\tau \wedge *_t d\tau\\
			=&2\int_M g(\tau_t,\delta d\tau ) \vol +\int_M g(d\tau,**_t d\tau)\vol \\
			\doteq & I+II.
		\end{align*}
Recall that \eqref{eq:mid1} in \corref{cor:var delta_vphi} says 
\begin{align*}
	\tau_t=&-g(\vphi,\vphi_t)\delta \vphi+\frac{1}{3}\delta [g(\vphi,\vphi_t)  \vphi] +\frac{1}{2}\i_{\delta \vphi}\j_{\vphi}\vphi_t-\delta \vphi_t+\frac{1}{2}\delta *[\vphi\wedge*(\vphi\wedge\vphi_t)].
\end{align*}
Substituting it into the first part and integrating by parts, we have that
\begin{align*}
I=&2\int_M g(\tau_t,\delta d\tau ) \vol\\
=&-2\int_M g(\vphi,\vphi_t)g(\tau,\delta d\tau) \vol+\frac{2}{3}\int_M g(\vphi,\vphi_t)g(\vphi,d\delta d\tau) \vol\\
&+\int_Mg(\i_{\tau}\j_\vphi\vphi_t,\delta d\tau) \vol-2\int_M g(\vphi_t,d\delta d\tau) \vol\\
&+\int_M g(\vphi\wedge*(\vphi\wedge\vphi_t),*d\delta d\tau) \vol.
\end{align*}
Using \lemref{lem:ex_j}, we get
	\begin{align*}
	\int_Mg(\i_{\tau}\j_\vphi\vphi_t,\delta d\tau) \vol=&2\int_Mg(\vphi_t,\i_\vphi\j_{\tau}\delta d\tau) \vol.
	\end{align*}
By definition of the Hodge star operator, we see that 
$$\int_M g(\a,\b)\vol=\int_M\a\wedge *\b=\int_M\b\wedge *\a. $$
So we have
\begin{align*}
	&\int_M g(\vphi\wedge*(\vphi\wedge\vphi_t),*d\delta d\tau) \vol=\int_M d\delta d\tau\wedge \vphi\wedge*(\vphi\wedge\vphi_t) \\
	=&\int_M *(d\delta d\tau\wedge \vphi) \wedge \vphi\wedge\vphi_t=\int_M g\left(*[*(d\delta d\tau\wedge \vphi) \wedge \vphi],\vphi_t \right)\vol.
\end{align*}
Now we use the identities above to get
\begin{equation}\label{eq:DM_PartI}
\begin{aligned}
I=&-2\int_M g(g(\tau,\delta d\tau)\vphi,\vphi_t) \vol+\frac{2}{3}\int_M g(g(\vphi,d\delta d\tau) \vphi,\vphi_t)\vol\\
	&+2\int_Mg(\vphi_t,\i_\vphi\j_{\tau}\delta d\tau) \vol-2\int_M g( \vphi_t,d\delta d\tau) \vol\\
	&+\int_M g\left(*[*(d\delta d\tau\wedge \vphi) \wedge \vphi],\vphi_t \right)\vol.
\end{aligned}
\end{equation}

Using the variation formula of $*_t$ from \corref{lem:var_*_general}, the second part is
		\begin{align*}
		II=&	\int_M  g(d\tau,**_t d\tau)\vol
			=\int_M g\left(d\tau, \frac{4}{3}g(\vphi,\vphi_t) d\tau-\frac{1}{2}\i_{d\tau}\j_{\vphi}\vphi_t \right)\vol\\
				=&\frac{4}{3}\int_M g(d\tau, d\tau)g(\vphi_t,\vphi)\vol-\frac{1}{2}\int_M g\left(d\tau,\i_{d\tau}\j_{\vphi}\vphi_t\right)\vol.
		\end{align*}
	We apply \lemref{lem:ex_j} to see that
	\begin{align*}
		-\frac{1}{2}\int_M g(\i_{d\tau}\j_{\vphi}\vphi_t,d\tau)\vol=-\frac{1}{2}\int_M g(\i_\vphi\j_{d\tau}(d\tau),\vphi_t)\vol.
		\end{align*}
	Hence we have
		\begin{align}\label{eq:DM_PartII}
		II=&\frac{4}{3}\int_M g(\vphi_t,g(d\tau, d\tau)\vphi)\vol-\frac{1}{2}\int_M g(\i_\vphi\j_{d\tau}(d\tau),\vphi_t)\vol.
	\end{align}

Adding \eqref{eq:DM_PartI}, \eqref{eq:DM_PartII} together, we see  the first order derivative of the $\mathcal{E}^M$ energy is exactly
$$\int_M g\left(\vphi_t, F^M(\vphi(t))\right)\vol_{\vphi(t)}.$$
	\end{proof}

%%%%%%%%%%%%%%%%%%%%%%%%%%%%%%%%%%%%%
\subsection{The Dirichlet energy and the Dirichlet flow }\label{sec:5.3}
We list some properties for the Dirichlet energy $\mathcal E^D$. Recall that the Einstein-Hilbert functional defined in the space of Riemannian metrics is
	\begin{align*}
		\mathcal E(g)=\int_M  Scal(g) \vol.
	\end{align*}
	
	\begin{prop}
		The Dirichlet energy in $\mathcal M$ is in fact the negative total scalar curvature, that is
		\begin{align*}
			\mathcal E^D(\vphi)=-2\mathcal E(g_\vphi).
		\end{align*}
	\end{prop}
	\begin{proof}
		Since $\delta \vphi=\tau,$ we have
		\begin{align}\label{eq:energy_D}
			\mathcal E^D(\vphi) =\int_M|\tau|^2_\vphi\vol=-2\int_M  Scal(g_\vphi) \vol.
		\end{align}
		The last equality use the fact that the scalar curvature of a closed $G_2$-structure is given by $Scal(g_\vphi)=-\frac{1}{2}|\tau|^2_{\vphi}.$
	\end{proof}
	
	Therefore, the Dirichlet energy $\mathcal E^D(\vphi)$ could be viewed both as a functional of the 3-form $\vphi$, and also a functional of the Riemannian metric $g_\vphi$. It is known that the Einstein-Hilbert functional is invariant under diffeomorphisms and the critical metric of $\mathcal E(g_\vphi)$ is Ricci flat. But it is not scaling invariant, unless we modify it by some power of the volume.
	
	We recall the classical result on the gradient flow of the Einstein-Hilbert functional, c.f. \cite[Page 112, Chapter 2, Section 6]{MR2274812}.
	\begin{lem}\label{Var EH functional}
		Let $g(t)$ be a family of Riemannian metrics on a compact manifold $M$. The first order derivative of the Einstein-Hilbert functional is
		\begin{align*}
			\frac{d}{dt}\mathcal{E}(g)=\int_M g(\frac{\p g}{\p t},-Ric(g)+\frac{1}{2}Scal(g)\cdot g) \vol_{g}.
		\end{align*}
	\end{lem}

%	\begin{lem}\label{ED for T}
%		\begin{align}\label{ED for T equation}
%			\mathcal E^D(\vphi)=\frac{1}{2}\int_M \tau\wedge *\tau .
%		\end{align}
%	\end{lem}
%	\begin{proof}
%		Since $\vphi(t)$ are closed forms, their torsion tensor $\tau(t)\in\Omega^2_{14}$ satisfy \begin{align*}d\psi=\tau\wedge \vphi=-*\tau,
%		\end{align*}
%		(c.f. Section \ref{Torsion and curvatures}).
%		Now we have
%		\begin{align*}
%			\delta_\vphi\vphi  =-*d*\vphi=-*d\psi=\tau,
%		\end{align*}
%		and
%		$$\Delta_\vphi \vphi=d\delta \vphi=-2dT=-(\Na_i T_{jk})dx^{ijk}.$$
%		Furthermore, by (2.21) in \cite{MR3613456},
%		\begin{align}\label{eq:Laplacian_vphi}
%			\Delta_\vphi \vphi=&\i_\vphi((\Na_mT_{in}\vphi_j^{mn}-\frac{1}{3}|T|^2g_{ij}+T_{im}T_{jn}g^{mn})dx^i\otimes dx^j )\\
%			=& \frac{1}{2}\left(\Na_m T_{in}\vphi^{rmn}\vphi_{rab}-\frac{1}{3}|T|^2\vphi_{iab}-T_{im}T^{lm}\vphi_{lab}\right) dx^{iab},\\
%			=&\frac{1}{2}\left( \Na_a T_{ib}-\Na_b T_{ia}+\Na_m T_{in}\psi_{ab}^{mn}-\frac{1}{3}|T|^2\vphi_{iab}-T_{im}T^{lm}\vphi_{lab}\right)dx^{iab}.
%		\end{align}
%		and
%		$$\pi_1^3^\vphi(\Delta_\vphi \vphi)=\frac{4}{7}|T|^2_{\vphi}\vphi.$$
%		Now we write
%		\begin{align*}
%			\mathcal E^D(\vphi(t))&= \frac{1}{2}\int_M g_{\vphi}(\delta_\vphi\vphi, \delta_\vphi\vphi)  \vol =2\int_M g_{\vphi}(T,T ) \vol.
%		\end{align*}
%	\end{proof}
%	

We now study the second order derivative of the Dirichlet energy.
	\begin{prop}
		The Euler-Lagrange equation of the Dirichlet energy is
		\begin{align}\label{critical ED}
			F_d^D(\vphi)=\pi_d [F^D(\vphi)]=0.
		\end{align}
		Precisely, using \eqref{PD}, the critical equation \eqref{critical ED} reads
		\begin{align}\label{P metric}
			\Delta \vphi&=\pi_d \left[-\frac{1}{6}|\tau|^2_{\vphi}\vphi +\frac{1}{2}\i_\vphi\j_\tau \tau\right] .
		\end{align}
	\end{prop}
	It is clear that all torsion free $G_2$-structures satisfy the critical equation. This is because torsion free condition says $\tau =\delta \vphi=0$ and $\Delta \vphi=0$ , which further implies $F^D(\vphi)=0$.

	\begin{prop}\label{Dirichlet energy convex}
		The Dirichlet energy functional is concave at torsion free $G_2$-structures.
		% (or at the critical metric \eqref{critical ED}???)
	\end{prop}
	\begin{proof}
		Taking $t$-derivative on \eqref{eq:energy_D} twice, we have
		$$\frac{d^2}{dt^2}{\mathcal E}^D(\vphi)=\int_M (\tau\wedge *_t\tau_t+2\tau_t\wedge *\tau_t+2\tau_{tt}\wedge *\tau+\tau\wedge *_{tt} \tau).$$
		At torsion-free point where $\tau=0$, we have
		$$\frac{d^2}{dt^2}{\mathcal E}^D(\vphi)=2\int_M \tau_t\wedge *\tau_t,$$
		which is positive definite.
	\end{proof}

At last, we show that in order to deal with the projection operator $\pi_d^{\vphi(t)}$ in the Dirichlet flow \eqref{eq:Drichletflow}
\begin{align*}
	\frac{\p}{\p t}\vphi(t)=-\pi_d F^D(\vphi(t))	,\quad F^D(\vphi)=-2\Delta_\vphi \vphi-\frac{1}{3}|\tau|^2_{\vphi}\vphi +\i_\vphi \j_\tau\tau.
\end{align*}
 we could apply a method from the study of the pseudo-Calabi flow \cite{MR3010550}.
We rewrite the Dirichlet flow equation as
\begin{equation} \label{eq:ED flow2}
	\frac{\p}{\p t}\vphi(t) =d\b(t),\quad
	\Delta_{\vphi(t)} \b(t) = -\delta_{\vphi(t)} F^D(\vphi(t)).
\end{equation}
In the above flow equation, we need to normalise $\beta$ to be a $\delta_{\vphi(t)}$-exact form. So, we abstract $\beta$ by the harmonic form $H(t)$ satisfying
\begin{align*}
	\Delta_{\vphi(t)} H(t) = 0.
\end{align*}
According to the Hodge theorem, the harmonic form $H(t)$ w.r.t the metric $g_{\vphi(t)}$ is unique in a given cohomology class $[\beta(t)]$.
Consequently, we have
\begin{equation} \label{eq:ED flow3}
	\frac{\p}{\p t}\vphi(t) =d\delta_{\vphi(t)}\gamma(t),\quad
	\Delta_{\vphi(t)} \delta_{\vphi(t)}\gamma(t) = -\delta_{\vphi(t)} F^D(t).
\end{equation}

\begin{prop}
	The two equations \eqref{eq:ED flow2}, \eqref{eq:ED flow3} are equivalent.
\end{prop}
\begin{proof}
	It's sufficient to show that if $\b(t)=\delta_{\vphi(t)}\gamma(t)$ satisfies the second equation in \eqref{eq:ED flow3}, then it also satisfies
	$d\beta=-\pi_d^{\vphi(t)} P(t).$ Note that 
	\begin{align*}
		- \delta_{\vphi(t)} F(t)=\Delta_{\vphi(t)} \b(t) =d\delta_{\vphi(t)}\b(t)+\delta_{\vphi(t)}d \b(t)=\delta_{\vphi(t)}d \b(t).
	\end{align*}
	Furthermore, by applying \lemref{Hodge2} with this identity, we get
	\begin{align*}
		-\pi_d^{\vphi(t)}F(t)=-Gd\delta_{\vphi(t)}F(t)=G_{\vphi(t)}d\delta_{\vphi(t)}d\b(t).
	\end{align*}
	Note that the right-hand side of the above formula is $$G_{\vphi(t)}d\delta_{\vphi(t)}d\b(t)=G_{\vphi(t)}\Delta_{\vphi(t)} d\b(t)=d\b(t),$$
	thus $-\pi_d^{\vphi(t)}F^D(t)=d\b(t)$, which completes the proof.
\end{proof}

%%%%%%%%%%%%%%%%%%%%%%
%%%%%%%%%%%%%%%%%%%%%%%%%%%%%%%%%%%%%%%%%%%%%%

\subsection{Ebin metric and Laplacian metric}\label{sec:5.5}
We denote by $\mathcal P$ the space of all $G_2$ structures on $M$ (which is not necessary to be restricted in the given cohomology class $c$). Let $\mathfrak{G}$ be the space of all Riemannian metrics on $M$.

Equations \eqref{eq:G2metric}, \eqref{eq:G2metric1} defined a map between $G_2$-structures and metrics by
\begin{align*}
	\mathcal F : \mathcal{P}\rightarrow \mathfrak{G}, \quad
	\vphi\mapsto g_{\vphi}.
\end{align*} 
Hereafter we still denote $g_\vphi$ and the induced volume form $\vol_\vphi$ by $g$ and $\vol$.

The tangent space $T_g\mathfrak{G}$ consists of all symmetric $2$-tensors on $M$.
Ebin metric \cite{MR0267604} on $\mathfrak{G}$ is defined to be
\begin{align}\label{eq:defebin}
	 \mathcal G^E_g(A,B)=\int_M g(A,B) \vol_g,\quad A,B\in T_g\mathfrak{G},
	 \end{align}

By similar argument with \secref{sec:3.2}, we see that the tangent space of $\mathcal{P}$ at any point $\vphi$ is simply $\Omega^3(M)$, the space of all $3$-forms on $M$. The Laplacian metric in Definition \ref{defn:L2} could be generalised to $T_\vphi \mathcal P$.
 \begin{defn} [Generalised $\mathcal G^L$ metric]
  We define the generalised Laplacian metric 
\begin{align*}
\bar{\mathcal G}^L_{\vphi}(X,Y)=\int_M g(X,Y)\vol, \quad \forall X,Y\in T_\vphi \mathcal P.
 \end{align*}
 \end{defn}

\begin{thm}\label{EbinLaplacian}
The pull back of the Ebin metric reads
	\begin{align*}
		\bar{\mathcal G}^E_\vphi(X,Y)=&\frac{3}{4}\bar{\mathcal G}^L_{\vphi}(X,Y)-\int_M \left[\frac{11}{18}g(X,\vphi)g(Y,\vphi)\right.\\
		&-\left. \frac{1}{16}\left(X_a^{\ jk}Y^{abc}\psi_{jkbc}
		+\vphi_{ajk}X^{ijk} \vphi_{ibc}Y^{abc}\right) \right] \vol 
	\end{align*}
for all $X,Y \in T_{\vphi} \mathcal{P}$.
\end{thm}
\begin{proof}

Recall the $3$-form decomposition \eqref{eq:decomposition_3form}.
\begin{align*}
X =\pi_1^3X+\pi_7^3X+\pi_{27}^3X=3f_0 \vphi+* (f_1\wedge \vphi)+f_3.
 \end{align*}
 where $f_0\in C^{\infty}(M), f_1\in\Om^1, f_3\in\Om^3_{27}$.
The terms $f_0$ and $f_1$ are given in \eqref{eq:f_0,f_1} as
	\begin{align*}
		f_0=\frac{1}{21}g( X ,\vphi),\quad f_1=\frac{1}{4}\ast(X\wedge\vphi).
	\end{align*}

At any $\vphi \in \mathcal P$, according to \lemref{lem:var_g} of \secref{sec:3.1}, $\mathcal F$ induces the linear map $\mathcal F_*:T_\vphi \mathcal P\rightarrow T_{g}\mathfrak{G}$  by
	\begin{align}\label{eq:pullback}
	\mathcal F_*(X) =\frac{1}{2}\j_\vphi f_3 +2 f_0 g,\ \ \forall X\in T_\vphi \mathcal P.
\end{align}
By (1) and (3) of \propref{prop:operator_i and_B}, we have $\j_\vphi\pi_1^3X=3f_0\j_\vphi\vphi=18f_0 g$ and $\j_\vphi(\pi_7^3X)=0$.
So, using the form decomposition, we further get
\begin{align} \label{eq:pullback2}
	\mathcal F_*(X) =&\frac{1}{2}\j_\vphi(X-\pi_1^3X-\pi_7^3X) +2 f_0 g
	=\frac{1}{2}\j_\vphi X-\frac{1}{3}g(X,\vphi)g.
\end{align}

We pull the Ebin metric back on $\mathcal{M}$, which reads
\begin{align}\label{eq:pullback3}
	\bar{\mathcal G}^E_\vphi(X,Y)=\mathcal G^E_{g}(\mathcal F_*(X),\mathcal F_*(Y))=\int_M g(\mathcal F_*(X),\mathcal F_*(Y))\vol_g ,
	\end{align}for all $X,Y \in T_{\vphi} \mathcal{P}$.

Then we calculate $\bar{\mathcal G}$ precisely.
	Note that $g(g, g)=\frac{7}{2}$ and $g(g,\j_\vphi X)=3g(X,\vphi)$. Thus we have from \eqref{eq:pullback2} that
	\begin{align*}
	g(\mathcal F_*(X),\mathcal F_*(Y))=&\frac{1}{4}g(\j_\vphi X,\j_\vphi Y)-\frac{1}{6}g(X,\vphi)g(\j_\vphi Y,g)-\frac{1}{6}g(Y,\vphi)g(\j_\vphi X,g)\\
	&+\frac{1}{9}g(X,\vphi)g(Y,\vphi)g(g,g)\\
	=&\frac{1}{4}g(\j_\vphi X,\j_\vphi Y)-\frac{11}{18}g(X,\vphi)g(Y,\vphi).
	\end{align*}
Recall \lemref{lem:norm_j}, which says
		$$g(\j_\vphi X, \j_\vphi Y)=3g(X,Y)+ \frac{1}{4}X_a^{\ jk}Y^{abc}\psi_{jkbc}
	+\frac{1}{4}\vphi_{ajk}X^{ijk} \vphi_{ibc}Y^{abc}.$$
	We conclude from above that
	\begin{align*}
g(\mathcal F_*(X),\mathcal F_*(X))=&\frac{3}{4}g(X,X)-\frac{11}{18}g^2(X,\vphi)+ \frac{1}{16}X_a^{\ jk}X^{abc}\psi_{jkbc}\\
&	+\frac{1}{16}\vphi_{ajk}X^{ijk} \vphi_{ibc}X^{abc}.
		\end{align*}
Inserting it into \eqref{eq:pullback3}, we obtain
	the formula for the pull back Ebin metric.
\end{proof}

\begin{prop}\label{EbinLaplacian zero}
	$\bar{\mathcal G}^E_\vphi(X,X)\geq 0$ and equality holds if and only if $X_\vphi \in (\Omega_\vphi)_7^3$.
\end{prop}
\begin{proof}
	This can be seen by the definition that
	$$\bar{\mathcal G}^E_\vphi(X,X)=\mathcal G^E_{g_{\vphi}}(\mathcal F_*(X),\mathcal F_*(X))\geq 0,$$
	where it vanishes if and only if $\mathcal F_*(X)=0$, which is equivalent to $\j_\vphi f_3+2f_0 g=0$ 
	 by \eqref{eq:pullback}. Recalling (4) of \propref{prop:operator_i and_B}, $\j_\vphi f_3 $ lies in $ S_0^2T^*M$, the space of trace-free symmetric $2$-tensors. This clearly leads to $$\j_\vphi f_3=0,\  2f_0=0.$$ Furthermore, by (6) of  \propref{prop:operator_i and_B}, we see 
	 $f_3=\frac{1}{4}\i_\vphi\j_\vphi f_3=0$. Hence $\pi_1^3 X=\pi_{27}^3 X=0$ and $ X\in (\Omega_\vphi)_7^3$.
\end{proof}
  %%%%%%%%%%%%%%%%%%%%%%%%%%%%%%%%%%%%%%%%%%%%%%%%%%%%%%%%%%%%%%%%%%%%%%%%%%%%%%%%%%%%%%%%%%%%%%%%%%%%%%%%%%%%%%%%%%%%%%
  Note that each $G_2$-structure $\vphi$ determines a unique Riemannian metric $g_\vphi$. But different $G_2$-structures could define the same metric $g_\vphi$. We denote this set by $\mathcal P_\vphi$. Bryant \cite[Equation (3.6)]{BR05} showed that such $G_2$-structure $\bar\vphi$ obeys a simple formula
  \begin{align*}
\mathcal P_\vphi=\{\bar\vphi=(f^2-|\eta|_\vphi^2)\vphi+2f\ast(\eta\wedge\vphi)+2\i_\vphi(\eta\otimes\eta)\},\quad f^2+|\eta|_\vphi^2=1
  \end{align*}where $f$ is a function and $\eta$ is a $1$-form.
  We see from the proof of Proposition \ref{EbinLaplacian zero} that, if $X\in T_\vphi \mathcal P_\vphi$, then $\mathcal F_*(X)=0$ and $X\in  (\Omega_\vphi)_7^3$.
  \begin{cor}
  If $X\in T_\vphi \mathcal P_\vphi$, then $\bar{\mathcal G}^E_\vphi(X,X)=0$.
  \end{cor}

\end{document}